\documentclass[reqno,10pt,centertags,draft]{amsart}
\usepackage{amsmath,amsthm,amscd,amssymb,latexsym,upref,stmaryrd}
\usepackage{amsfonts,mathrsfs}

\newcommand{\R}{{\bbR}}
\newcommand{\N}{{\mathbb N}}

\newcommand{\C}{{\mathbb C}}

\newcommand{\bbC}{{\mathbb{C}}}

\newcommand{\bbN}{{\mathbb{N}}}

\newcommand{\bbR}{{\mathbb{R}}}

\newcommand{\cB}{{\mathcal B}}
\newcommand{\cC}{{\mathcal C}}

\newcommand{\cH}{{\mathcal H}}

\newcommand{\cK}{{\mathcal K}}

\newcommand{\cP}{{\mathcal P}}

\newcommand{\cS}{{\mathcal S}}

\DeclareMathOperator{\supp}{supp}

\DeclareMathOperator{\dom}{dom}

\DeclareMathOperator{\tr}{tr}

\DeclareMathOperator*{\nlim}{n-lim}
\DeclareMathOperator*{\slim}{s-lim}

\renewcommand{\Im}{\text{\rm Im}}
\renewcommand{\ln}{\text{\rm ln}}

\newcommand{\beq}{\begin{equation}}
\newcommand{\enq}{\end{equation}}

\newcommand{\p}{\prime}
\newcommand{\sgn}{{\textrm{sgn}}}

\newcommand{\no}{\notag}
\newcommand{\lb}{\label}
\newcommand{\f}{\frac}

\newcommand{\ol}{\overline}

\newcommand{\wti}{\widetilde}

\newcommand{\hatt}{\widehat}
\newcommand{\bi}{\bibitem}

\let\geq\geqslant
\let\leq\leqslant

\makeatletter
\def\theequation{\@arabic\c@equation}

\allowdisplaybreaks 
\numberwithin{equation}{section}

\newtheorem{theorem}{Theorem}[section]

\newtheorem{lemma}[theorem]{Lemma}
\newtheorem{corollary}[theorem]{Corollary}
\newtheorem{hypothesis}[theorem]{Hypothesis}

\theoremstyle{remark}
\newtheorem{remark}[theorem]{Remark}
\newtheorem{definition}[theorem]{Definition}

\begin{document}

\title[Weak Convergence of Spectral Shift Functions]{Weak Convergence of Spectral Shift Functions for One-Dimensional Schr\"odinger Operators}

\author[F.\ Gesztesy and R.\ Nichols]{Fritz Gesztesy and Roger Nichols}
\address{Department of Mathematics,
University of Missouri, Columbia, MO 65211, USA}
\email{gesztesyf@missouri.edu}
\urladdr{http://www.math.missouri.edu/personnel/faculty/gesztesyf.html}
\address{Department of Mathematics,
University of Missouri, Columbia, MO 65211, USA}
\email{nicholsrog@missouri.edu}

\dedicatory{Dedicated with great pleasure to Eduard Tsekanovskii on the occasion of his 75th birthday}
\thanks{Submitted to {\it Math. Nachrichten.}}
\date{\today}

\begin{abstract} 
We study the manner in which spectral shift functions associated with 
self-adjoint one-dimensional Schr\"odinger operators on the finite interval $(0,R)$ 
converge in the infinite volume limit $R\to\infty$ to the half-line spectral shift function.  

Relying on a Fredholm determinant approach combined with certain measure theoretic facts, we show that prior vague convergence results in the literature in the special 
case of Dirichlet boundary conditions extend to the notion of weak convergence and arbitrary separated self-adjoint boundary conditions at $x=0$ and $x=R$. 
\end{abstract}

\maketitle

\section{Introduction}  \lb{s1}

In a nutshell, we are interested in the manner in which spectral shift functions 
converge in an infinite volume limit. More precisely, in this paper we give an 
exhaustive treatment of the one-dimensional case in which the finite interval $(0,R)$ 
converges to the half-line $[0,\infty)$ (the case of the entire real line being completely analogous). We explicitly consider the case of all separated self-adjoint boundary 
conditions at the endpoints $0$ and $R$. The multi-dimensional case, based on an 
abstract approach to this circle of ideas, will appear elsewhere \cite{GN11}. 

Before we focus on the abstract situation discussed in this paper, it is 
appropriate to briefly survey the known results in this area. Consider self-adjoint 
Schr\"odinger operators $H_j$ and $H_{0,j}$ 
in $L^2((-j,j)^n; d^n x)$, $n\in\bbN$, $n\geq 2$, generated 
by the differential expression $-\Delta + V$ and $-\Delta$ on $(-j,j)^n$, respectively, 
with Dirichlet boundary conditions on $\partial (-j,j)^n$, where 
$0 \leq V \in L^\infty(\bbR^n; d^nx)$ is of fixed compact support in $(-j,j)^n$, 
real-valued, and nonzero a.e.\ Denoting by $\xi\big(\lambda; H_j,H_{0,j}\big)$ for 
a.e.\ $\lambda \in \R$, the spectral shift function associated with the pair 
$(H_j, H_{0,j})$ (cf.\ \cite[Ch.\ 8]{Ya92}), normalized to be zero in a neighborhood of 
$-\infty$, Kirsch \cite{Ki87} showed in 1987 that (perhaps, somewhat unexpectedly) for any $\lambda > 0$, 
\begin{equation}
\sup_{j\in\bbN} \big|\xi\big(\lambda; H_j,H_{0,j}\big)\big| = \infty.    \lb{1.1}
\end{equation}
Moreover, denoting by $H$ and $H_0$ the corresponding self-adjoint Schr\"odinger operators in $L^2(\bbR^n; d^n x)$ generated by the differential expression 
$-\Delta + V$ and $-\Delta$ on $\R^n$, respectively, one cannot expect pointwise 
a.e.\ convergence (or convergence in measure) 
of $\xi\big(\cdot; H_j,H_{0,j}\big)$ to $\xi\big(\lambda; H,H_0\big)$ in the infinite volume limit $j\to\infty$ by the following elementary argument: For a.e.\ $\lambda > 0$, 
$\xi\big(\lambda; H,H_0\big)$ is a continuous function with respect to $\lambda$, 
related to the determinant of the underlying fixed energy scattering matrix. Yet 
$\xi\big(\cdot; H_j,H_{0,j}\big)$, as a difference of eigenvalue counting functions corresponding to the number of eigenvalues (counting multiplicity) of $H_j$ and 
$H_{0,j}$, respectively, is integer-valued and hence cannot possibly converge to a non-constant continuous function as $j\to\infty$. In particular, this argument applies to the 
one-dimensional context (in which case $\xi\big(\lambda; H,H_0\big) \to 0$ as 
$\lambda\to\infty$).

Having ruled out pointwise a.e.\ convergence of spectral shift functions in the infinite 
volume limit $j\to\infty$ in all space dimensions, it becomes clear that one has to invoke the concept of certain  generalized limits. Indeed, in 1995, Geisler, Kostrykin, and 
Schrader \cite{GKS95} proved for potentials $V \in \ell^1(L^2(\bbR^3; d^3x))$ (a 
Birman--Solomyak space, cf., e.g., \cite[Ch.\ 4]{Si05}) that for all $\lambda \in \R$,
\begin{equation}
\lim_{j\to\infty} \int_{(-\infty,\lambda]} \xi\big(\lambda'; H_j,H_{0,j}\big) d\lambda'
= \int_{(-\infty,\lambda]} \xi\big(\lambda'; H,H_0\big) d\lambda'.    \lb{1.2}
\end{equation}
Since $H_j$ and $H_{0,j}$ are bounded from below uniformly with respect to 
$j\in\bbN$, the limiting relation \eqref{1.2} involving distribution functions of the spectral 
shift measures is equivalent to vague convergence of the latter as observed in 
\cite[Prop.\ 4.3]{HLMW01}.  

In the one-dimensional half-line context, Borovyk and Makarov \cite{BM09} (see also 
Borovyk \cite{Bo08}) proved in 2009 that for potentials $V \in L^1((0,\infty);(1+ |x|)dx)$ 
real-valued, and denoting by $H_R$ the self-adjoint Schr\"odinger operator in 
$L^2((0,R);dx)$ and $H$ the corresponding self-adjoint Schr\"odinger operator in 
$L^2((0,\infty);dx)$, both with Dirichlet boundary conditions (and otherwise maximally defined or defined in terms of quadratic forms), and analogously for $H_{0,R}$ and 
$H_0$ in the unperturbed case $V=0$, the following vague limit holds: 
\begin{align}
\begin{split}
& \, \text{For any $g \in C_0(\bbR)$,} \\ 
& \lim_{R\to\infty} \int_{\R} \xi\big(\lambda; H_R,H_{0,R}\big) d\lambda \, g(\lambda)
= \int_{\R} \xi\big(\lambda; H,H_0\big) d\lambda \, g(\lambda).   \lb{1.3}
\end{split} 
\end{align}
In addition, they proved that the following Ces{\` a}ro limit, 
\begin{equation}
\lim_{R\to\infty} \f{1}{R} \int_0^R \xi\big(\lambda; H_r,H_{0,r}\big) dr = 
\xi\big(\lambda; H,H_0\big), \quad 
\lambda \in \bbR\backslash\big(\sigma_{\rm d} (H) \cup \{0\}\big)    \lb{1.4}
\end{equation}
exists (and the limit in \eqref{1.4} extends to $\lambda = 0$ if $H$ has no zero energy resonance). 

Returning to the case of multi-dimensional boxes $[-R,R]^n$, Hislop and 
M\"uller \cite{HM10} (see also \cite{HM10a}) proved a result going somewhat beyond 
vague convergence in 2010. More precisely, assuming a real-valued background 
potential $V^{(0)}$ satisfying $V_-^{(0)} \in K(\R^n)$, $V_+^{(0)} \in K_{\rm loc}(\R^n)$ 
and a potential $0 \leq V \in K_{\rm loc}(\R^n)$, $\supp (V)$\,compact (cf.\ \cite{AS82}, 
\cite{Si82} for the definition of (local) Kato classes), they show that 
\begin{align}
& \, \text{For any $f \in C_0(\bbR)$, and for any $f = \chi_{J}$, $J\subset\R$ a finite interval,}    \no \\ 
& \lim_{R\to\infty} \int_{\R} 
\xi\big(\lambda; H_{0,R} + V^{(0)} +V,H_{0,R}+ V^{(0)}\big) d\lambda \, f(\lambda)  
\lb{1.5} \\
& \quad = 
\int_{\R} \xi\big(\lambda; H_0 + V^{(0)} +V,H_0+ V^{(0)}\big) d\lambda \, f(\lambda).   \no
\end{align}
In addition, they derived a weaker version than the Ces\'aro limit in (1.4) in the multi-dimensional context. More precisely, they proved that for every sequence of lengths 
$\{L_j\}_{j\in\N} \subset (0,\infty)$ with $\lim_{j\to\infty} L_j = \infty$, there exists a subsequence $\{j_k\}_{k\in\N} \subset \N$ with
$\lim_{k\to\infty} j_k = \infty$, such that for every subsequence 
$\{k_{\ell}\}_{\ell\in\N} \subset \N$ with $\lim_{\ell\to\infty} k_{\ell} = \infty$,
\begin{align}
\begin{split} 
& \lim_{L\to\infty} \f{1}{L} \sum_{\ell=1}^L 
\xi\big(\lambda; H_{L_{j_{k_\ell}}}^{(0)} + V^{(0)} +V, 
H_{L_{j_{k_\ell}}}^{(0)} + V^{(0)}\big)    \lb{1.6} \\
& \quad \leq \xi\big(\lambda; H_0 + V^{(0)} +V,H_0 + V^{(0)}\big)
\end{split} 
\end{align}
for (Lebesgue) a.e.\ $\lambda \in \R$.

Before describing our results we should mention that spectral shift functions feature prominently in the context of eigenvalue counting functions and hence in the 
context of the integrated density of states. We refer, for instance, to \cite{CHKN02}, 
\cite{CHN01}, \cite{HK02}, \cite{HKNSV06}, \cite{HS02}, \cite{Ko99}, \cite{KS99}, 
\cite{KS00}, \cite{Na01}, and the references cited therein. For bounds on the 
spectral shift function we refer to \cite{CHN01}, \cite{HKNSV06}, \cite{HS02}, and 
\cite{So93}. 

In Section \ref{s3} we introduce basic facts on one-dimensional Schr\"odinger operators,  
describe associated Green's functions, Krein-type resolvent formulas, compute 
the integral kernel for the square root of the resolvent of the Dirichlet Laplacian on a finite interval and on a half-line. In addition, we derive Jost--Pais-type 
\cite{JP51} (cf.\ also \cite{GM03} and the extensive literature therein) reductions of Fredholm determinants corresponding to Birman--Schwinger-type kernels to Wronski determinants of appropriate solutions of the associated Schr\"odinger equation. Basic convergence results of resolvents and closely related operators (including 
Birman--Schwinger-type kernels) are discussed in Section \ref{s4}. Finally, our 
principal Section \ref{s5} provides a new approach to vague, and especially, weak convergence of spectral shift functions in the infinite volume limit based on the 
convergence of underlying Fredholm determinants combined with certain measure theoretic facts. This then considerably extends \eqref{1.3} 
in a variety of ways: First, $g(\cdot)$ in \eqref{1.3} can now be replaced by 
$f(\cdot)/(1 + \lambda^2)$ with $f$ a bounded continuous function (this can further be improved). Second, we can permit any separated self-adjoint boundary conditions (not just 
Dirichlet boundary conditions) at $x=0$ and $x=R$. Third, we can permit real-valued (singular) potentials $V$ satisfying $V \in L^1((0,\infty); dx)$ (in the case of Dirichlet boundary conditions at $x=0$ one can even permit $V \in L^1((0,\infty); [x/(1+x)]dx)$). Appendix \ref{sA} collects basic properties of spectral shift functions used in 
Section \ref{s5}, and Appendix \ref{sB} derives a particular decomposition formula relating spectral shift functions for the interval $(0,R_1)$ and $(0,R_2)$, $0<R_1<R_2$, associated with Dirichlet boundary conditions at $x=0,R_1,R_2$. 
 
Finally, we briefly summarize some of the notation used in this paper: Let $\cH$ be a
separable complex Hilbert space, $(\cdot,\cdot)_{\cH}$ the scalar product in $\cH$
(linear in the second argument), and $I_{\cH}$ the identity operator in $\cH$.
Next, let $T$ be a linear operator mapping (a subspace of) a
Hilbert space into another, with $\dom(T)$ and $\ker(T)$ denoting the
domain and kernel (i.e., null space) of $T$. 
The resolvent set 
of a closed linear operator in $\cH$ will be denoted by $\rho(\cdot)$. 
The Banach spaces of bounded (resp., compact) linear operators on $\cH$ is
denoted by $\cB(\cH)$ (resp., $\cB_{\infty}(\cH)$). The corresponding 
$\ell^p$-based trace ideals will be denoted by $\cB_p (\cH)$, $p>0$. 

The form sum (resp.\ difference) of two self-adjoint operators $A$ and $W$ will be denoted by $A +_q W$ (resp., $A -_q W = A +_q (-W)$). 

The operator closure of a closable operator $T$ in $\cH$ will be denoted by 
$\ol T$.  

\medskip

\section{Preliminaries on One-Dimensional Schr\"odinger Operators and Associated Fredholm Determinants}  \lb{s3}

In this section we introduce the basics of one-dimensional Schr\"odinger operators,  
including Green's functions, Krein-type resolvent formulas, the integral kernel of 
the square root of the resolvent of the Dirichlet Laplacian on $(0,R)$ and $(0,\infty)$, 
and discuss Jost--Pais-type reductions of associated Fredholm determinants to 
Wronski determinants. 

We start with the case of finite interval Schr\"odinger operators on $(0,R)$. 

Let $R>0$, $\alpha, \beta \in [0,\pi)$, 
\begin{equation}
V \in L^1((0,R); dx) \, \text{ real-valued},    \lb{3.1}
\end{equation}
and introduce the differential expression $\tau$ by
\begin{equation}
\tau = - \f{d^2}{dx^2} + V(x), \quad x \in J,    \lb{3.2}
\end{equation}
with $J \subseteq \bbR$ an appropriate interval.

We start by considering the self-adjoint (regular) Schr\"odinger operator 
$H_{(0,R),\alpha,\beta}$ in $L^2((0,R); dx)$, 
\begin{align}
& (H_{(0,R),\alpha,\beta} f)(x) = (\tau f)(x), \quad x \in (0,R),   \no \\
& f \in \dom(H_{(0,R),\alpha,\beta}) =\big\{g \in L^2((0,R); dx) \, \big|\, 
g, g' \in AC([0,R]);   \no \\ 
& \;\,\, \sin(\alpha) g^{\prime}(0)+\cos(\alpha) g(0)=0, \, 
\sin(\beta) g^{\prime}(R)+\cos(\beta) g(R)=0;    \lb{3.3} \\
& \hspace*{6.45cm}  \tau f \in L^2((0,R); dx)\big\}.     \no 
\end{align}

Denoting by $\psi_{0,\alpha}(z,\cdot)$ and $\psi_{R,\beta}(z,\cdot)$ the Weyl-type solutions for 
$H_{(0,R),\alpha,\beta}$ satisfying $ \tau_{(0,R)} \psi = z \psi$ on $(0,R)$ and the boundary conditions in \eqref{3.3},
\begin{align} 
\begin{split} 
\sin(\alpha) \psi_{0,\alpha}^{\prime}(z,0)+\cos(\alpha) \psi_{0,\alpha}(z,0)&=0,  \\
\sin(\beta) \psi_{R,\beta}^{\prime}(z,R)+\cos(\beta) \psi_{R,\beta}(z,R)&=0,  
\lb{3.4}
\end{split}
\end{align}
at $0$ and $R$, respectively, the resolvent of $H_{(0,R),\alpha,\beta}$ is given by
\begin{align} 
\begin{split} 
& \big((H_{(0,R),\alpha,\beta} - z I_{(0,R)})^{-1}f \big)(x)=\int_0^R dx^{\prime} \, 
G_{(0,R),\alpha,\beta}(z,x,x^{\prime})f(x^{\prime}),    \lb{3.5} \\ 
& \hspace*{2.45cm} x\in (0,R), \; z\in \rho(H_{(0,R),\alpha,\beta}), \; f \in L^2((0,R); dx), 
\end{split}
\end{align}
with the Green's function $G_{(0,R),\alpha,\beta}$ of $H_{(0,R),\alpha,\beta}$ expressed in terms of the Weyl solutions $\psi_{0,\alpha}$ and $\psi_{R,\beta}$ by 
\begin{align} 
\begin{split} 
& G_{(0,R),\alpha,\beta}(z,x,x^{\prime}) 
= (H_{(0,R),\alpha,\beta}-z I_{(0,R)})^{-1}(x,x')     \\
&\quad =\frac{-1}{W(\psi_{0,\alpha}(z,\cdot),\psi_{R,\beta}(z,\cdot))}\begin{cases}
\psi_{0,\alpha}(z,x) \, \psi_{R,\beta}(z,x^{\prime}), & 0\leq x\leq x^{\prime}\leq R, \\
\psi_{0,\alpha}(z,x^{\prime}) \, \psi_{R,\beta}(z,x), & 0\leq x^{\prime}\leq x\leq R. 
\end{cases}    \lb{3.6}
\end{split}
\end{align}
Here the symbol $I_J$ abbreviates the identity operator in $L^2(J; dx)$ and 
\begin{equation}
W(f,g)(x)= f(x) g'(x) - f'(x) g(x), \quad x \in J,    \lb{3.7}
\end{equation}
denotes the Wronskian of $f$ and $g$ ($f, g \in C^1(J)$), with $J \subseteq \bbR$ 
an appropriate interval.

\medskip

The corresponding half-line Schr\"odinger operators on $(0,\infty)$ are 
defined as follows.

Let $\alpha \in [0,\pi)$, 
\begin{equation}
V \in L^1((0,\infty); dx) \, \text{ real-valued},    \lb{3.8}
\end{equation}
and consider the self-adjoint Schr\"odinger operator $H_{(0,\infty),\alpha}$ in 
$L^2((0,\infty); dx)$, 
\begin{align}
& (H_{(0,\infty),\alpha} f)(x) = (\tau f)(x), \quad x \in (0,\infty),   \no \\
& f \in \dom(H_{(0,\infty),\alpha}) =\big\{g \in L^2((0,\infty); dx) \, \big|\, 
g, g' \in AC([0,R]) \, \text{for all $R>0$};      \lb{3.9} \\ 
& \hspace*{3.7cm}  \sin(\alpha) g^{\prime}(0)+\cos(\alpha) g(0)=0; \, 
\tau f \in L^2((0,\infty); dx)\big\}.    \no
\end{align}
Again, we denote by $\psi_{0,\alpha}(z,\cdot)$ and $\psi_+(z,\cdot)$ the Weyl-type 
solutions for $H_{(0,\infty),\alpha}$ satisfying $\tau \psi = z \psi$ on $(0,\infty)$ and  
\begin{equation} 
\psi_+(z,\cdot)\in L^2((0,\infty);dx), \quad z\in \bbC\backslash \bbR,   \lb{3.10}
\end{equation}
with $\psi_{0,\alpha}(z,\cdot)$ satisfying the boundary condition in \eqref{3.9}
\begin{equation} 
\sin(\alpha) \psi_{0,\alpha}^{\prime}(z,0)+\cos(\alpha) \psi_{0,\alpha}(z,0)=0.   
\lb{3.11}
\end{equation}
One notes that $\psi_+(z,\cdot)$ is unique up to constant (possibly, 
$z$-dependent) multiples, and no boundary condition is needed at $+\infty$ 
due to the fact that assumption \eqref{3.8} implies the limit point case of 
$\tau$ at $+\infty$.
 
The resolvent of $H_{(0,\infty),\alpha}$ is then given by 
\begin{align} 
\begin{split} 
& \big((H_{(0,\infty),\alpha}-z I_{(0,\infty)})^{-1}f \big)(x)
=\int_0^{\infty}dx^{\prime} \, 
G_{(0,\infty),\alpha}(z,x,x^{\prime})f(x^{\prime}),    \lb{3.12} \\
& \hspace*{2.25cm} x\in (0,\infty), \; z\in \rho(H_{(0,\infty),\alpha}), 
\; f \in L^2((0,\infty); dx),
\end{split}
\end{align}
with the Green's function $G_{(0,\infty),\alpha}$ of $H_{(0,\infty),\alpha}$ 
expressed in terms of $\psi_{0,\alpha}$ and $\psi_+$ by
\begin{align}
\begin{split}  
& G_{(0,\infty),\alpha}(z,x,x^{\prime}) 
= (H_{(0,\infty),\alpha} - z I_{(0,\infty)})^{-1}(x,x')    \lb{3.13} \\
& \quad = \frac{-1}{W(\psi_{0,\alpha}(z,\cdot),\psi_+(z,\cdot))} \begin{cases}
\psi_{0,\alpha}(z,x) \, \psi_+(z,x^{\prime}), & 0\leq x\leq x^{\prime} < \infty, \\
\psi_{0,\alpha}(z,x^{\prime}) \, \psi_+(z,x), & 0\leq x^{\prime}\leq x < \infty. \end{cases}
\end{split} 
\end{align}

\medskip

We start with some Green's function formulas. First, we compare 
$G_{(0,R),\alpha,\beta}$ and $G_{(0,\infty),\alpha}$:

\begin{lemma}\lb{l3.1a}
Assume \eqref{3.8}. Then the Green's functions $G_{(0,R),\alpha,\beta}$ and 
$G_{(0,\infty),\alpha}$ defined in \eqref{3.6} and \eqref{3.13} satisfy the relation
\begin{align}
& G_{(0,R),\alpha,\beta}(z,x,x^{\p})=G_{(0,\infty),\alpha}(z,x,x^{\p}) \no\\
& \quad 
+\frac{\big[\sin(\beta)\psi_+^{\p}(z,R)+\cos(\beta)\psi_+(z,R) \big]
\psi_{0,\alpha}(z,x)\psi_{0,\alpha}(z,x^{\p})}
{\big[\sin(\beta)\psi_{0,\alpha}^{\p}(z,R)
+\cos(\beta)\psi_{0,\alpha}(z,R) \big]
W(\psi_{0,\alpha}(z,\cdot),\psi_+(z,\cdot))},     \lb{BT} \\
&\hspace*{3cm} 
z\in \rho(H_{(0,R),\alpha,\beta})\cap\rho(H_{(0,\infty),\alpha}), \; 0\leq x,x^{\p}\leq R.    \no 
\end{align}
\end{lemma}
\begin{proof}
Define
\begin{align}
g_{\alpha,\beta}(z,x)&=\int_0^R dx^{\p} \, G_{(0,R),\alpha,\beta}(z,x,x^{\p})f(x^{\p}), \quad x\in [0,R],\lb{B9}\\
g_{\alpha}(z,x)&=\int_0^{\infty}dx^{\p} \, G_{(0,\infty),\alpha}(z,x,x^{\p})f(x^{\p}), \quad x\geq 0,\lb{B10}\\
& \hspace*{-1.3cm} 
z\in \rho(H_{(0,\infty),\alpha})\cap\rho(H_{(0,R),\alpha,\beta}), \ f\in L^2((0,\infty);dx). \no 
\end{align}
Evidently,
\begin{equation}\lb{B11}
\bigg(-\frac{d^2}{dx^2}+V(x)-z \bigg)g_{\alpha,\beta}(z,x)=f(x)
=\bigg(-\frac{d^2}{dx^2}+V(x)-z \bigg)g_{\alpha}(z,x), \quad x\in [0,R], 
\end{equation}
that is, 
\begin{equation}\lb{B12}
\bigg(-\frac{d^2}{dx^2}+V(x)-z \bigg)\big(g_{\alpha,\beta}(z,x)-g_{\alpha}(z,x) \big)=0, 
\quad x\in [0,R].
\end{equation}
Therefore, $g_{\alpha,\beta}-g_{\alpha}$ is expressible as a linear combination 
of $\psi_{a,\alpha}(z,\cdot)$ and $\psi_+(z,\cdot)$ on $[0,R]$.  More precisely, there exist 
$z$-dependent constants $c_1(z)$ and $c_2(z)$ such that
\begin{align}\lb{B13}
\begin{split} 
& g_{\alpha,\beta}(z,x)=g_{\alpha}(z,x)+c_1(z)\psi_{0,\alpha}(z,x)
+c_2(z)\psi_+(z,x),   \\
& \hspace*{1.75cm}
x\in [0,R], \; z\in \rho(H_{(0,R),\alpha,\beta})\cap\rho(H_{(0,\infty),\alpha}).
\end{split} 
\end{align}
Since $g_{\alpha,\beta}(z,\cdot)$ belongs to the domain of $H_{(0,R),\alpha,\beta}$, it satisfies the boundary conditions in \eqref{3.3}, 
\begin{align}
\sin(\alpha)g_{\alpha,\beta}^{\p}(z,0)+\cos(\alpha)g_{\alpha,\beta}(z,0)&=0,  \lb{B15}\\
\sin(\beta)g_{\alpha,\beta}^{\p}(z,R)+\cos(\beta)g_{\alpha,\beta}(z,R)&=0.   \lb{B16}
\end{align}
Equation \eqref{B10} then yields 
\begin{align}
& g_{\alpha,\beta}(z,x)=-\frac{\psi_+(z,x)}{W_{0,\alpha,+}} \int_0^x dx^{\p} \, 
\psi_{0,\alpha}(z,x^{\p})f(x^{\p}) \no \\
&\quad -\frac{\psi_{0,\alpha}(z,x)}{W_{0,\alpha,+}}\int_x^R dx^{\p} \,\psi_+(z,x^{\p})f(x^{\p})
+c_1(z)\psi_{0,\alpha}(z,x)+c_2(z) \psi_+(z,x),     \lb{B17}\\
& g_{\alpha,\beta}^{\p}(z,x)= - \frac{\psi_+^{\p}(z,x)}{W_{0,\alpha,+}}\int_0^x dx^{\p} \, 
\psi_{0,\alpha}(z,x^{\p})f(x^{\p})    \no \\
&\quad -\frac{\psi_{0,\alpha}^{\p}(z,x)}{W_{0,\alpha,+}}\int_x^R dx^{\p} \, \psi_+(z,x^{\p})f(x^{\p}) +c_1(z)\psi_{0,\alpha}^{\p}(z,x)+c_2(z)\psi_+^{\p}(z,x),     \lb{B18}
\end{align}
where we abbreviated
\begin{equation}
W_{0,\alpha,+} = W(\psi_{0,\alpha}(z,\cdot),\psi_+(z,\cdot)).     \lb{B18A}
\end{equation}
Thus, \eqref{B15} becomes
\begin{align}
&\sin(\alpha)\bigg[\frac{-\psi_{0,\alpha}^{\p}(z,0)}{W_{0,\alpha,+}}\int_0^R dx^{\p} \, 
\psi_+(z,x^{\p})f(x^{\p})+c_1\psi_{0,\alpha}^{\p}(z,0)+c_2\psi_+^{\p}(z,0)\bigg]   \lb{B19} \\
&\quad +\cos(\alpha)\bigg[\frac{-\psi_{0,\alpha}(z,0)}{W_{0,\alpha+}}\int_0^R dx^{\p}
\, \psi_+(z,x^{\p})f(x^{\p}) + c_1\psi_{0,\alpha}(z,0)+c_2\psi_+(z,0) \bigg]=0.  \no 
\end{align}
Together with
\begin{equation}\lb{B21}
\sin(\alpha)\psi_{0,\alpha}^{\p}(z,0)+\cos(\alpha)\psi_{0,\alpha}(z,0)=0,
\end{equation}
since $\psi_{0,\alpha}$ satisfies the boundary condition at $0$, \eqref{B19} implies
\begin{equation}\lb{B22}
c_2(z)=0.
\end{equation}
To obtain \eqref{B22}, we have made use of the fact that
\begin{equation}
\sin(\alpha)\psi_+^{\p}(z,0)+\cos(\alpha)\psi_+(z,0)\neq 0;
\end{equation}
otherwise, $\psi_+(z,0)$ would satisfy the boundary condition in \eqref{3.9}, and thus be an eigenfunction of $H_{(0,\infty),\alpha}$, implying $z\in \sigma(H_{(0,\infty),\alpha})$.
In view of \eqref{B17} and \eqref{B18}, \eqref{B16} can be rewritten as
\begin{align}
&\sin(\beta)\bigg[\frac{-\psi^{\p}_+(z,R)}{W_{0,\alpha,+}}\int_0^R dx^{\p} \, 
\psi_{0,\alpha}(z,x^{\p})f(x^{\p})+c_1(z) \psi_{0,\alpha}^{\p}(z,R) 
+ c_2(z) \psi_+^{\p}(z,R)\bigg]    \no \\
& \quad +\cos(\beta)\bigg[\frac{-\psi_+(z,R)}{W_{0,\alpha,+}}\int_0^R dx^{\p} \, 
\psi_{0,\alpha}(z,x^{\p})f(x^{\p}) + c_1(z) \psi_{0,\alpha}(z,R)   \no \\
& \hspace*{7.3cm}  + c_2(z) \psi_+(z,R) \bigg]=0,    \lb{B20} 
\end{align}
which, together with \eqref{B22} yields
\begin{equation}\lb{B23}
c_1(z)=\frac{\big[ \sin(\beta)\psi_+^{\p}(z,R)+\cos(\beta)\psi_+(z,R)\big]
\frac{1}{W_{0,\alpha,+}}\int_0^R dx^{\p} \, \psi_{0,\alpha}(z,x^{\p})f(x^{\p})}{\sin(\beta)
\psi_{0,\alpha}^{\p}(z,R)+\cos(\beta)\psi_{0,\alpha}(z,R)}.
\end{equation}
One notes that the denominator in \eqref{B23} is always nonzero; otherwise, by  
\eqref{B21}, $\psi_{0,\alpha}(z,x)$ would satisfy \eqref{3.3}, and thus be an eigenfunction 
of $H_{(0,R),\alpha,\beta}$, implying $z\in \sigma(H_{(0,R),\alpha,\beta})$.

The result \eqref{BT} now follows by substituting \eqref{B22} 
and \eqref{B23} into \eqref{B17}, noting that $f\in L^2((0,\infty);dx)$ was arbitrary. 
\end{proof}

Next, we compare the half-line Green's functions $G_{(0,\infty),\widetilde{\alpha}}$ 
and $G_{(0,\infty),\alpha}$, that is, we investigate the integral kernels associated with 
a special case of Krein's formula for resolvents (cf.\ \cite[\& 106]{AG81}):

\begin{lemma}\lb{l3.2a} 
Assume \eqref{3.8} and $\widetilde{\alpha},\alpha\in [0,\pi)$ with $\widetilde{\alpha}\neq \alpha$. Then the Green's functions $G_{(0,\infty),\widetilde{\alpha}}$ and 
$G_{(0,\infty),\alpha}$ defined in \eqref{3.13} satisfy the relation 
\begin{align}
&G_{(0,\infty),\widetilde{\alpha}}(z,x,x^{\p})=G_{(0,\infty),\alpha}(z,x,x^{\p}) \no \\
&\quad +\frac{\big[\sin(\widetilde{\alpha})\psi_{0,\alpha}^{\p}(z,0)
+\cos(\widetilde{\alpha})\psi_{0,\alpha}(z,0) \big]\psi_+(z,x)\psi_+(z,x^{\p})}
{W_{0,\alpha,+}\big[\sin(\widetilde{\alpha})\psi_+^{\p}(z,0)+\cos(\widetilde{\alpha})
\psi_+(z,0)\big]},      \lb{B39} \\
&\hspace*{3.05cm} 
z\in \rho(H_{(0,\infty),\widetilde{\alpha}})\cap\rho(H_{(0,\infty),\alpha}), 
\; x,x^{\p}\geq 0.    \no 
\end{align}
\end{lemma}
\begin{proof}
For $\alpha\in [0,\pi)$, define
\begin{align}
\begin{split} 
& g_{\alpha}(z,x)= \int_0^{\infty}dx^{\p} \, G_{(0,\infty),\alpha}(z,x,x^{\p}) f(x'),     \\
& x\geq 0, \; z\in \rho(H_{(0,\infty),\alpha}), \; f\in L^2((0,\infty);dx).   \lb{B28} 
\end{split}
\end{align}
If $\widetilde{\alpha}\neq \alpha$, by the same argument leading to \eqref{B13} in the 
proof of Lemma \ref{l3.1a}, there are $z$-dependent constants, $c_1(z)$ and 
$c_2(z)$, such that
\begin{align}
g_{\widetilde{\alpha}}(z,x)&=g_{\alpha}(z,x)+c_1(z)\psi_{0,\alpha}(z,x) 
+c_2(z)\psi_+(z,x),    \lb{B29}\\
&\quad \,\, x\in[0,\infty), \; 
z\in \rho(H_{(0,\infty),\widetilde{\alpha}})\cap\rho(H_{(0,\infty),\alpha}). \no 
\end{align}
Since $g_{\widetilde{\alpha}},g_{\alpha}\in L^2((0,\infty);dx)$, one concludes that 
\begin{equation}\lb{B31}
c_1(z)=0, \quad z\in \rho(H_{(0,\infty),\widetilde{\alpha}})\cap\rho(H_{(0,\infty),\alpha}). 
\end{equation}
as $\psi_{0,\alpha}(z,\cdot)$ (resp., $\psi_{0,\wti \alpha}(z,\cdot)$) corresponding 
to $\alpha$ (resp., $\widetilde{\alpha}$) is not 
$L^2((0,\infty);dx)$ for $z\in \rho(H_{(0,\infty),\alpha})$ (resp., 
$z\in \rho(H_{(0,\infty),\widetilde{\alpha}})$).

Moreover, since
\begin{equation}  \lb{B32}
g_{\alpha}\in \dom(H_{(0,\infty),\alpha})\, \text{ and } \,  
g_{\widetilde{\alpha}}\in \dom(H_{(0,\infty),\widetilde{\alpha}}),
\end{equation}
each must satisfy the boundary condition in \eqref{3.9} at $0$, 
\begin{align}
\sin(\widetilde{\alpha})g_{\widetilde{\alpha}}^{\p}(z,0) 
+ \cos(\widetilde{\alpha}) g_{\widetilde{\alpha}}(z,0)&=0,      \lb{B33a}\\
\sin(\alpha)g_{\alpha}^{\p}(z,0)+\cos(\alpha)g_{\alpha}(z,0)&=0.       \lb{B33b}
\end{align}
Applying \eqref{B28},
\begin{align}
g_{\widetilde{\alpha}}(z,x)&=-\frac{1}{W_{0,\alpha,+}}\int_0^x dx^{\p} \, 
\psi_{0,\alpha}(z,x^{\p})\psi_+(z,x^{\p})f(z^{\p})  \no \\
&\quad-\frac{1}{W_{0,\alpha,+}}\int_x^{\infty}dx^{\p} \, \psi_{0,\alpha}(z,x^{\p})
\psi_+(z,x^{\p})f(x^{\p})  +c_2(z) \psi_+(z,x)     \lb{B34}\\
g_{\alpha,\beta}^{\p}(z,x)&= -\frac{\psi_+^{\p}(z,x)}{W_{0,\alpha,+}}\int_0^x dx^{\p} \,
\psi_{0,\alpha}(z,x^{\p})f(x^{\p}) \no \\
&\quad  -\frac{\psi_{0,\alpha}^{\p}(z,x)}{W_{0,\alpha,+}}\int_x^{\infty} dx^{\p} \, 
\psi_+(z,x^{\p})f(x^{\p}) + c_2(z) \psi_+(z,x),      \lb{B35}\\
&\hspace*{4.4cm} x\geq 0, \; f\in L^2((0,\infty);dx). \no 
\end{align}
Therefore, \ref{B33a} becomes
\begin{align}
&\sin(\widetilde{\alpha})\bigg[\frac{\psi_{0,\alpha}^{\p}(z,0)}
{W_{0,\alpha,+}}\int_0^{\infty} dx^{\p} \, \psi_+(z,x^{\p})f(x^{\p})
+ c_2(z) \psi_+^{\p}(z,0)\bigg] \no \\
&\quad + \cos(\widetilde{\alpha})\bigg[\frac{\psi_{0,\alpha}(z,0)}
{W_{0,\alpha,+}}\int_0^{\infty} dx^{\p} \, \psi_+(z,x^{\p})f(x^{\p})
+ c_2(z) \psi_+(z,0) \bigg]=0.    \lb{B36}
\end{align}
Hence,
\begin{equation}    \lb{B37}
c_2(z)=\frac{\big[\sin(\widetilde{\alpha})\psi_{0,\alpha}^{\p}(z,0)+\cos(\widetilde{\alpha})\psi_{0,\alpha}(z,0) \big]\frac{1}{W_{0,\alpha,+}}\int_0^{\infty} dx^{\p} \, 
\psi_+(z,x^{\p})f(x^{\p})}{\sin(\widetilde{\alpha})\psi_+^{\p}(z,0) 
+ \cos(\widetilde{\alpha})\psi_+(z,0)}.
\end{equation}
One notes that the denominator in \eqref{B37} is nonzero; otherwise,  
$z\in \sigma(H_{(0,\infty),\widetilde{\alpha}})$.  Since $f\in L^2((0,\infty);dx)$ is arbitrary, substituting \eqref{B31} and \eqref{B37} into \eqref{B29} yields \eqref{B39}.
\end{proof}

In the following, specializing to the unperturbed case, where $V = 0$ on 
$(0,R)$, we abbreviate the corresponding quantities with the superscript $(0)$. This notation applies to the unperturbed Schr\"odinger operator $H_{(0,R),\alpha,\beta, }^{(0)}$, the unperturbed differential expression $\tau^{(0)} = - d^2/dx^2$, the unperturbed 
Green's function $G_{(0,R),\alpha,\beta, }^{(0)}$, and to the corresponding unperturbed Weyl-type solutions of $\tau^{(0)} \psi = z \psi$ on $(0,R)$. The latter are given by  
\begin{align}
\psi_{0,\alpha}^{(0)}(z,x)&=\cos(\alpha)\frac{\sin(z^{1/2}x)}{z^{1/2}} 
- \sin(\alpha)\cos(z^{1/2}x)    \no \\
\psi_{R,\beta}^{(0)}(z,x)&=\cos(\beta)\frac{\sin(z^{1/2}(R-x))}{z^{1/2}} 
+ \sin(\beta) \cos(z^{1/2}(R-x)),    \lb{3.14} \\
&\hspace*{4.27cm}  x \in (0,R), \; \Im(z^{1/2})\geq 0,    \no
\end{align}
implying 
\begin{align}
W\big(\psi_{R,\beta}^{(0)}(z,\cdot),\psi_{0,\alpha}^{(0)}(z,\cdot)\big) 
&=\sin(\beta-\alpha)\cos(z^{1/2}R) 
+ \sin(\alpha)\sin(\beta)z^{1/2}\sin(z^{1/2}R)    \no \\
&\quad + \cos(\alpha)\cos(\beta)\frac{\sin(z^{1/2}R)}{z^{1/2}}.     \lb{3.15}
\end{align}

For the corresponding perturbed solutions (i.e., for $V \neq 0$) one obtains 
\begin{align}
\psi_{0,\alpha}(z,x)&=\psi_{0,\alpha}^{(0)}(z,x) 
+ \int_0^x dx' \, \frac{\sin(z^{1/2}(x-x'))}{z^{1/2}}V(x')\psi_{0,\alpha}(z,x'), 
\lb{3.16a} \\
\psi_{R,\beta}(z,x)&=\psi_{R,\beta}^{(0)}(z,x) 
- \int_x^R dx' \, \frac{\sin(z^{1/2}(x-x'))}{z^{1/2}} V(x') 
\psi_{R,\beta}(z,x'),     \lb{3.16} \\
&\hspace*{4.9cm}   0\leq x\leq R, \; \Im(z^{1/2}) \geq 0,    \no
\end{align}
implying 
\begin{align} 
\psi_{0,\alpha}^{\p}(z,x)&=\psi_{0,\alpha}^{(0)\p}(z,x) + \int_0^x dx^{\p} \, 
\cos(z^{1/2}(x-x^{\p})) V(x^{\p}) \psi_{0,\alpha}(z,x^{\p}),    \no \\
\psi_{R,\beta}^{\p}(z,x)&=\psi_{R,\beta}^{(0)\p}(z,x) - \int_x^R dx^{\p} \, 
\cos(z^{1/2}(x-x^{\p})) V(x^{\p}) \psi_{R,\beta}(z,x^{\p}),  \lb{3.17}\\
&\hspace*{5cm} x \in (0, R), \; \Im(z^{1/2}) \geq 0, \no
\end{align}
and hence,
\begin{align}
\begin{split}
\psi_{0,\alpha}(z,0) &= \psi_{0,\alpha}^{(0)}(z,0), \quad \;\;\;
\psi_{0,\alpha}^{\p}(z,0) = \psi_{0,\alpha}^{(0)\p}(z,0).   \lb{3.18} \\
\psi_{R,\beta}(z,R) &= \psi_{R,\beta}^{(0)}(z,R), \quad 
\psi_{R,\beta}^{\p}(z,R) = \psi_{R,\beta}^{(0)\p}(z,R).     
\end{split} 
\end{align}

In the half-line case $(0,\infty)$ one introduces, in obvious notation, the unperturbed Schr\"odinger operator $H_{(0,\infty),\alpha}^{(0)}$, the unperturbed 
Green's function $G_{(0,\infty),\alpha}^{(0)}$, and the corresponding unperturbed 
Weyl-type solutions of $\tau^{(0)} \psi = z \psi$ on $(0,\infty)$. The latter are given by  
\begin{align}
\begin{split} 
\psi_{0,\alpha}^{(0)}(z,x) = \cos(\alpha)\frac{\sin(z^{1/2}x)}{z^{1/2}} 
- \sin(\alpha)\cos(z^{1/2}x),   \quad 
\psi_+^{(0)}(z,x) = e^{iz^{1/2}x},&   \lb{3.19} \\
x \in (0,\infty), \; \Im(z^{1/2})\geq 0,& 
\end{split}
\end{align}
implying 
\begin{equation}\lb{3.20}
W\big(\psi_+^{(0)}(z,\cdot),\psi_{0,\alpha}^{(0)}(z,\cdot)\big)=\cos(\alpha)+i z^{1/2}\sin(\alpha).
\end{equation}

For the corresponding perturbed solutions (i.e., for $V \neq 0$) one obtains 
\begin{align}
\psi_{0,\alpha}(z,x)&=\psi_{0,\alpha}^{(0)}(z,x) 
+ \int_0^x dx^{\p} \, \frac{\sin(z^{1/2}(x-x^{\p}))}{z^{1/2}} V(x^{\p}) 
\psi_{0,\alpha}(z,x^{\p}),    \no \\
\psi_+(z,x)&=\psi_+^{(0)}(z,x) - \int_x^{\infty} dx^{\p} \, 
\frac{\sin(z^{1/2}(x-x^{\p}))}{z^{1/2}} V(x^{\p}) \psi_+(z,x^{\p}),   \lb{3.21} \\
&\hspace*{4.65cm} x \in (0,\infty), \; \Im(z^{1/2}) \geq 0,    \no
\end{align}
implying
\begin{align}
\psi_{0,\alpha}^{\p}(z,x)&=\psi_{0,\alpha}^{(0)\p}(z,x) + \int_0^x dx^{\p} \, 
\cos(z^{1/2}(x-x^{\p})) V(x^{\p}) \psi_{0,\alpha}(z,x^{\p}), \no \\
\psi_+^{\p}(z,x)&=\psi_+^{(0)\p}(z,x) - \int_x^{\infty} dx^{\p} \, 
\cos(z^{1/2}(x-x^{\p})) V(x^{\p}) \psi_+(z,x^{\p}),    \lb{3.22}\\
&\hspace*{4.7cm} x \in (0,\infty), \; \Im(z^{1/2}) \geq 0,   \no
\end{align}
and hence, 
\begin{equation}\lb{3.23}
W(\psi_+(z,\cdot),\psi_{0,\alpha}(z,\cdot))=\cos(\alpha)\psi_+(z,0)+\sin(\alpha)\psi_+^{\p}(z,0).
\end{equation}

To make the connection between Fredholm determinants and ratios of Wronskians, we first need a technical result on integral kernels of square roots 
of resolvents in the special case $V=0$ a.e.\ on $(0,R)$. For simplicity we state the result in the case of Dirichlet boundary conditions (i.e., for $\alpha = \beta = 0$) only:

We start by recalling the explicit formulas for the Dirichlet Green's functions, 
\begin{align}
& G_{(0,R),0,0}^{(0)}(z,x,x')   \no \\
& \quad = \f{1}{z^{1/2}\sin(z^{1/2}R)} \begin{cases} 
\sin(z^{1/2}x) \sin(z^{1/2}(R-x')), & 0 \leq x \leq x' \leq R, \\
\sin(z^{1/2}x') \sin(z^{1/2}(R-x)), & 0 \leq x' \leq x \leq R,
\end{cases}     \lb{3.23a} \\[1mm]
& G_{(0,\infty),0}^{(0)}(z,x,x') = \f{1}{z^{1/2}} \begin{cases} 
\sin(z^{1/2}x) e^{z^{1/2}x'}, & 0 \leq x \leq x' < \infty, \\
\sin(z^{1/2}x') e^{z^{1/2}x}, & 0 \leq x' \leq x < \infty. 
\end{cases}     \lb{3.23b}
\end{align}

\begin{lemma} \lb{l3.1}
Let $E>0$. \\
$(i)$ The integral kernel $R_{(0,R),0,0}^{(0),1/2}(-E,\cdot,\cdot)$
for $(H_{(0,R),0,0}^{(0)} + E)^{-1/2}$ 
is given by 
\begin{equation}\lb{3.24}
R_{(0,R),0,0}^{(0), 1/2}(-E,x,x^{\p})=\begin{cases}
2\pi^{-1}\int_{E^{1/2}}^{\infty}\frac{\sinh(sx)}{(s^2-E)^{1/2}}\frac{\sinh(s(R-x^{\p}))}{\sinh(sR)}ds, & 0\leq x \leq x^{\p} \leq R, \\
2\pi^{-1}\int_{E^{1/2}}^{\infty}\frac{\sinh(sx^{\p})}{(s^2-E)^{1/2}}\frac{\sinh(s(R-x))}{\sinh(sR)} ds, & 0\leq x^{\p} \leq x \leq R.
\end{cases} 
\end{equation}
$(ii)$ The integral kernel $R_{(0,\infty),0}^{(0), 1/2}(-E,\cdot,\cdot)$ of 
$(H_{(0,\infty),0} + E)^{-1/2}$ is given by 
\begin{equation}\lb{3.25}
R_{(0,\infty),0}^{(0), 1/2}(-E,x,x^{\p})
= \pi^{-1} \big[K_0\big(E^{1/2}|x^{\p}-x|\big)-K_0\big(E^{1/2}(x+x^{\p})\big)\big], 
\quad x, x' \in (0,\infty), 
\end{equation}
where $K_0(\cdot)$ denotes the modified irregular Bessel function of order zero 
$($cf.\ \cite[Sect.\ 9.6]{AS72}$)$.  
\end{lemma}
\begin{proof}
If $A$ is positive-type operator with integral kernel $H(t,x,x^{\p})$ such that
\begin{equation}
[(A+t I_J)^{-1}u](x)=\int_{J} dx' \, H(t,x,x^{\p})u(x^{\p}), \quad x\in J, \; t>0,  \lb{3.26}
\end{equation}
$J \subset \bbR$ an appropriate interval, then the operator $A^{-q}$, 
$0 < q < 1$, has the integral kernel $R^q (\cdot,\cdot)$ 
(cf., e.g., \cite[Sect.\ 16]{KZPS76}) 
\begin{equation}
R^q (x,x^{\p})=\frac{\sin(\pi q)}{\pi} \int_{0}^{\infty} dt \, 
t^{-q} H(t,x,x^{\p}), \quad x, x' \in J.   \lb{3.27}
\end{equation}

Applying \eqref{3.27} to $A=H_{(0,R),0,0}^{(0)} + E$, one calculates for 
$0\leq x \leq x^{\p} \leq R$,
\begin{align}
R_{(0,R),0,0}^{(0), 1/2}(-E,x,x^{\p})&=\pi^{-1} \int_0^{\infty} dt \, t^{-1/2} 
G_{(0,R),0,0}^{(0)}(-E-t,x,x^{\p})    \no \\
&=\pi^{-1}\int_0^{\infty} dt \, t^{-1/2} \psi_{0,0}^{(0)}(-E-t,x) 
\psi_{R,0}^{(0)}(-E-t,x^{\p}) \no\\
&=2\pi^{-1}\int_{|E|^{1/2}}^{\infty} ds \, \frac{\sinh(sx)}{(s^2-E)^{1/2}}
\frac{\sinh(s(R-x^{\p}))}{\sinh(sR)}.  \lb{3.28}
\end{align}
For the case $0\leq x^{\p}\leq x$ one simply interchanges the roles of $x$ 
and $x^{\p}$.

Similarly, applying \eqref{3.27} to $A = H_{(0,\infty),0}^{(0)} + E$, one calculates for 
$0\leq x \leq x^{\p}$, 
\begin{align}
R_{(0,\infty),0}^{(0), 1/2}(-E,x,x^{\p})&=\pi^{-1}\int_{0}^{\infty} dt \, t^{-1/2} 
G_{(0,\infty),0}^{(0)}(-E-t,x,x^{\p})    \no\\
&=\pi^{-1}\int_0^{\infty} dt \, t^{-1/2} \psi_{0,0}^{(0)}(-E-t,x) 
\psi_{+}^{{(0)}}(-E-t,x^{\p})   \no \\
&=2\pi^{-1}\int_{E^{1/2}}^{\infty} ds \, \frac{\sinh(sx)}{(s^2-E)^{1/2}}e^{-sx^{\p}}  \no \\
&=\pi^{-1}\int_1^{\infty} ds \, \frac{e^{-E^{1/2}s(x^{\p}-x)}}{(s^2-1)^{1/2}}  
- \pi^{-1}\int_1^{\infty} ds \, \frac{e^{-E^{1/2}s(x+x^{\p})}}{(s^2-1)^{1/2}}.  \lb{3.29}
\end{align}
The last two integrals in \eqref{3.29} can be computed in terms of the modified irregular Bessel function of order zero, $K_0(\cdot)$ 
(see \cite[No.\ 8.432.3]{GR80}) and one obtains 
\begin{equation}
R_{(0,\infty),0}^{(0), 1/2}(-E,x,x^{\p}) 
= \pi^{-1}\big(K_0\big(E^{1/2}(x^{\p}-x)\big) - K_0\big(E^{1/2}(x+x^{\p})\big)\big), 
\quad 0\leq x \leq x^{\p}.   \lb{3.30}
\end{equation}
The case $0\leq x^{\p}\leq x$ is treated analogously. 
\end{proof}

\begin{remark} \lb{r3.2}
The integrand in \eqref{3.28} may be expressed as 
\begin{equation}\lb{3.31}
\frac{\sinh(sx)}{(s^2-E)^{1/2}}\frac{\sinh(s(R-x^{\p}))}{\sinh(sR)} 
= \frac{\sinh(sx)}{(s^2-E)^{1/2}}\bigg(e^{-sx^{\p}} 
- 2\frac{e^{-2Rs}}{1-e^{-2Rs}}\sinh(sx^{\p})\bigg), 
\end{equation}
and hence the integrand in \eqref{3.28} has the integrable majorant
\begin{equation}\lb{3.32}
\frac{\sinh(sx)}{(s^2-E)^{1/2}}\frac{\sinh(s(R-x^{\p}))}{\sinh(sR)} 
\leq \frac{\sinh(sx)}{(s^2-E)^{1/2}}e^{-sx^{\p}}, 
\end{equation}
implying 
\begin{equation}\lb{3.33}
0 \leq R_{(0,R),0,0}^{(0),1/2}(-E,x,x^{\p}) \leq R_{(0,\infty),0}^{(0), 1/2}(-E,x,x^{\p}), 
\quad 0\leq x,x^{\p}\leq R, \; E>0.
\end{equation}

More generally, assuming \eqref{3.8} and employing domain monotonicity for Dirichlet Green's functions (see, e.g., \cite[Sect.\ 1.VII.6]{Do84}), that is, 
\begin{equation}
0 \leq G_{(0,R),0,0}(-E,x,x^{\prime}) \leq 
G_{(0,\infty),0}(-E,x,x^{\prime}), \quad 0\leq x, x' \leq R, \; E>0,  \lb{3.33a}
\end{equation}
and inserting \eqref{3.33a} into the analog of \eqref{3.27} yields
\begin{equation}\lb{3.33b}
0 \leq R_{(0,R),0,0}^q (-E,x,x^{\p}) \leq 
R_{(0,\infty),0}^q (-E,x,x^{\p}), 
\quad 0 \leq x,x^{\p}\leq R, \; E>0, \; 0< q <1.
\end{equation}
\end{remark}

\medskip

Next, we factor $V$ as
\begin{equation}
V(x)=u(x)v(x),\quad v(x)=|V(x)|^{1/2}, \quad u(x)=v(x) \, \sgn(V(x)), \quad 
x \in (0,\infty),    \lb{3.34}
\end{equation}
and denote 
\begin{equation}
V_R(x) = V(x)|_{(0,R)}, \quad v_R(x) = v(x)|_{(0,R)}, \quad 
u_R(x) = u(x)|_{(0,R)}, \quad x \in (0,R).    \lb{3.35}
\end{equation}

\begin{theorem}\lb{t3.3}
Let $R>0$ and $V_R \in L^1((0,R); dx)$ be real-valued, and assume that 
$z \in \rho \big(H_{(0,R),\alpha,\beta, }^{(0)}\big)$. Then
\begin{equation}
\ol{u_R\big(H_{(0,R),\alpha,\beta, }^{(0)}-z I_{(0,R)}\big)^{-1}v_R} 
\in \cB_1\big(L^2((0,R); dx)\big),      \lb{3.36}
\end{equation}
and the determinant 
${\det}_{L^2((0,R); dx)}\Big(I_{(0,R)}+\ol{u_R 
\big(H_{(0,R),\alpha,\beta, }^{(0)}-z I_{(0,R)}\big)^{-1}v_R}\Big)$ is a quotient of Wronskians,
\begin{align}
&{\det}_{L^2((0,R); dx)}\Big(I_{(0,R)} + 
\ol{u_R \big(H_{(0,R),\alpha,\beta, }^{(0)}-z I_{(0,R)}\big)^{-1}v_R}\Big) 
= \frac{W(\psi_{R,\beta}(z,\cdot),\psi_{0,\alpha}(z,\cdot))}
{W(\psi_{R,\beta}^{(0)}(z,\cdot),\psi_{0,\alpha}^{(0)}(z,\cdot))}  \label{03.4} \\
& \quad 
= [\sin(\beta)\psi_{0,\alpha}^{\p}(z,R) + \cos(\beta)\psi_{0,\alpha}(z,R)]
\big[\cos(\alpha)\cos(\beta) z^{-1/2}\sin(z^{1/2}R)   \no \\
& \qquad  - \sin(\alpha - \beta)\cos(z^{1/2}R)  
+ \sin(\alpha)\sin(\beta) z^{1/2}\sin(z^{1/2}R)\big]^{-1}
 \lb{3.37}  \\
& \quad = [\sin(\alpha)\psi_{R,\beta}^{\p}(z,0) + \cos(\alpha)\psi_{R,\beta}(z,0)]
\big[\cos(\alpha)\cos(\beta) z^{-1/2}\sin(z^{1/2}R)   \no \\
& \qquad  - \sin(\alpha - \beta)\cos(z^{1/2}R)  
+ \sin(\alpha)\sin(\beta) z^{1/2}\sin(z^{1/2}R)\big]^{-1}.  \lb{3.38}
\end{align} 
\end{theorem}
\begin{proof}
To get \eqref{3.37} and \eqref{3.38}, observe that
\begin{align}
\frac{W(\psi_{R,\beta}(z,\cdot),\psi_{0,\alpha}(z,\cdot))}
{W(\psi_{R,\beta}^{(0)}(z,\cdot),\psi_{0,\alpha}^{(0)}(z,\cdot))}&=\frac{\psi_{R,\beta}(z,x)\psi_{0,\alpha}^{\prime}(z,x)-\psi_{R,\beta}^{\prime}(z,x)\psi_{0,\alpha}(z,x)}{W(\psi_{R,\beta}^{(0)}(z,\cdot),\psi_{0,\alpha}^{(0)}(z,\cdot))}, \label{03.10}\\
&\qquad \qquad \ \qquad \quad x\in [0,R], \ z \in \rho \big(H_{(0,R),\alpha,\beta, }^{(0)}\big).\no
\end{align}
In view of \eqref{3.14}, \eqref{3.15}, and \eqref{3.18}, the identity \eqref{3.37} (resp., \eqref{3.38}) is given by evaluating \eqref{03.10} at $x=R$ (resp. $x=0$).

In order to verify \eqref{03.4}, let $x, x' \in [0,R]$ and $z \in \rho \big(H_{(0,R),\alpha,\beta, }^{(0)}\big)$, 
$\Im(z^{1/2}) \geq 0$. Define the functions
\begin{align}
\begin{split} 
f_R^{(1)}(z,x)&=-u_R(x)\psi_{R,\beta}^{(0)}(z,x),  \quad
f_R^{(2)}(z,x)=\frac{-u_R(x)\psi_{0,\alpha}^{(0)}(z,x)}
{W(\psi_{R,\beta}^{(0)},\psi_{0,\alpha}^{(0)})},      \\
g_R^{(1)}(z,x)&=\frac{v_R(x)\psi_{0,\alpha}^{(0)}(z,x)}
{W(\psi_{R,\beta}^{(0)},\psi_{0,\alpha}^{(0)})},  \quad \;\;\,   
g_R^{(2)}(z,x)=v_R(x)\psi_{R,\beta}^{(0)}(z,x),       \lb{3.39}
\end{split} 
\end{align}
where, for brevity, we have written
\begin{equation}\label{03.12}
W(\psi_{R,\beta}^{(0)},\psi_{0,\alpha}^{(0)})=W(\psi_{R,\beta}^{(0)}(z,\cdot),\psi_{0,\alpha}^{(0)}(z,\cdot)),
\end{equation}
and introduce the Volterra integral kernel
\begin{align}
\begin{split}
H_R(z,x,x^{\p})&=f_R^{(1)}(z,x)g_R^{(1)}(z,x^{\p})-f_R^{(2)}(z,x)g_R^{(2)}(z,x^{\p})    \\
&=u_R(x)\frac{\sin(z^{1/2}(x-x^{\p}))}{z^{1/2}}v_R(x^{\p}).    \lb{3.40}
\end{split}  
\end{align}
In addition, one introduces the functions $\widehat{f}_R^{(j)}(z,\cdot)$, $j=1,2$, as solutions 
of the Volterra integral equations
\begin{align}
\widehat{f}_R^{(1)}(z,x)&=f_R^{(1)}(z,x)-\int_x^Rdx^{\p}H(z,x,x^{\p}) 
\widehat{f}_R^{(1)}(z,x^{\p}),      \lb{3.41} \\
\widehat{f}_R^{(2)}(z,x)&=f_R^{(2)}(z,x)+\int_0^xdx^{\p}H(z,x,x^{\p}) 
\widehat{f}_R^{(2)}(z,x^{\p}).      \lb{3.42}
\end{align}

Comparing \eqref{3.41} to \eqref{3.16} and \eqref{3.42} to \eqref{3.16a}, one infers 
\begin{align}
\widehat{f}_R^{(1)}(z,x)&=-u_R(x)\psi_{R,\beta}(z,x)   \lb{3.43}\\
\widehat{f}_R^{(2)}(z,x)& = \frac{-u_R(x)\psi_{0,\alpha}(z,x)}{W(\psi_{R,\beta}^{(0)}(z,\cdot),\psi_{0,\alpha}^{(0)}(z,\cdot))},  \lb{3.43b}\\
& \qquad \qquad \qquad \qquad \quad \quad x \in [0,R].\no
\end{align}
Introducing 
\begin{equation}
U_R(z,x) = \begin{pmatrix}
1-\int_x^Rdx^{\p}g_R^{(1)}(z,x^{\p})\widehat{f}_R^{(1)}(z,x^{\p}) 
& \int_0^xdx^{\p}g_R^{(1)}(z,x^{\p})\widehat{f}_R^{(2)}(z,x^{\p}) \\
\int_x^Rdx^{\p}g_R^{(2)}(z,x^{\p})\widehat{f}_R^{(1)}(z,x^{\p})  
& 1-\int_0^xdx^{\p}g_R^{(2)}(z,x^{\p})\widehat{f}_R^{(2)}(z,x^{\p})
\end{pmatrix},    \lb{3.44}
\end{equation}
then applying \cite[Theorem 3.2]{GM03}, one obtains 
\begin{align}
{\det}_{L^2((0,R); dx)} \Big(I_{(0,R)}
+ \ol{u_R \big(H_{(0,R),\alpha,\beta}^{(0)}-z\big)^{-1}v_R}\Big)
&= {\det}_{\C^2} (U_R (z,0))\lb{3.45}\\
&={\det}_{\C^2} (U_R (z,R)).\lb{3.45b}
\end{align}
Using \eqref{3.39} and \eqref{3.43}--\eqref{3.44}, one readily verifies 
\begin{align}
{\det}_{\C^2} (U_R (z,0))&= 1+\int_0^Rdx^{\p} \frac{V(x^{\p})\psi_{0,\alpha}^{(0)}(z,x^{\p})\psi_{R,\beta}(z,x^{\p})}{W(\psi_{R,\beta}^{(0)}(z,\cdot),\psi_{0,\alpha}^{(0)}(z,\cdot))}\lb{3.45c}\\
{\det}_{\C^2} (U_R (z,R))&= 1+\int_0^Rdx^{\p} \frac{V(x^{\p})\psi_{R,\beta}^{(0)}(z,x^{\p})\psi_{0,\alpha}(z,x^{\p})}{W(\psi_{R,\beta}^{(0)}(z,\cdot),\psi_{0,\alpha}^{(0)}(z,\cdot))}.\lb{3.45d}
\end{align}

By definition,
\begin{equation}\lb{03.1}
W(\psi_{R,\beta}(z,\cdot),\psi_{0,\alpha}(z,\cdot))=\psi_{R,\beta}(x)\psi_{0,\alpha}^{\p}(x)-\psi_{R,\beta}^{\p}(x)\psi_{0,\alpha}(x), \qquad x\in [0,R].
\end{equation}
Substitution of \eqref{3.16a} and \eqref{3.16} into the r.h.s. of \eqref{03.1} leads one to the following lengthy expression for the constant $W(\psi_{R,\beta},\psi_{0,\alpha})$:
\begin{align}
&W(\psi_{R,\beta}(z,\cdot),\psi_{0,\alpha}(z,\cdot))\nonumber\\
&= \bigg[ \psi_{R,\beta}^{(0)\p}(z,x)-z^{-1/2}\int_x^Rdx^{\p}\sin(z^{1/2}(x-x^{\p}))V(x^{\p})\psi_{R,\beta}(z,x^{\p}) \bigg]\no\\
&\quad \quad \times \bigg[\psi_{0,\alpha}^{(0)\p}(z,x)+\int_0^xdx^{\p} \cos(z^{1/2}(x-x^{\p}))V(x^{\p})\psi_{0,\alpha}(z,x^{\p}) \bigg]\no\\
&\quad -\bigg[ \psi_{R,\beta}^{(0)\p}(z,x)-\int_x^Rdx^{\p}\cos(z^{1/2}(x-x^{\p}))V(x^{\p})\psi_{R,\beta}(z,x^{\p}) \bigg]\no\\
&\qquad \quad \times \bigg[ \psi_{0,\alpha}^{(0)}(z,x)+z^{-1/2}\int_0^xdx^{\p}\sin(z^{1/2}(x-x^{\p}))V(x^{\p})\psi_{0,\alpha}(z,x^{\p}) \bigg],\label{03.2}\\
&\hspace*{9.3cm} x\in [0,R].\no
\end{align}
Evaluation of the r.h.s. of \eqref{03.2} at $x=0$ and $x=R$ yields (after straightforward manipulation)
\begin{align}
W(\psi_{R,\beta}(z,\cdot),\psi_{0,\alpha}(z,\cdot))&=W(\psi_{R,\beta}^{(0)},\psi_{0,\alpha}^{(0)})+\int_0^Rdx^{\p}V(x^{\p})\psi_{0,\alpha}^{(0)}(z,x^{\p})\psi_{R,\beta}(z,x^{\p})\no\\
&=W(\psi_{R,\beta}^{(0)},\psi_{0,\alpha}^{(0)})+\int_0^Rdx^{\p}V(x^{\p})\psi_{R,\beta}^{(0)}(z,x^{\p})\psi_{0,\alpha}(z,x^{\p}),\no
\end{align}
again making use of the abbreviation \eqref{03.12}, and it follows from \eqref{3.45c} and \eqref{3.45d} that
\begin{equation}\lb{03.3}
{\det}_{\C^2} (U_R (z,0))={\det}_{\C^2} (U_R (z,R))=\frac{W(\psi_{R,\beta}(z,\cdot),\psi_{0,\alpha}(z,\cdot))}{W(\psi_{R,\beta}^{(0)}(z,\cdot),\psi_{0,\alpha}^{(0)}(z,\cdot))}.
\end{equation}
The identity \eqref{03.4} then follows from \eqref{03.3} and either \eqref{3.45} or \eqref{3.45b}.
\end{proof}

The corresponding half-line result then reads as follows: 

\begin{theorem} \lb{t3.4}
Assume that $V \in L^1((0,\infty); dx)$ is real-valued and suppose that  
$z \in \rho \big(H_{(0,\infty),\alpha}^{(0)}\big)$. Then
\begin{equation}
\ol{u \big(H_{(0,\infty),\alpha}^{(0)} - z I_{(0,\infty)}\big)^{-1}v} 
\in \cB_1\big(L^2((0,\infty); dx)\big),      \lb{3.46}
\end{equation}
and the determinant 
${\det}_{L^2((0,\infty);dx)}\Big(I_{(0,\infty)} 
+ \ol{u \big(H_{(0,\infty),\alpha}^{(0)}-z I_{(0,\infty)}\big)^{-1}v}\Big)$ is a quotient of Wronskians,
\begin{align}
&{\det}_{L^2((0,\infty;dx)}\Big(I_{(0,\infty)} 
+ \ol{u \big(H_{(0,\infty),\alpha}^{(0)}-z I_{(0,\infty)}\big)^{-1}v}\Big)
=\frac{W(\psi_{+}(z,\cdot),\psi_{0,\alpha}(z,\cdot))}
{W(\psi_{+}^{(0)}(z,\cdot),\psi_{0,\alpha}^{(0)}(z,\cdot))}  \label{03.11} \\
&\quad=\frac{\sin(\alpha)\psi_+^{\p}(z,0) + \cos(\alpha)\psi_+(z,0)}
{\cos(\alpha) + iz^{1/2}\sin(\alpha)}.     \lb{3.47}
\end{align}
\end{theorem}
\begin{proof}
The equality \eqref{3.47} is a trivial consequence of \eqref{3.20} and \eqref{3.23}.  Thus, we proceed to verify \eqref{3.11}.
Let $x, x' \in [0,\infty)$ and $z \in \rho\big(H_{(0,\infty),\alpha}^{(0)}\big)$, 
$\Im(z^{1/2}) \geq 0$. Define the functions
\begin{align}
\begin{split}
f_1(z,x)&=-u(x)\psi_+^{(0)}(z,x),  \qquad \qquad \quad 
f_2(z,x)=\frac{-u(x)\psi_{0,\alpha}^{(0)}(z,x)}
{W(\psi_+^{(0)}(z,\cdot),\psi_{0,\alpha}^{(0)}(z,\cdot))},     \\
g_1(z,x)&=\frac{v(x)\psi_{0,\alpha}^{(0)}(z,x)}{W(\psi_+^{(0)}(z,\cdot),\psi_{0,\alpha}^{(0)}(z,\cdot))},   \quad \;\; 
g_2(z,x)=v(x)\psi_+^{(0)}(z,x),     \lb{3.48}
\end{split} 
\end{align}
and introduce the Volterra integral kernel
\begin{align}
\begin{split} 
H(z,x,x^{\p})&=f_1(z,x)g_1(z,x^{\p})-f_2(z,x)g_2(z,x^{\p})     \\
&=u(x)\frac{\sin(z^{1/2}(x-x^{\p}))}{z^{1/2}}v(x^{\p}).   \lb{3.49}  
\end{split} 
\end{align}
Introducing $\widehat{f}_1$, $\widehat{f}_2$ as the solutions of the Volterra integral 
equations
\begin{align}
\begin{split} 
\widehat{f}_1(z,x)&=f_1(z,x)-\int_x^{\infty}dx^{\p}H(z,x,x^{\p})\widehat{f}_1(z,x^{\p}),    \\
\widehat{f}_2(z,x)&=f_2(z,x)+\int_0^xdx^{\p}H(z,x,x^{\p})\widehat{f}_2(z,x^{\p}),  \lb{3.50}
\end{split} 
\end{align}
a comparison with \eqref{3.21} yields
\begin{equation}\lb{3.51}
\widehat{f}_1(z,x)=-u(x)\psi_+(z,x).
\end{equation}
Assuming temporarily that $V$ has compact support, one concludes 
that $\widehat{f}_1\in L^2((0,\infty);dx)$ and applying \cite[Theorem 3.2]{GM03} results in 
\begin{equation}\lb{3.52}
{\det}_{L^2((0,\infty;dx)} \Big(I_{(0,\infty)}
+ \ol{u \big(H_{(0,\infty),\alpha}^{(0)}-z\big)^{-1}v}\Big)
= {\det}_{\C^2}(U(z,0)),
\end{equation}
where
\begin{equation}
U(z,x) = \begin{pmatrix}
1-\int_x^{\infty}dx^{\p}g_1(z,x^{\p})\widehat{f}_1(z,x^{\p}) & \int_0^xdx^{\p}g_1(z,x^{\p})\widehat{f}_2(z,x^{\p})\\
\int_x^{\infty}dx^{\p}g_2(z,x^{\p})\widehat{f}_1(z,x^{\p}) & 1-\int_0^xdx^{\p}g_2(z,x^{\p})\widehat{f}_2(z,x^{\p})\end{pmatrix}.    \lb{3.53}
\end{equation}

Using \eqref{3.48} and \eqref{3.51}, equations \eqref{3.52} and \eqref{3.53} immediately yield 
\begin{equation}\label{03.5}
{\det}_{\C^2}(U(z,0))=1+\int_0^{\infty}dx^{\p}\frac{V(x^{\p})\psi_{0,\alpha}^{(0)}(z,x^{\p})\psi_+(z,x^{\p})}{W(\psi_+^{(0)}(z,\cdot),\psi_{0,\alpha}^{(0)}(z,\cdot))}.
\end{equation}
By definition,
\begin{equation}\label{03.6}
W(\psi_+(z,\cdot),\psi_{0,\alpha}(z,\cdot))=\psi_+(z,x)\psi_{0,\alpha}^{\p}(z,x)-\psi_+^{\p}(z,x)\psi_{0,\alpha}(z,x), \qquad x\in [0,\infty).
\end{equation}
Substitution of \eqref{3.21} and \eqref{3.22} into the r.h.s. of \eqref{03.6} provides
\begin{align}
&W(\psi_+(z,\cdot),\psi_{0,\alpha}(z,\cdot))\nonumber\\
&= \bigg[\psi_+^{(0)}(z,x)-z^{-1/2}\int_x^{\infty}dx^{\p}\sin(z^{1/2}(x-x^{\p}))V(x^{\p})\psi_+(z,x^{\p}) \bigg] \nonumber\\
&\quad \quad  \times \bigg[\psi_{0,\alpha}^{(0)\p}(z,x)+\int_0^xdx^{\p}\cos(z^{1/2}(x-x^{\p}))V(x^{\p})\psi_{0,\alpha}(z,x^{\p}) \bigg]\nonumber\\
&\quad -\bigg[\psi_+^{(0)\p}(z,x)-\int_x^{\infty}dx^{\p}\cos(z^{1/2}(x-x^{\p}))V(x^{\p})\psi_+(z,x^{\p}) \bigg]\nonumber\\
&\qquad \quad \times\bigg[\psi_{0,\alpha}^{(0)}(z,x)+z^{-1/2}\int_0^{x}dx^{\p}\sin(z^{1/2}(x-x^{\p}))V(x^{\p})\psi_{0,\alpha}(z,x^{\p}) \bigg],\label{03.7}\\
&\hspace*{9.2cm} x\in [0,\infty).\nonumber
\end{align}
Evaluation of the r.h.s. of \eqref{03.7} at $x=0$ yields (after straightforward manipulation)
\begin{equation}\label{03.8}
W(\psi_+(z,\cdot),\psi_{0,\alpha}(z,\cdot))=W(\psi_+^{(0)}(z,\cdot),\psi_{0,\alpha}^{(0)}(z,\cdot))+\int_0^{\infty}dx^{\p}V(x^{\p})\psi_{0,\alpha}^{(0)}(z,x^{\p})\psi_+(z,x^{\p}),
\end{equation}
and it follows from \eqref{03.5} and \eqref{03.8} that
\begin{equation}\label{03.9}
{\det}_{\C^2}(U(z,0))=\frac{W(\psi_+(z,\cdot),\psi_{0,\alpha}(z,\cdot))}{W(\psi_+^{(0)}(z,\cdot),\psi_{0,\alpha}^{(0)}(z,\cdot))}.
\end{equation}
The identities \eqref{3.52} and \eqref{03.9} yield \eqref{03.11} under the additional assumption that $V$ has compact support.  
Next we now show how to remove this additional assumption.

Let $\chi_{\varepsilon}$ denote a smooth cutoff function of the type
\begin{equation}\lb{125}
0\leq \chi\leq 1, \quad \chi(x)=\begin{cases} 
1, & x\in [0,1], \\
0 & x\geq 2, \end{cases} 
\quad \chi_{\varepsilon}(x)=\chi(\varepsilon x), \ \varepsilon>0, 
\end{equation}
and
\begin{equation}
u_{(\varepsilon)}=u\chi_{\varepsilon}, \quad v_{(\varepsilon)}=v\chi_{\varepsilon}, 
\quad V_{(\varepsilon)}=V\chi_{\varepsilon}.
\end{equation}
For convenience, we denote
\begin{equation}
K_{\alpha,(\varepsilon)}^{(0)}(z) 
= \ol{u_{(\varepsilon)} \big(H_{(0,\infty),\alpha}^{(0)}-z I_{(0,\infty)}\big)^{-1}v_{(\varepsilon)}}, \quad 
K_{\alpha}^{(0)}(z)=\ol{u \big(H_{(0,\infty),\alpha}^{(0)} 
- z I_{(0,\infty)}\big)^{-1}v}.  
\end{equation}
The integral kernel $G_{(0,\infty),\alpha}^{(0)}(z,x,x^{\p})$
of $\big(H_{(0,\infty),\alpha}^{(0)}-z I_{(0,\infty)}\big)^{-1}$ can be expressed in terms of $G_{(0,\infty),0}^{(0)}(z,x,x^{\p})$, the integral kernel of  the resolvent of the free half-line Dirichlet  operator, that is, 
$\big(H_{(0,\infty),0}^{(0)}-z I_{(0,\infty)}\big)^{-1}$ as follows (cf., Lemma \ref{l3.2a}),
\begin{equation}\lb{126}
G_{(0,\infty),\alpha}^{(0)}(z,x,x^{\p})
=G_{(0,\infty),0}^{(0)}(z,x,x^{\p})
- \frac{\sin(\alpha)e^{iz^{1/2}x}e^{iz^{1/2}x^{\p}}}{\cos(\alpha)+iz^{1/2}\sin(\alpha)}, 
\quad x, x' \geq 0.
\end{equation}
Therefore,
\begin{align}
K_{\alpha}^{(0)}(z)&=\ol{u \big(H_{(0,\infty),\alpha}^{(0)}
- z I_{(0,\infty)}\big)^{-1}v}  \no \\
&=\ol{u \big(H_{(0,\infty),0}^{(0)}-z I_{(0,\infty)}\big)^{-1}v}    \no \\ 
& \quad - \frac{\sin(\alpha)}{\cos(\alpha)+iz^{1/2}\sin(\alpha)} 
\big(v \ol{e^{iz^{1/2}\cdot}},\cdot\big)_{L^2((0,\infty);dx)} 
ue^{iz^{1/2}\cdot},  
\end{align}
and 
\begin{align}
K_{\alpha,(\varepsilon)}^{(0)}(z)&
=\ol{u_{(\varepsilon)} \big(H_{(0,\infty),\alpha}^{(0)} 
-z I_{(0,\infty)}\big)^{-1}v_{(\varepsilon)}}  \no \\
&=\ol{u_{(\varepsilon)} \big(H_{(0,\infty),0}^{(0)}-z I_{(0,\infty)}\big)^{-1} 
v_{(\varepsilon)}}   \no \\
& \quad -\frac{\sin(\alpha)}{\cos(\alpha)+iz^{1/2}\sin(\alpha)} \big(v_{(\varepsilon)}\ol{e^{iz^{1/2}\cdot}},\cdot\big)_{L^2((0,\infty);dx)} 
u_{(\varepsilon)}e^{iz^{1/2}\cdot}.   
\end{align}
Thus,
\begin{align}
& \big\|K_{\alpha}^{(0)}(z)-K_{\alpha,(\varepsilon)}^{(0)}(z)\big\|_{\cB_1(L^2((0,\infty);dx))}  
\no \\
&\quad \leq \|K_0^{(0)}(z)-K_{0,\varepsilon}^{(0)}\|_{\cB_1(L^2((0,\infty);dx))}  
\no \\
& \qquad +|\eta(\alpha)|\Big\|\big(v \ol{e^{iz^{1/2}\cdot}},\cdot\big)_{L^2((0,\infty);dx)} 
ue^{iz^{1/2}\cdot}     \no \\
& \hspace*{2.2cm} 
- (v_{(\varepsilon)}\ol{e^{iz^{1/2}\cdot}},\cdot)_{L^2((0,\infty);dx)} 
u_{(\varepsilon)}e^{iz^{1/2}\cdot} \Big\|_{\cB_1(L^2((0,\infty);dx))}  \no \\
& \quad =\Big\|\ol{u \big(H_{(0,\infty),0}^{(0)}-z I_{(0,\infty)}\big)^{-1}v}
- \ol{u_{(\varepsilon)} \big(H_{(0,\infty),0}^{(0)}-z I_{(0,\infty)}\big)^{-1}v}  
\no \\
&\qquad \quad +\ol{u_{(\varepsilon)} 
\big(H_{(0,\infty),0}^{(0)}-z I_{(0,\infty)}\big)^{-1}v}   \no \\
&\qquad \quad - \ol{u_{(\varepsilon)} \big(H_{(0,\infty),0}^{(0)}
-zI_{(0,\infty)}\big)^{-1}v_{(\varepsilon)}}
\Big\|_{\cB_1(L^2((0,\infty);dx))}  \no \\
& \qquad + |\eta(\alpha)|\Big\|\big(v \ol{e^{iz^{1/2}\cdot}},\cdot\big)_{L^2((0,\infty);dx)} 
ue^{iz^{1/2}\cdot}     \no \\
& \hspace*{2.2cm}
- \big(v_{(\varepsilon)}\ol{e^{iz^{1/2}\cdot}},\cdot\big)_{L^2((0,\infty);dx)} 
u_{(\varepsilon)}e^{iz^{1/2}\cdot} 
\Big\|_{\cB_1(L^2((0,\infty);dx))}  \no \\
& \quad \leq C(z,\alpha)\big[\|u-u_{(\varepsilon)}\|_{L^2((0,\infty);dx)}
+ \|v-v_{(\varepsilon)}\|_{L^2((0,\infty);dx)} \big]  \no \\
&\quad =2C(z,\alpha)\|u-u_{(\varepsilon)}\|_{L^2((0,\infty);dx)}\lb{127},
\end{align}
where $C(z,\alpha)$ is an appropriate constant, and we abbreviated 
\begin{equation}
\eta(\alpha)=\frac{\sin(\alpha)}{\cos(\alpha)+iz^{1/2}\sin(\alpha)}.
\end{equation}
Since \eqref{127} converges to zero as $\varepsilon \downarrow 0$, one infers
\begin{equation}\lb{128}
\lim_{\varepsilon\downarrow 0}{\det}_{L^2((0,\infty);dx)}
\big(I_{(0,\infty)} + K_{\alpha,(\varepsilon)}^{(0)}(z)\big) 
= {\det}_{L^2((0,\infty);dx)} \big(I_{(0,\infty)} + K_{\alpha}^{(0)}(z)\big).
\end{equation}
Since $V_{(\varepsilon)}$ has compact support, one has 
\begin{equation}\lb{129}
{\det}_{L^2((0,\infty);dx)}
\big(I_{(0,\infty)} + K_{\alpha,(\varepsilon)}^{(0)}(z)\big)
= \frac{\cos(\alpha)\psi_{+,(\varepsilon)}(z,0)
+ \sin(\alpha)\psi_{+,(\varepsilon)}^{\p}(z,0)}{\cos(\alpha)+i z^{1/2} \sin(\alpha)},
\end{equation}
where $\psi_{+,(\varepsilon)}$ denotes the Jost solution corresponding to the 
potential $V_{(\varepsilon)}$.  Therefore,
\begin{equation}\lb{130}
{\det}_{L^2((0,\infty);dx)}\big(I_{(0,\infty)} + K_{\alpha}^{(0)}(z)\big) = 
\lim_{\varepsilon \downarrow 0}\frac{\cos(\alpha)
\psi_{+,(\varepsilon)}(z,0)+\sin(\alpha)
\psi_{+,(\varepsilon)}^{\p}(z,0)}{\cos(\alpha)+i z^{1/2} \sin(\alpha)}. 
\end{equation}
Iterating the Volterra integral equations \eqref{3.21} and \eqref{3.22}, with $V$ replaced 
by $V_{(\varepsilon)}$, and $\psi_+$ replaced by $\psi_{+,(\varepsilon)}$, using the estimates,
\begin{align}
& |\psi_+(z,x)| + |\psi_{+,(\varepsilon)}(z,x)| \leq Ce^{-\Im(z^{1/2})x}, 
\quad \Im(z^{1/2})\geq0, \; x\geq 0, \\
& |V_{(\varepsilon)}(x)| \leq |V(x)| \, \text{ for a.e.\ $x\geq 0$,} 
\end{align}
together with $V\in L^1((0,\infty);dx)$ and Lebesgue's dominated convergence theorem 
readily yield 
\begin{equation}\lb{131}
\lim_{\varepsilon\downarrow 0}\psi_{+,(\varepsilon)}(z,x)
=\psi_{+}(z,x), \quad \lim_{\varepsilon\downarrow0} 
\psi_{+,(\varepsilon)}^{\p}(z,x) = \psi_+^{\p}(z,x),  \quad x\geq 0.
\end{equation}
Thus, taking the limit $\varepsilon \downarrow 0$ on the right-hand side of \eqref{130} 
proves \eqref{3.47}.
\end{proof}

In the special Dirichlet case $\alpha=0$, \eqref{3.47} is due to Jost and Pais \cite{JP51} 
(see also \cite{BF60}).

\begin{remark} \lb{r3.7} 
Due to our hypotheses \eqref{3.1} and \eqref{3.8} on $V$ one can view 
$H_{(0,R),\alpha,\beta}$ as a quadratic form sum of 
$H_{(0,R),\alpha,\beta}^{(0)}$ and $V$ in $L^2((0,R);dx)$, 
\begin{equation}
H_{(0,R),\alpha,\beta} = H_{(0,R),\alpha,\beta}^{(0)} +_q V,
\end{equation}
and similarly, one can view $H_{(0,\infty),\alpha}$ as a quadratic form sum of 
$H_{(0,\infty),\alpha}^{(0)}$ and $V$ in $L^2((0,\infty);dx)$, 
\begin{equation}
H_{(0,\infty),\alpha} = H_{(0,\infty),\alpha}^{(0)} +_q V.
\end{equation}
This represents a special case of the perturbation theory approach used also in 
Appendix \ref{sA} (cf.\ \eqref{2.6}) and will play a role in our principal Section \ref{s5}. 
More precisely, the corresponding closed, symmetric, and densely defined 
sequilinear form, denoted by $Q_{H_{(0,R),\alpha,\beta}}$, associated with $H_{(0,R),\alpha,\beta}$ (cf.\ \cite[p.\ 312, 321, 327--328]{Ka80}), reads as follows:
\begin{align}
& Q_{H_{(0,R),\alpha,\beta}}(f,g) = \int_0^R dx \big[\ol{f'(x)} g'(x) + V(x) \ol{f(x)} g(x)\big]   \no \\
& \hspace*{2.65cm} - \cot(\alpha) \ol{f(0)} g(0) -  \cot(\beta) \ol{f(R)} g(R), \lb{2.2f} \\
& f, g \in \dom(Q_{H_{(0,R),\alpha,\beta}}) = H^1((0,R))  \no \\
& \quad = \big\{h\in L^2((0,R); dx) \,|\,
h \in AC ([0,R]); \, h' \in L^2((0,R); dx)\big\}, \quad \alpha, \beta \in (0,\pi),   \no \\
& Q_{H_{(0,R),0,\beta}}(f,g) = \int_0^R dx \big[\ol{f'(x)} g'(x) + V(x) \ol{f(x)} g(x)\big]
-  \cot(\beta) \ol{f(R)} g(R),  \lb{2.2g} \\
&  f, g \in \dom(Q_{H_{(0,R),0,\beta}})    \no \\
& \quad = \big\{h\in L^2((0,R); dx) \,|\,
h \in AC ([0,R]); \, h(0) =0; \, h' \in L^2((0,R); dx)\big\},  \no \\
& \hspace*{9.5cm}  \beta \in (0,\pi),   \no \\
& Q_{H_{(0,R),\alpha,0}}(f,g) = \int_0^R dx \big[\ol{f'(x)} g'(x) + V(x) \ol{f(x)} g(x)\big]
-  \cot(\alpha) \ol{f(0)} g(0),  \lb{2.2h} \\
&  f, g \in \dom(Q_{H_{(0,R),\alpha,0}})    \no \\
& \quad = \big\{h\in L^2((0,R); dx) \,|\,
h \in AC ([0,R]); \, h(R) =0; \, h' \in L^2((0,R); dx)\big\},   \no \\
& \hspace*{9.6cm}  \alpha \in (0,\pi),   \no \\
& Q_{H_{(0,R),0,0}}(f,g) = \int_0^R dx \big[\ol{f'(x)} g'(x) + V(x) \ol{f(x)} g(x)\big],  
\lb{2.2i} \\
& f, g \in \dom(Q_{H_{(0,R),0,0}}) = \dom\big(|H_{0,0}|^{1/2}\big) = H^1_0((0,R))  \no \\
& \quad = \big\{h\in L^2((0,R); dx) \,|\,
h \in AC([0,R]); \, h(0)=0, \, h(R) =0;    \no \\
& \hspace*{6.6cm}   h' \in L^2((0,R); dx)\big\}.   \no
\end{align}

Equations \eqref{2.2f}--\eqref{2.2i} follow from the fact that for any $\varepsilon>0$,
there exists $\eta(\varepsilon)>0$ such that for all $h \in H^1((0,R))$,
\begin{align}
& |h(x_0)| \leq \varepsilon \|h' \|_{L^2((0,R); dx)} +
\eta(\varepsilon) \|h \|_{L^2((0,R); dx)}, \quad x_0 \in [0,R],    \lb{2.123} \\
& \||V|^{1/2} h\|_{L^2((0,R); dx)} \leq \varepsilon \|h' \|_{L^2((0,R); dx)} +
\eta(\varepsilon) \|h \|_{L^2((0,R); dx)}    \lb{2.124}
\end{align}
(cf.\ \cite[p.\ 193, 345--346]{Ka80}). In particular, closely examining the inequalities 
in \cite[p.\ 193, 345--346]{Ka80} shows that $\eta(\varepsilon)$ can be chosen to be $2/\varepsilon$ and hence independently of $R>0$. Thus, 
for fixed $\alpha, \beta \in [0,\pi)$, $H_{(0,R),\alpha,\beta}$ are bounded from below uniformly with respect to 
$R>0$,
\begin{equation}
H_{(0,R),\alpha,\beta} \geq c_{\alpha,\beta} I_{(0,R)}    \lb{3.108}
\end{equation}
for some $c_{\alpha,\beta}\in\bbR$ independent of $R>0$, a fact to be re-examined 
in Lemma \ref{l4.6a} and used in Section \ref{s5}. 
\end{remark}

\begin{remark} \lb{r3.8}
We note in passing that in the special case of Dirichlet boundary conditions 
$\alpha = 0$ at $x=0$ in \eqref{3.3} and \eqref{3.5}, it is well-known (and discussed in great detail, e.g., in \cite[Sects.\ I.1--I.4]{CS89}) that one can actually accommodate fairly strongly singular potentials at $x=0$. More precisely, one can permit a.e.\ real-valued (singular)  potentials of the type 
\begin{equation}
V \in L^1((0,R); x dx) \, \text{ respectively, } \, 
V \in L^1((0,\infty); [x/(1+x)] dx),     \lb{3.100}
\end{equation}   
and all results in this paper extend to potentials satisfying \eqref{3.100} as long as the Dirichlet boundary condition $\alpha = 0$ is chosen at the left endpoint $x=0$. 
\end{remark}

\section{Some Basic Convergence Results}  \lb{s4}

In this section we prove some basic convergence results involving the 
Schr\"odinger operators $H_{(0,\infty),0,0}$ and $H_{(0,\infty),0}$ with Dirichlet boundary conditions. The case of general separated boundary conditions will 
subsequently be discussed with the help of Lemma \ref{l3.1a} and \ref{l3.2a}.

We start by decomposing $L^2((0,\infty); dx)$ into
\begin{equation}
L^2((0,\infty); dx) = L^2((0,R); dx) \oplus L^2((R,\infty); dx),   \lb{4.1} 
\end{equation}
and all direct operator sums $A = B \oplus C$ from now on are viewed with 
respect to the decomposition \eqref{4.1}.

\begin{lemma}\lb{l4.1} Let $z\in \C\backslash  [0,\infty)$. Then 
\begin{equation}
\slim_{R\rightarrow \infty} \big(\big[H_{(0,R),0,0}^{(0)}\oplus 0\big]-z I_{(0,\infty)}\big)^{-1} 
= \big(H_{(0,\infty),0}^{(0)}-z I_{(0,\infty)}\big)^{-1}.    \lb{4.2}
\end{equation}
and 
\begin{align}
\begin{split}
\lim_{R\rightarrow \infty} \Big[u_R\big(H_{(0,R),0,0}^{(0)}-zI_{(0,R)}\big)^{-1}\oplus 0\Big]=u \big(H_{(0,\infty),0}^{(0)}-z I_{(0,\infty)}\big)^{-1}&   \\  
\text{ in $\cB_2\big(L^2((0,\infty);dx)\big)$.}&    \lb{4.3}
\end{split}
\end{align}
as well as 
\begin{align}
\begin{split}
\lim_{R\rightarrow \infty} \Big[\ol{\big(H_{(0,R),0,0}^{(0)}-zI_{(0,R)}\big)^{-1}v_R} 
\oplus 0\Big] = \ol{\big(H_{(0,\infty),0}^{(0)}-z I_{(0,\infty)}\big)^{-1}v}&   \\  
\text{ in $\cB_2\big(L^2((0,\infty);dx)\big)$.}&    \lb{4.3a}
\end{split}
\end{align}
\end{lemma}
\begin{proof}
Without loss of generality, it suffices to prove \eqref{4.2} and \eqref{4.3} for $z<0$. 

Since weak resolvent convergence in the context of self-adjoint operators is equivalent to strong resolvent convergence, in order to prove \eqref{4.2} it suffices to prove that 
\begin{align}\lb{39}
& \lim_{R\rightarrow \infty}\big(f,\big(\big[H_{(0,R),0,0}^{(0)}\oplus 0\big]
- z I_{(0,\infty)}\big)^{-1}g\big)_{L^2((0,\infty);dx)}   \no \\
& \quad 
= \big(f, \big(H_{(0,\infty),0}^{(0)}-z I_{(0,\infty)}\big)^{-1}g\big)_{L^2((0,\infty);dx)},    \quad  f,g\in L^2((0,\infty);dx)). 
\end{align}
In addition, it suffices to verify \eqref{39}) for all $f$ and $g$ belonging to a dense subset (see \cite[Theorem\ 4.26(b)]{We80}), which we choose to be 
\begin{equation}\lb{40}
L^1((0,\infty);dx)\cap L^2((0,\infty);dx).
\end{equation}
From now on let $f$ and $g$ belong to the set in \eqref{40}. Then 
\begin{align}
&\Big|\big( f,\big[\big(H_{(0,R),0,0}^{(0)}-z I_{(0,R)}\big)^{-1}\oplus 
(-z)^{-1}I_{(R,\infty)}\big]g
\big)_{L^2((0,\infty);dx)}   \no \\
& \; -\big(f, \big(H_{(0,\infty),0}^{(0)}-z I_{(0,\infty)}\big)^{-1}g 
\big)_{L^2((0,\infty);dx)}\Big|    \no \\
& \quad \leq  
\int_0^{\infty}\int_0^{\infty} dx^{\prime}dx \, 
\Big|\ol{f(x)}\big[\chi_{(0,R)}(x) \chi_{(0,R)}(x^{\prime}) 
G_{(0,R),0,0}^{(0)}(z,x,x^{\prime})     \no \\
& \hspace*{3.2cm} -G_{(0,\infty),0}^{(0)}(z,x,x^{\prime})\big]
g(x^{\prime})\Big|  
 +\bigg|\frac{1}{z}\int_R^{\infty} dx \, \ol{f(x)}g(x)\bigg|,   \lb{41}
\end{align}
where $G_{(0,R),0,0}^{(0)}(z,\cdot,\cdot)$ and 
$G_{(0,\infty),0}^{(0)}(z,\cdot,\cdot)$ are the integral kernels defined in \eqref{3.6} and \eqref{3.13}, respectively.

It is clear that the second term on the right-hand side  of \eqref{41} converges to zero as $R\rightarrow \infty$.  

In view of the identity
 \begin{equation}\lb{45}
\psi_{R,0}^{(0)}(z,x) = \frac{\sin(z^{1/2}(R-x))}{\sin(z^{1/2}R)}=e^{iz^{1/2}x}
 +2i\frac{e^{iz^{1/2}2R}}{e^{iz^{1/2}2R}-1}\sin(z^{1/2}x), 
 \quad \Im(z^{1/2})\geq0,
 \end{equation}
and noting that the second term on the right-hand side of \eqref{45} converges to zero for all $x>0$ since by hypothesis $\Im(z^{1/2})>0$,  it is clear from 
\eqref{3.6} that the integrand in 
\begin{equation}\lb{42}
 \int_0^{\infty}\int_0^{\infty} dx^{\prime}dx \, \Big|\ol{f(x)}\big[\chi_{(0,R)}(x)
 \chi_{(0,R)}(x^{\prime})G_{(0,\infty),0,0}^{(0)}(z,x,x^{\prime})
 - G_{(0,\infty),0}^{(0)}(z,x,x^{\prime}) \big]g(x^{\prime})\Big|
 \end{equation}
converges pointwise to zero as $R\rightarrow \infty$. 

Due to the explicit structure of 
$G_{(0,R),0,0}^{(0)}(z,\cdot,\cdot)$ and $G_{(0,\infty),0}^{(0)}(z,\cdot,\cdot)$
(cf.\ \eqref{3.6}, \eqref{3.13}, \eqref{3.23a}, and \eqref{3.23b}), there is a constant $C(z)>0$, depending only on $z$, such that 
 \begin{equation}\lb{43}
 \big|\chi_{(0,R)}(x)\chi_{(0,R)}(x^{\prime})G_{(0,\infty),0,0}^{(0)}(z,x,x^{\prime})
 - G_{(0,\infty),0}^{(0)}(z,x,x^{\prime}) \big|\leq C(z)
 \end{equation}
 uniformly for a.e.\ $x,x^{\prime}\in (0,\infty)$ and all $R>0$.  
 
 Thus, the integrand in \eqref{42} is majorized by the integrable function
 \begin{equation}\lb{44}
C(z)\big| f(x)g(x^{\prime})\big|.
 \end{equation}
for $R$ sufficiently large. (One recalls we asssumed $f,g\in L^1((0,\infty);dx)$). Consequently, convergence of \eqref{42} to zero as $R\rightarrow \infty$ follows from an application of Lebesgue's dominated convergence theorem, implying 
\eqref{4.2}.

To prove \eqref{4.3}, we first note that by \eqref{3.36} and \eqref{3.46}, 
\begin{align}
u_R\big(H_{(0,R),\alpha,\beta, }^{(0)}-z I_{(0,R)}\big)^{-1/2} &\in 
\cB_2\big(L^2((0,R); dx)\big),    \\
u \big(H_{(0,\infty),\alpha}^{(0)} - z I_{(0,\infty)}\big)^{-1/2} 
&\in \cB_2\big(L^2((0,\infty); dx)\big),
\end{align} 
in particular,
\begin{align}
u_R\big(H_{(0,R),\alpha,\beta, }^{(0)}-z I_{(0,R)}\big)^{-1} &\in 
\cB_2\big(L^2((0,R); dx)\big),    \lb{4.13} \\
u \big(H_{(0,\infty),\alpha}^{(0)} - z I_{(0,\infty)}\big)^{-1} 
&\in \cB_2\big(L^2((0,\infty); dx)\big).   \lb{4.14}
\end{align} 
Thus, it suffices to verify that 
\begin{align}\lb{46}
\begin{split}
& \lim_{R\rightarrow \infty} \int_{(0,\infty)} \int_{(0,\infty)} dx dx' \, 
\big| u(x)\chi_{(0,R)}(x)\chi_{(0,R)}(x^{\prime})G_{(0,R),0,0}^{(0)}(z,x,x^{\prime})
\\
& \hspace*{3.9cm}
-u(x)G_{(0,\infty),0}^{(0)}(z,x,x^{\prime})\big|^2 = 0.
\end{split}
\end{align}
As in the case of \eqref{4.2}, the integrand in \eqref{46} converges pointwise to 
zero as $R\rightarrow \infty$.

Employing the domain monotonicity \eqref{3.33a} for Dirichlet Green's functions one obtains the desired $R$-independent integrable majorant (cf.\ \eqref{4.14})
\begin{align}
\begin{split} 
& \big| u(x)\chi_{(0,R)}(x)\chi_{(0,R)}(x^{\prime})G_{(0,R),0,0}^{(0)}(z,x,x^{\prime})
- u(x)G_{(0,\infty),0}^{(0)}(z,x,x^{\prime})\big|^2    \\
& \quad \leq 2 |u(x)|^2 \big|G_{(0,\infty),0}^{(0)}(z,x,x^{\prime})\big|^2
\quad \text{for a.e.\ $x, x' \in (0,\infty)$.} 
\end{split} 
\end{align}
An application of Lebesgue's dominated convergence theorem then completes 
the proof of \eqref{4.3}. 

Relation \eqref{4.3a} follows after taking adjoints in \eqref{4.3} (replacing 
$u_R, u$ by $v_R, v$, resp.). 
\end{proof}

Next, recalling our notation $u_R, v_R, V_R$ in \eqref{3.35} and the trace class properties proved in \eqref{3.36} and \eqref{3.46}, we now state the following result:

\begin{lemma}\lb{l4.2}
Assume \eqref{3.8} and let $z\in \C\backslash  [0,\infty)$. Then  
\begin{align}\lb{11}
& \lim_{R\to \infty}\Big\|\Big[\ol{u_R \big(H_{(0,R),0,0}^{(0)} 
-z I_{(0,R)}\big)^{-1} v_R} \oplus 0\Big] 
\no \\
& \qquad \;\;\, -  
\ol{u \big(H_{(0,\infty),0}^{(0)}-z I_{(0,\infty)}\big)^{-1}v}\Big\|_{\cB_1(L^2((0,\infty);dx))} = 0
\end{align}
and 
\begin{align}
\begin{split} 
& \lim_{R\to\infty}{\det}_{L^2((0,R);dx)} \Big(I_{(0,R)} 
+ \ol{u_R \big(H_{(0,R),0,0} - z I_{(0,R)}\big)^{-1} v_R}\Big)     \\
& \quad = {\det}_{L^2((0,\infty);dx)} \Big(I_{(0,\infty)} 
+ \ol{u \big(H_{(0,\infty),0} - z I_{(0,\infty)}\big)^{-1} v}\Big).   \lb{12}
\end{split} 
\end{align}
\end{lemma}
\begin{proof}
It suffices to verify \eqref{11} for real energies $z<0$. Convergence for 
$z\in \C\backslash \R$ then follows from convergence for real energies via 
the first resolvent identity.  Therefore, let $z = - E <0$ be fixed for the remainder 
of this proof.  In order to establish \eqref{11} it will be sufficient to show that 
\begin{align}
& \lim_{R\to\infty} \Big\|\Big[u_R  
\big(H_{(0,R),0,0}^{(0)} + E I_{(0,R)}\big)^{-1/2}
\oplus 0\Big]   \no \\
& \qquad \quad - u\big(H_{(0,\infty),0}^{(0)} 
+ E I_{(0,\infty)}\big)^{-1/2}\Big\|_{\cB_2(L^2((0,\infty); dx))}^2    
\no \\
& \quad = \lim_{R\to\infty} \int_0^\infty \int_0^\infty dx dx' \, 
\big|\cK_R(-E,x,x')-\cK(-E,x,x')\big|^2 = 0,     \lb{13}
\end{align}
where $\cK_R(-E,\cdot,\cdot)$ and $\cK(-E,\cdot,\cdot)$ denote the integral 
kernels for the Hilbert--Schmidt 
operators $u_R \big[\big(H_{(0,R),0,0}^{(0)} + E I_{(0,R)}\big)^{-1/2} 
\oplus 0\big]$ and 
$u \big(H_{(0,\infty),0}^{(0)} + E I_{(0,\infty)}\big)^{-1/2}$, respectively.  Specifically, 
\begin{align}
\cK_R(-E,x,x^{\p})&=\begin{cases}
u_R(x)R_{(0,R),0,0}^{(0), 1/2} (-E,x,x^{\p}) & \text{for a.e.\ $x,x^{\p}\in (0,R)$,} \\
0, & \text{for a.e.\ $x\geq R$ or $x'\geq R$}, \end{cases}     \\
\cK(-E,x,x^{\p})&=u(x) R_{(0,\infty),0}^{(0), 1/2}(-E,x,x^{\p}) \;\;  
\text{ for a.e.\ $x, x' \in (0,\infty)$.}   
\end{align}
An application of \eqref{3.25} and \eqref{3.33} then yields the $R$-independent majorant
\begin{align}
& |\cK_R(-E,x,x^{\p})-\cK(-E,x,x^{\p})|^2  \leq  4 |\cK(-E,x,x^{\p})|^2     \no\\
& \quad \leq 4|u(x)|^2 \big|R_{(0,R),0,0}^{(0), 1/2}(-E,x,x^{\p}))\big|^2\no\\
& \quad \leq 4 \pi^{-2} |u(x)|^2 \big[|K_0(E^{1/2}|x-x^{\p}|)|^2  
+ |K_0(E^{1/2}(x+x^{\p}))|^2\big]      \lb{14}
\end{align}
for a.e.\ $x,x' \in (0,\infty)$. 
The zeroth order irregular modified Bessel function $K_0(\cdot)$ satisfies the estimate 
\begin{equation}
0 \leq K_0(x)\leq C \, \ln\bigg(\frac{ex}{1+x}\bigg) \frac{e^{-x}}{2\pi x^{1/2}+1}, 
\quad  x>0,    \lb{14a}
\end{equation}
 for a suitable constant $C>0$ (cf.\ \cite[Sect.\ 9.6]{AS72} for the proper asymptotic 
 relations as $x\downarrow 0$ and $x \to\infty$, implying \eqref{14a}).
Thus, the right-hand side of \eqref{14} represents an $R$-independent integrable majorant of $|\cK_R(-E,\cdot,\cdot)-\cK(-E,\cdot,\cdot)|^2$ a.e.\ on $(0,\infty)^2$.

Applying Lebesgue's dominated convergence theorem to \eqref{3.28}, employing the integrable majorant \eqref{3.32} (equivalently, applying the domain monotonicity \eqref{3.33a}) one infers for a.e.\ $x,x^{\p} \in (0,\infty)$ that 
\begin{align}
\lim_{R\rightarrow \infty}\cK_R(z,x,x^{\p}) 
& = \lim_{\substack{R\rightarrow \infty \\ R\geq \max\{x,x^{\p}\}}}
u_R(x) R_{(0,R),0,0}^{(0),1/2}(z,x,x^{\p})\no\\
&=u(x)  R_{(0,\infty),0}^{(0),1/2}(z,x,x^{\p})   \no \\
&=\cK(z,x,x^{\p}),   \lb{17}
\end{align}
that is, $\cK_R(z,\cdot,\cdot)$ converges pointwise almost everywhere on 
$(0,\infty)^2$ to $\cK(z,\cdot,\cdot)$.

An application of Lebesgue's dominated convergence theorem to \eqref{13}, using the $R$-independent integrable majorant in \eqref{14} then completes the proof of \eqref{11}.

In view of the identities,
\begin{align}
&{\det}_{L^2((0,\infty);dx)}\Big(I_{(0,\infty)}
+ \ol{u \big[\big(H_{(0,R),0,0}^{(0)}-z I_{(0,R)}\big)^{-1}\oplus 0\big]v}\Big)     \no \\
&\quad={\det}_{L^2((0,\infty);dx)}\Big(I_{(0,R)}\oplus I_{(R,\infty)}
+ \ol{u_R\big(H_{(0,R),0,0}^{(0)} - z I_{(0,R)}\big)^{-1}v_R}\oplus 0\Big)    \no \\
&\quad = {\det}_{L^2((0,R);dx)}\Big(I_{(0,R)}
+ \ol{u_R\big(H_{(0,R),0,0}^{(0)} - z I_{(0,R)}\big)^{-1}v_R}\Big).
\end{align}
and the fact that the mapping 
\begin{equation}
\begin{cases} \cB_1(\cH) \to \bbC, \\
A \mapsto {\det}_{\cH} (I_{\cH} + A), 
\end{cases}
\end{equation}
is continuous in $A$ with respect to the trace norm on $\cB_1(\cH)$, \eqref{12} follows from \eqref{11}. 
\end{proof}

\begin{lemma}\lb{l4.3}
Assume \eqref{3.8} and let $z\in \C\backslash  [0,\infty)$. Then 
\begin{equation}
\slim_{R\rightarrow \infty} \big([H_{(0,R),0,0} \oplus 0]-z I_{(0,\infty)}\big)^{-1}
= (H_{(0,\infty),0}-z I_{(0,\infty)})^{-1}.     \lb{4.27}
\end{equation} 
\end{lemma}
\begin{proof}
Since 
\begin{align}
& \begin{pmatrix} H_{(0,R),0,0} - z I_{(0,R)} & 0 \\ 0 & -z I_{(R,\infty)} 
\end{pmatrix}^{-1} 
= \begin{pmatrix} (H_{(0,R),0,0} - z I_{(0,R)})^{-1} & 0 \\ 0 & - z^{-1} I_{(R,\infty)} 
\end{pmatrix}   \no \\
& \quad = \big[(H_{(0,R),0,0} - z I_{(0,R)})^{-1} \oplus 0\big] 
- \big[0 \oplus z^{-1} I_{(R,\infty)}\big]
\end{align}
with respect to the decomposition \eqref{4.1}, and since 
$\slim_{R\to\infty} \big[0\oplus z^{-1}I_{(R,\infty)}\big] = 0$, it suffices to show that 
\begin{equation}\lb{62}
\slim_{R\to\infty}\big[(H_{(0,R),0,0}-z I_{(0,R)})^{-1}\oplus 0\big] 
= (H_{(0,\infty),0}-z I_{(0,\infty)})^{-1}. 
\end{equation}
  To verify \eqref{62}, we use Lemma \ref{l4.1}\,$(i), (ii)$ and \eqref{11} together with the resolvent identities
\begin{align}
& (H_{(0,R),0,0}-z I_{(0,R)})^{-1}
= \big(H_{(0,R),0,0}^{(0)}-z I_{(0,R)}\big)^{-1}    \no \\
& \quad - \ol{\big(H_{(0,R),0,0}^{(0)}-z I_{(0,R)}\big)^{-1}v_R}\Big[I_{(0,R)}
+\ol{u_R \big(H_{(0,R),0,0}^{(0)}-z I_{(0,R)}\big)^{-1}v_R} \Big]^{-1}   \no \\
& \qquad \times u_R \big(H_{(0,R),0,0}^{(0)}-z I_{(0,R)}\big)^{-1},    \lb{63} \\
& (H_{(0,\infty),0}-z I_{(0,\infty)})^{-1}
= \big(H_{(0,\infty),0}^{(0)}-z I_{(0,\infty)}\big)^{-1}    \no \\
& \quad -\ol{\big(H_{(0,\infty),0}^{(0)}-z I_{(0,\infty)}\big)^{-1}v}\bigg[I_{(0,\infty)}
+\ol{u \big(H_{(0,\infty),0}^{(0)}-z I_{(0,\infty)}\big)^{-1}v} \bigg]^{-1}   \no \\
& \qquad \times u \big(H_{(0,\infty),0}^{(0)}-z I_{(0,\infty)}\big)^{-1}.     \lb{60}
\end{align}
Using
\begin{align}
&\Big[I_{(0,R)}+\ol{u_R \big(H_{(0,R),0,0}^{(0)}-z I_{(0,R)}\big)^{-1}v_R}
\Big]^{-1}\oplus 0   \no\\
&\quad = \Big[I_{(0,\infty)}
+ \Big[\ol{u_R \big(H_{(0,R),0,0}^{(0)}-z I_{(0,R)}\big)^{-1}v_R}
\oplus 0\Big]\Big]^{-1}-\big(0\oplus I_{(R,\infty)}\big)\no
\end{align}
together with \eqref{63} yields
\begin{align}
&(H_{(0,R),0,0}-z I_{(0,R)})^{-1}\oplus 0 
= \big[\big(H_{(0,R),0,0}^{(0)}-z I_{(0,R)}\big)^{-1}\oplus 0\big]   \no \\
& \quad -\Big(\ol{(H_{(0,R),0,0}^{(0)}-z I_{(0,R)})^{-1}v_R}\oplus0\Big) \no\\
&\qquad \times \bigg(\Big[I_{(0,\infty)}
+\big(\ol{u_R(H_{(0,R),0,0}^{(0)}-z I_{(0,R)})^{-1}v_R} 
\oplus 0\big) \Big]^{-1}-\big(0\oplus I_{(R,\infty)} \big)\bigg)    \no\\
& \qquad \times \Big( u_R \big(H_{(0,R),0,0}^{(0)}-z I_{(0,R)}\big)^{-1}\oplus 0\Big). 
\lb{63b}
\end{align}
Based on the following convergence results,
\begin{align}
& \nlim_{R\rightarrow \infty}\Big[I_{(0,\infty)}  
+\ol{u_R(H_{(0,R),0,0}^{(0)}-z I_{(0,R)})^{-1}v_R}\oplus 0 \Big]^{-1}  \no \\
& \quad = \Big[I_{(0,\infty)}
+\ol{u \big(H_{(0,\infty),0}^{(0)} -z I_{(0,\infty)}\big)^{-1}v}\Big]^{-1},  \lb{64}\\
& \slim_{R\rightarrow \infty}\big(0\oplus I_{(R,\infty)}\big) = 0,   \lb{64d}\\
& \slim_{R\rightarrow \infty}
\Big(\big(H_{(0,R),0,0}^{(0)}-z I_{(0,R)}\big)^{-1}\oplus0\Big) 
= \big(H_{(0,\infty),0}^{(0)}-z I_{(0,\infty)}\big)^{-1},   \lb{64b}\\
& \lim_{R\rightarrow \infty}\Big\|\Big[u_R \big(H_{(0,R),0,0}^{(0)} 
- z I_{(0,R)}\big)^{-1}\oplus0\Big]    \no \\
& \qquad \quad - u \big(H_{(0,\infty),0}^{(0)}-z I_{(0,\infty)}\big)^{-1}
\Big\|_{\cB_2(L^2((0,\infty); dx))} = 0,   
\lb{64c}\\
& \lim_{R\rightarrow \infty} \Big\|
\Big[\ol{\big(H_{(0,R),0,0}^{(0)}-z I_{(0,R)}\big)^{-1}v_R}\oplus0\Big]    \no \\
& \qquad \quad - \ol{\big(H_{(0,\infty),0}^{(0)}-z I_{(0,\infty)}\big)^{-1}v}
\Big\|_{\cB_2(L^2((0,\infty); dx))} = 0,    \lb{65}
\end{align}
the right-hand side  of \eqref{63b} converges strongly to the right-hand side  of \eqref{60}.  We note that \eqref{64} relies on \eqref{11}, while \eqref{64b}, \eqref{64c}, and \eqref{65} follow from Lemma \ref{l4.1}. 
\end{proof}

Next we extend Lemmas \ref{l4.1}--\ref{l4.3} to the case of general separated boundary conditions applying Lemmas \ref{l3.1a} and \ref{l3.2a}.

\begin{lemma}\lb{l4.4} 
Assume \eqref{3.8} and let $z\in \C\backslash  [0,\infty)$ and 
$\alpha, \beta \in [0,\pi)$. Then 
\begin{equation}
\slim_{R\rightarrow \infty} \big(\big[H_{(0,R),\alpha,\beta}^{(0)}\oplus 0\big]-z I_{(0,\infty)}\big)^{-1} 
= \big(H_{(0,\infty),\alpha}^{(0)}-z I_{(0,\infty)}\big)^{-1}.    \lb{4.38}
\end{equation}
and 
\begin{align}
\begin{split}
\lim_{R\rightarrow \infty} \Big[u_R\big(H_{(0,R),\alpha,\beta}^{(0)}-zI_{(0,R)}\big)^{-1}\oplus 0\Big]=u \big(H_{(0,\infty),\alpha}^{(0)}-z I_{(0,\infty)}\big)^{-1}&   \\  
\text{ in $\cB_2\big(L^2((0,\infty);dx)\big)$.}&    \lb{4.39}
\end{split}
\end{align}
as well as
\begin{align}
\begin{split}
\lim_{R\rightarrow \infty} \Big[\ol{\big(H_{(0,R),\alpha,\beta}^{(0)}-zI_{(0,R)}\big)^{-1} v_R} \oplus 0\Big]= \ol{\big(H_{(0,\infty),\alpha}^{(0)}-z I_{(0,\infty)}\big)^{-1} v}&   \\  
\text{ in $\cB_2\big(L^2((0,\infty);dx)\big)$.}&    \lb{4.39a}
\end{split}
\end{align}

Moreover, 
\begin{align}\lb{4.40}
& \lim_{R\to \infty}\Big\|\Big[\ol{u_R \big(H_{(0,R),\alpha,\beta}^{(0)} 
-z I_{(0,R)}\big)^{-1} v_R} \oplus 0\Big] 
\no \\
& \qquad \;\;\, -  
\ol{u \big(H_{(0,\infty),\alpha}^{(0)}-z I_{(0,\infty)}\big)^{-1}v}
\Big\|_{\cB_1(L^2((0,\infty);dx))} = 0
\end{align}
and 
\begin{align}
\begin{split} 
& \lim_{R\to\infty}{\det}_{L^2((0,R);dx)} \Big(I_{(0,R)} 
+ \ol{u_R \big(H_{(0,R),\alpha,\beta} - z I_{(0,R)}\big)^{-1} v_R}\Big)     \\
& \quad = {\det}_{L^2((0,\infty);dx)} \Big(I_{(0,\infty)} 
+ \ol{u \big(H_{(0,\infty),\alpha} - z I_{(0,\infty)}\big)^{-1} v}\Big).   \lb{4.41}
\end{split} 
\end{align}
In addition, assume \eqref{3.8}. Then 
\begin{equation}
\slim_{R\rightarrow \infty} \big([H_{(0,R),\alpha,\beta} \oplus 0]-z I_{(0,\infty)}\big)^{-1}
= (H_{(0,\infty),\alpha}-z I_{(0,\infty)})^{-1}.    \lb{4.42}
\end{equation} 
\end{lemma}
\begin{proof}
Consecutively employing Lemmas \ref{l3.1a} and \ref{l3.2a}, we will reduce the proofs for Lemma \ref{l4.4} to those of Lemmas \ref{l4.1}--{l4.3} as follows:
\begin{align}
& G_{(0,R),\alpha,\beta} (z,x,x') = G_{(0,\infty),\alpha}^{(0)} (z,x,x')    \no \\
& \qquad + \f{\sin(\beta) \psi_+'(z,R) + \cos(\beta) \psi_+(z,R)}{\sin(\beta) 
\psi_{0,\alpha}'(z,R) + \cos(\beta) \psi_{0,\alpha}(z,R)} 
\psi_{0,\alpha}(z,x) \psi_{0,\alpha}(z,x')    \no \\
& \quad = G_{(0,\infty),0}^{(0)} (z,x,x')     \no \\
& \qquad + \f{\sin(\alpha) \psi_{0,0}'(z,0) + \cos(\alpha) \psi_{0,0}(z,0)}
{W(\psi_{0,0}(z,\cdot),\psi_+(z,\cdot))[\sin(\alpha) \psi_+'(z,0) + \cos(\beta) 
\psi_+(z,0)]} \psi_{+}(z,x) \psi_{+}(z,x')    \no \\
& \qquad + \f{\sin(\beta) \psi_+'(z,R) + \cos(\beta) \psi_+(z,R)}{\sin(\beta) 
\psi_{0,\alpha}'(z,R) + \cos(\beta) \psi_{0,\alpha}(z,R)} 
\psi_{0,\alpha}(z,x) \psi_{0,\alpha}(z,x')    \no \\
& \quad = G_{(0,R),0,0} (z,x,x')    \no \\
& \qquad + \f{\sin(\alpha) \psi_{0,0}'(z,0) + \cos(\alpha) \psi_{0,0}(z,0)}
{W(\psi_{0,0}(z,\cdot),\psi_+(z,\cdot))[\sin(\alpha) \psi_+'(z,0) + \cos(\beta) 
\psi_+(z,0)]} \psi_{+}(z,x) \psi_{+}(z,x')    \no \\
& \qquad - \f{\psi_+(z,R)}{W(\psi_{0,0}(z,\cdot),\psi_+(z,\cdot)) \psi_{0,0}(z,R)}
\psi_{0,0}(z,x) \psi_{0,0}(z,x')    \no \\
& \qquad + \f{\sin(\beta) \psi_+'(z,R) + \cos(\beta) \psi_+(z,R)}{\sin(\beta) 
\psi_{0,\alpha}'(z,R) + \cos(\beta) \psi_{0,\alpha}(z,R)} 
\psi_{0,\alpha}(z,x) \psi_{0,\alpha}(z,x'),  \quad x, x' \in (0,R).   \lb{4.43}
\end{align}
Similarly, one has 
\begin{align}
& G_{(0,\infty),\alpha}^{(0)} (z,x,x') = G_{(0,\infty),0}^{(0)} (z,x,x')     \no \\
& \qquad + \f{\sin(\alpha) \psi_{0,0}'(z,0) + \cos(\alpha) \psi_{0,0}(z,0)}
{W(\psi_{0,0}(z,\cdot),\psi_+(z,\cdot))[\sin(\alpha) \psi_+'(z,0) + \cos(\beta) 
\psi_+(z,0)]} \psi_{+}(z,x) \psi_{+}(z,x'),    \no \\ 
& \hspace*{9cm}   x, x' \in (0,\infty).    \lb{4.44}
\end{align} 
One notes that the second terms on the right-hand sides in \eqref{4.43} and \eqref{4.44} are identical except for the underlying interval $(0,R)$ versus $(0,\infty)$. In this context we refer to the subsequent, elementary result 
\eqref{4.59} in Lemma \ref{l4.5}, which deals with precisely such situations abstractly.

Next, one observes that iterating the Volterra integral equations \eqref{3.16} 
and \eqref{3.21} yields the bounds (cf.\ also \cite[Sects.\ I.2, I.3]{CS89}), 
\begin{align}
& |\psi_{0,0}(z,R)| \leq C \f{R}{1 + |z^{1/2}| R} e^{\Im(z^{1/2}) R},   
\lb{4.45} \\
& \big|\psi_{0,0}(z,R) - \psi_{(0,0}^{(0)}(z,R)\big| \leq C \f{R e^{\Im(z^{1/2}) R}}
{1 + |z^{1/2}| R} \int_0^R dx \, \f{x}{1 + |z^{1/2}| x} |V(x)|,   \\
& |\psi_+ (z,R)| \leq C e^{-\Im(z^{1/2}) R},   \\
& |\psi_+' (z,R)| \leq C |z|^{1/2} e^{-\Im(z^{1/2}) R},   \\
& \big|\psi_+ (z,R) - \psi_+^{(0)}(z,R)\big| \leq C \f{e^{\Im(z^{1/2}) R}}
{|z^{1/2}|} \int_R^{\infty} dx \, |V(x)|,  \quad z \in \bbC\backslash\{0\},    \\
& \big|\psi_+' (z,R) - \psi_+^{(0),\prime}(z,R)\big| \leq C e^{\Im(z^{1/2}) R} 
\int_R^{\infty} dx \, |V(x)|,   \lb{4.50}
\end{align}
where we recall our convention of $\Im(z^{1/2}) \geq 0$, $z\in\bbC$. 
In addition one notes that $W(\psi_{0,0}(z,\cdot),\psi_+(z,\cdot))$ is 
$x$-independent and hence just a function of $z$.  

Due to the estimates \eqref{4.45}--\eqref{4.50}, the third and fourth terms on the right-hand side of \eqref{4.43} generate rank-one operators whose 
norms exhibit a decay of the order $e^{-2\Im(z^{1/2}) R}$ as $R\to\infty$. 

At this point one can multiply \eqref{4.43} and \eqref{4.44} on the left with 
$u_R(x)$ and $u(x)$ and reduce the proof of \eqref{4.38} and \eqref{4.39} to  
\eqref{4.2} and \eqref{4.3} in Lemma \ref{l4.1} together with the exponential decay of the two rank-one operators generated by the third and fourth term on the right-hand side of \eqref{4.43}. Analogous observations apply to 
\eqref{4.39a}.

Similarly, multiplying \eqref{4.43} and \eqref{4.44} on the left with 
$u_R(x)$ and $u(x)$ and on the right with $v_R(x')$ and $v(x')$, respectively, 
one can prove \eqref{4.40} and \eqref{4.41} by reducing them to 
\eqref{11} and \eqref{12} in Lemma \ref{l4.2}. 

Finally, \eqref{4.42} follows in the same vain from \eqref{4.27} in 
Lemma \ref{l4.3}. 
\end{proof}

\begin{remark}
For subsequent purposes, we decided to provide a purely operator theoretic proof of the convergence of Fredholm determinants in \eqref{4.41}. An elementary alternative proof based on Jost functions and the Wronski relations \eqref{03.4} and \eqref{03.11} is also possible, see \cite{LS11}.
\end{remark}

In order to deal with the the second terms on the right-hand sides in 
\eqref{4.43} and \eqref{4.44} we formulate the following elementary abstract result: Let $\cH$ be a complex, separable Hilbert space and $\{P_n\}_{n\in \bbN}$ a sequence of orthogonal projections in $\cH$ converging strongly to the identity as $n\to\infty$,
\begin{equation}
P_n=P_n^2 = P_n^*, \; n \in\bbN, \quad \slim_{n\to\infty} P_n = I_{\cH}. 
\lb{4.51}
\end{equation}
We also introduce 
\begin{equation}
P_n^\bot = I_{\cH} - P_n, \quad \cH_n = P_n \cH, \quad 
\cH_n^\bot = P_n^\bot \cH, \quad n\in\bbN, 
\end{equation}
and write for any $h \in \cH$, and $n \in\bbN$, 
\begin{equation}
h = \langle h_n, h_n^\bot \rangle = h_n \oplus h_n^\bot\, \text{ with respect to } \, \cH = \cH_n \oplus \cH_n^\bot, 
\end{equation}
with the standard convention for the inner product,
\begin{equation}
(h,k)_{\cH} = \big(\langle h_n,h_n^\bot \rangle, 
\langle k_n, k_n^\bot \rangle\big)_{\cH} = 
(h_n,k_n)_{\cH_n} + (h_n^\bot, k_n^\bot)_{\cH^\bot}, \quad 
h, k \in \cH.    \lb{4.54}
\end{equation} 
Finally, for $f = \langle f_n, f_n^\bot\rangle, g = \langle g_n, g_n^\bot \rangle \in \cH$, we introduce the rank-one operator $T$ in $\cH$ by
\begin{equation} 
T h = (f, h)_{\cH} g, \quad h = (h_n, h_n^\bot) \in\cH,    \lb{4.55}
\end{equation}
and for each $n\in\bbN$, the rank-one operator $T_n$ in $\cH_n$ by
\begin{equation}
T_n h_n = (f_n, h_n) g_n, \quad n \in\bbN.   \lb{4.56}
\end{equation}

\begin{lemma} \lb{l4.5}
With the notation introduced in \eqref{4.51}--\eqref{4.56}, one has
\begin{equation}
\big[T-(T_n\oplus 0)\big] h= ((I-P_n)g,h)_{\cH}g + (P_nf, h)_{\cH}(I-P_n)g, 
\quad h \in\cH,
\end{equation}
and
\begin{align}
\begin{split} 
\|T-(T_n\oplus 0)\|_{\cB_1(\cH)} & \leq \|(I-P_n)f\|_{\cH}\|g\|_{\cH}
+ \|P_nf\|_{\cH}\|(I-P_n)g\|_{\cH}    \\
& \leq 2 \|f\|_{\cH} \|g\|_{\cH}.
\end{split} 
\end{align}
In particular, 
\begin{equation}
\lim_{n\rightarrow \infty}\|T_n\oplus 0 - T\|_{\cB_1(\cH)} = 0.   \lb{4.59}
\end{equation}
\end{lemma}
\begin{proof}
One computes for $h=h_n\oplus h_n^{\bot}\in \cH$
\begin{align}
\big[T-(T_n\oplus 0)\big]h&=Th-(T_n\oplus 0)h   \no \\
&=(f, h )_{\cH}g-\big[(f_n, h_n )_{\cH_n}g_n\oplus 0 \big]   \no \\
&=\big[(f_n, h_n)_{\cH_n}+(f_n^{\bot},h_n^{\prime})_{\cH_n^{\bot}} \big]g-\big[(f_n, h_n)_{\cH_n}\oplus 0 \big]   \no \\
&=(f_n^{\bot},h_n^{\bot})_{\cH_n^{\bot}}g_n \oplus \big[(f_n, h_n )_{\cH_n}+(f_n^{\bot},h_n^{\bot} )_{\cH_n^{\bot}} \big]g_n^{\bot}    \no \\
&=(f_n^{\bot}, h_n^{\bot})_{\cH_n^{\bot}}g+\big[0\oplus (f_n, h_n)_{\cH_n}g_n^{\bot} \big]    \no \\
&=((I-P_n)g,h)_{\cH}g+(P_nf,h)_{\cH}(I-P_n)g.
\end{align}
Thus,
\begin{equation}
T-(T_n\oplus 0)=((I-P_n)f,\cdot)_{\cH}g+(P_nf, \cdot )_{\cH}(I-P_n)g,
\end{equation}
implying 
\begin{align}
\|T-(T_n\oplus0)\|_{\cB_1(\cH)}&\leq \|((I-P_n)f,\cdot)_{\cH}g \|_{\cB_1(\cH)} +\no\\
&\quad+ \|(P_nf, \cdot )_{\cH}(I-P_n)g \|_{\cB_1(\cH)}   \no \\
&\leq \|(I-P_n)f\|_{\cH}\|g\|_{\cH}+\|P_nf\|_{\cH}\|(I-P_n)g\|_{\cH}.\lb{C1a}
\end{align}
The two terms in \eqref{C1a} go to zero since $\slim_{n\to\infty} P_n=I_{\cH}$ 
and $\|P_nf\|_{\cB(\cH)}=1$, $n\in\bbN$. 
\end{proof}

\begin{remark} \lb{r4.6}
The convergence results of Lemma \ref{l4.1}\,(i) and Lemma \ref{l4.3} and similarly those of \eqref{4.38} and \eqref{4.42} in Lemma \ref{l4.4} 
are similar in spirit to \cite[Theorem\ 6]{SW93} (c.f. also \cite{BEWZ93}, 
\cite[Theorem\ 5]{SW93}, \cite[Theorem\ 4]{SW95}, and \cite[Lemma\ 4]{Te08}).  More specifically, the authors show that $A_n\rightarrow A$ in the sense of 
\textit{generalized strong convergence}, that is, 
\begin{equation} 
\slim_{n\to\infty}(A_n - z I_{(a_n,b_n)})^{-1}P_n = (A - z I_{(a,b)})^{-1}, 
\quad z\in \C\backslash  \R, 
\end{equation} 
where $A_n$ (resp., $A$) is the self-adjoint realization of a (rather general) 
Sturm--Liouville differential expression in $L^2((a_n,b_n);r(x)dx)$ (resp., 
$L^2((a,b);r(x)dx)$), with $a_n\downarrow a$ (or $a_n=a$ identically), 
$b_n\uparrow b$, and $P_n$ denotes the orthogonal projection in 
$L^2((a,b);r(x)dx)$ onto $L^2((a_n,b_n);r(x)dx)$,  
with the latter identified with a closed subspace of $L^2((a,b);r(x)dx)$. 
\end{remark}

Next, we briefly re-examine the uniform boundedness from below of 
$H_{(0,R),\alpha,\beta}$ with respect to $R$ (cf.\ \eqref{3.108}).

\begin{lemma} \lb{l4.6a}
Assume \eqref{3.8}. \\
$(i)$ There exists $E(\alpha,\beta) < 0$, such that 
for all $z \leq E(\alpha,\beta)$ and all $R>0$, 
\begin{equation}
\Big\|\ol{u_R \big(H_{(0,R),\alpha,\beta}^{(0)} - z I_{(0,R)}\big)^{-1} v_R} 
\oplus 0\Big\|_{\cB_1 (L^2((0,\infty); dx))} \leq C_{\alpha,\beta} |z|^{-1/2}.   \lb{4.66} 
\end{equation}
for some constant $C_{\alpha, \beta}>0$, independent of $R>0$. In particular,
\begin{equation}
H_{(0,R),\alpha,\beta} \geq c_{\alpha,\beta} I_{(0,R)}    \lb{4.67}
\end{equation} 
for some $c_{\alpha,\beta} \leq 0$ independent of $R>0$. \\ 
$(ii)$ There exists $E(\alpha) < 0$, such that for all $z \leq E(\alpha)$, 
\begin{equation}
\Big\|\ol{u \big(H_{(0,\infty),\alpha}^{(0)} - z I_{(0,R)}\big)^{-1} v}
\Big\|_{\cB_1 (L^2((0,\infty); dx))} \leq C_{\alpha} |z|^{-1/2}.     \lb{4.68}
\end{equation}
for some constant $C(\alpha)>0$. In particular,
\begin{equation}
H_{(0,\infty),\alpha} \geq c_{\alpha} I_{(0,\infty)}    \lb{4.69}
\end{equation} 
for some $c_{\alpha} \leq 0$. 
\end{lemma} 
\begin{proof}
Combining \eqref{3.6}, \eqref{3.13}, \eqref{3.14}, \eqref{3.15}, \eqref{3.19}, and 
\eqref{3.20}, one infers that 
\begin{align}
& \Big|G_{(0,R),\alpha,\beta}^{(0)} (z,x,x)\Big| \leq C(\alpha,\beta) |z|^{-1/2}, \quad 
x \in [0,R], \; R>0, \; z < E(\alpha,\beta),    \\[1mm] 
& \Big|G_{(0,\infty),\alpha}^{(0)} (z,x,x)\Big| \leq C(\alpha) |z|^{-1/2}, \quad 
x \in [0,\infty), \; z < E(\alpha) 
\end{align}
for some constants $C(\alpha,\beta)>0$, $C(\alpha) > 0$. 

Thus, using $u = \sgn(V) v$, with $\sgn(V(x) = 1$ if $V(x) \geq 0$ a.e., and $\sgn(V(x)) = -1$ 
if $V(x) < 0$ a.e., rendering $\sgn(V)$ a unitary operator of multiplication in 
$L^2((0,\infty); dx)$, and analogously for $u_R = \sgn(V_R) v_R$, one infers for $z<0$ 
sufficiently negative, that 
\begin{align}
& \Big\|u_R \big(H_{(0,R),\alpha,\beta}^{(0)} - z I_{(0,R)})^{-1} v_R
\Big\|_{\cB_1(L^2((0,R); dx))}    \no \\
& \quad = \Big\|v_R \big(H_{(0,R),\alpha,\beta}^{(0)} 
- z I_{(0,R)})^{-1} v_R\Big\|_{\cB_1(L^2((0,R); dx))}    \no \\
& \quad = {\tr}_{L^2((0,R); dx)} \Big(v_R \big(H_{(0,R),\alpha,\beta}^{(0)} 
- z I_{(0,R)}\big)^{-1} v_R\Big)   \no \\
& \quad = \int_0^R dx \, |V(x)| G_{(0,R),\alpha,\beta}^{(0)} (z,x,x)   \no \\
& \quad \leq C(\alpha,\beta) \int_0^R dx \, |V(x)| \, |z|^{-1/2}, 
\end{align}
and analogously, for $z<0$ sufficiently negative,
\begin{equation}
\Big\|u \big(H_{(0,\infty),\alpha}^{(0)} - z I_{(0,\infty)})^{-1} v
\Big\|_{\cB_1(L^2((0,\infty); dx))}  \leq C(\alpha) \int_0^\infty dx \, |V(x)| \, |z|^{-1/2}, 
\end{equation}
proving \eqref{4.66} and \eqref{4.68}. 

Thus, for $z \leq c_{\alpha,\beta}$, for some $c_{\alpha,\beta} \leq 0$ independent 
of $R>0$,
\begin{equation}
\Big[I_{(0,\infty)}  
+\ol{u_R(H_{(0,R),0,0}^{(0)}-z I_{(0,R)})^{-1}v_R}\oplus 0 \Big]^{-1} \in 
\cB \big(L^2((0,\infty); dx)\big), 
\end{equation}
proving \eqref{4.67}, employing the resolvent equation \eqref{63}. The lower bound 
\eqref{4.69} is proved analogously (but also follows from the estimate \eqref{2.124} 
with $(0,R)$ replaced by $(0,\infty)$, cf., e.g., \cite[eq.\ (2.7.8)]{Sc81}). 
\end{proof}

We conclude this section with the following observation that will play an important role in Section \ref{s5} later on:

\begin{remark} \lb{r4.7}
We emphasize that our choice of $H_{(0,R),0,0}^{(0)}$ and 
$H_{(0,\infty),0}^{(0)}$ in Lemmas \ref{l4.1} and \ref{l4.2} is of no importance.
In fact, using the general inequalities \eqref{3.33a} and \eqref{3.33b} (for 
$q = 1/2$), and iterating the corresponding Volterra integral equations 
\eqref{3.16} and \eqref{3.21} to obtain standard bounds on 
$\psi_{0,0}(z,\cdot)$, $\psi_{R,0}(z,\cdot)$, and $\psi_+(z,\cdot)$ in terms of those of $\psi_{0,0}^{(0)}(z,\cdot)$, 
$\psi_{R,0}^{(0)}(z,\cdot)$, and $\psi_+^{(0)}(z,\cdot)$ with certain $z$-dependent constants shows that Lemmas \ref{l4.1} and \ref{l4.2} extend 
to the case where $H_{(0,R),0,0}^{(0)}$ and $H_{(0,\infty),0}^{(0)}$ are replaced by general Schr\"odinger operators $H_{(0,R),0,0}$ and 
$H_{(0,\infty),0}$ as long as their underlying potentials $W$ satisfy  
$W \in L^1((0,R); dx)$ and $W \in L^1((0,\infty); dx)$, respectively. 

Using again \eqref{4.43} and \eqref{4.44} this observation extends to 
$H_{(0,R),\alpha,\beta}^{(0)}$ and $H_{(0,\infty),\alpha}^{(0)}$
in Lemma \ref{l4.4}. Hence adding a potential $W$ satisfying   
$W \in L^1((0,R); dx)$ and $W \in L^1((0,\infty); dx)$, respectively, 
to $H_{(0,R),\alpha,\beta}^{(0)}$ and $H_{(0,\infty),\alpha}^{(0)}$ will not 
alter the conclusions in Lemma \ref{l4.4}. Put differently, our methods immediately apply to the case where the ``unperturbed'' operator contains a background potential $W$ (with 
$W \in L^1((0,\infty); dx)$). 
\end{remark}

\section{Weak Convergence of Spectral Shift Functions}  \lb{s5}

In this section we prove our principal new results, the weak convergence 
of the absolutely continuous measures 
$\xi\big(\lambda; H_{(0,R),\alpha,\beta}, 
H_{(0,R),\alpha,\beta}^{(0)}\big) d\lambda$, associated 
with the pair of self-adjoint operators 
$\big(H_{(0,R),\alpha,\beta}, H_{(0,R),\alpha,\beta}^{(0)}\big)$, to the 
absolutely continuous measure  
$\xi\big(\lambda; H_{(0,\infty),\alpha}, 
H_{(0,\infty),\alpha}^{(0)}\big) d\lambda$, associated with the pair 
$\big(H_{(0,\infty),\alpha}, H_{(0,\infty),\alpha}^{(0)}\big)$. Our approach will be 
based on convergence of underlying Fredholm determinants combined with certain measure theoretic facts. 

We will freely use the facts recorded (and notations employed) on 
spectral shift functions in Appendix \ref{sA} and, in particular, note that 
\eqref{2.7}, \eqref{2.8} imply
\begin{align}
& - \f{d}{dz} {\det}_{L^2((0,\infty);dx)} \Big(I_{(0,\infty)} 
+ \ol{u \big(H_{(0,\infty),\alpha} - z I_{(0,\infty)}\big)^{-1} v}\Big)   \no \\
& \quad = {\tr}_{L^2((0,\infty);dx)}\big ((H_{(0,\infty),\alpha}-z)^{-1} 
- \big(H_{(0,\infty),\alpha}^{(0)}-z\big)^{-1}\big)    \lb{5.0}  \\
& \quad = - \int_{\R}\frac{\xi\big(\lambda; H_{(0,\infty),\alpha}, 
H_{(0,\infty),\alpha}^{(0)}\big) d\lambda}{(\lambda-z)^2},    \quad 
z\in \C\backslash \R.       \no 
\end{align}
Combined with \eqref{3.47}, this represents a result due to Buslaev and Faddeev 
\cite{BF60} in the Dirichlet case $\alpha = 0$. We will also use the analog of 
\eqref{5.0} with $(0,\infty)$ replaced by $(0,R)$ and $u, v$ replaced by $u_R, v_R$, etc.  

We start with the following basic result:

\begin{lemma}\lb{l5.1}
Assume \eqref{3.8} and let $a,z\in \bbC \backslash \bbR$. Then 
\begin{align}
& \lim_{R\rightarrow \infty}\ln \left(\frac{{\det}_{L^2((0,R);dx)} \Big(I_{(0,R)} 
+ \ol{u_R \big(H_{(0,R),\alpha,\beta} - z I_{(0,R)}\big)^{-1} v_R}\Big)}{{\det}_{L^2((0,R);dx)} \Big(I_{(0,R)} 
+ \ol{u_R \big(H_{(0,R),\alpha,\beta} - a I_{(0,R)}\big)^{-1} v_R}\Big)} \right)
\no \\
& \quad =\ln \left(\frac{{\det}_{L^2((0,\infty);dx)} \Big(I_{(0,\infty)} 
+ \ol{u \big(H_{(0,\infty),\alpha} - z I_{(0,\infty)}\big)^{-1} v}\Big)}{{\det}_{L^2((0,\infty);dx)} \Big(I_{(0,\infty)} 
+ \ol{u \big(H_{(0,\infty),\alpha} - a I_{(0,\infty)}\big)^{-1} v}\Big)}\right),    \lb{30} \\
& \lim_{R\rightarrow \infty}\int_{\R}\frac{\xi\big(\lambda; H_{(0,R),\alpha,\beta}, 
H_{(0,R),\alpha,\beta}^{(0)}\big) d\lambda}{(\lambda -a )(\lambda-z)^n}=\int_{\R}\frac{\xi\big(\lambda; H_{(0,\infty),\alpha}, 
H_{(0,\infty),\alpha}^{(0)}\big) d\lambda}{(\lambda-a)(\lambda-z)^n},\quad n\in \N.    \lb{31}
\end{align}
\end{lemma}
\begin{proof} 
Convergence in \eqref{30} is a direct consequence of \eqref{4.41}.  Moreover, since
\begin{align}
&\ln \left(\frac{{\det}_{L^2((0,R);dx)} \Big(I_{(0,R)} 
+ \ol{u_R \big(H_{(0,R),\alpha,\beta} - z I_{(0,R)}\big)^{-1} v_R}\Big)}{{\det}_{L^2((0,R);dx)} \Big(I_{(0,R)} 
+ \ol{u_R \big(H_{(0,R),\alpha,\beta} - a I_{(0,R)}\big)^{-1} v_R}\Big)} \right)   \no \\
& \quad =(z-a)\int_{\R}\frac{\xi\big(\lambda; H_{(0,R),\alpha,\beta}, 
H_{(0,R),\alpha,\beta}^{(0)}\big) d\lambda}{(\lambda-a)(\lambda-z)},     \lb{5-1}\\
& \ln\left(\frac{{\det}_{L^2((0,\infty);dx)} \Big(I_{(0,\infty)} 
+ \ol{u \big(H_{(0,\infty),\alpha} - z I_{(0,\infty)}\big)^{-1} v}\Big)}{{\det}_{L^2((0,\infty);dx)} \Big(I_{(0,\infty)}  
+ \ol{u \big(H_{(0,\infty),\alpha} - a I_{(0,\infty)}\big)^{-1} v}\Big)} \right)   \no \\
& \quad =(z-a)\int_{\R}\frac{\xi\big(\lambda; H_{(0,\infty),\alpha}, 
H_{(0,\infty),\alpha}^{(0)}\big) d\lambda}{(\lambda-a)(\lambda-z)}       \lb{5-2}
\end{align}
(cf.\ \eqref{5.0} and its half-line analog), one immediately gets
\begin{align}
\begin{split}
\lim_{R\rightarrow \infty}\int_{\R}\frac{\xi\big(\lambda; H_{(0,R),\alpha}, 
H_{(0,R),\alpha}^{(0)}\big) d\lambda}{(\lambda-a)(\lambda-z)} = 
\int_{\R}\frac{\xi\big(\lambda; H_{(0,\infty),\alpha}, 
H_{(0,\infty),\alpha}^{(0)}\big) d\lambda}{(\lambda-a)(\lambda-z)},    \\
z\neq a, \quad z,a\in \C\backslash \R.&     \lb{5-3}
\end{split}
\end{align}

To verify \eqref{31}, we start with the basic identities (see, e.g., \cite[Ch.\ 8]{Ya92})
\begin{align}
& {\tr}_{L^2((0,R);dx)}\big (\big(H_{(0,R),\alpha,\beta}^{(0)}-z I_{(0,R)}\big)^{-n} 
- (H_{(0,R),\alpha,\beta}-z I_{(0,R)})^{-n}\big)   \no \\
& \quad =n\int_{\R}\frac{\xi\big(\lambda; H_{(0,R),\alpha,\beta}, 
H_{(0,R),\alpha,\beta}^{(0)}\big) d\lambda}{(\lambda-z)^{n+1}},   \lb{5-4} \\
& {\tr}_{L^2((0,\infty);dx)}\big(\big(H_{(0,\infty),\alpha}^{(0)} -z I_{(0,\infty)}\big)^{-n}
-(H_{(0,\infty),\alpha} -z I_{(0,\infty)})^{-n}\big)    \no \\
& \quad =n\int_{\R}\frac{\xi\big(\lambda; H_{(0,\infty),\alpha}, 
H_{(0,\infty),\alpha}^{(0)}\big) d\lambda}{(\lambda-z)^{n+1}},       \lb{5-5} \\
&\hspace*{3.2cm} n\in \N, \; z\in \C\backslash \R.     \no 
\end{align}
Next we claim that 
\begin{align}
&\lim_{R\rightarrow \infty} {\tr}_{L^2((0,R);dx)}\big(\big(H_{(0,R),\alpha,\beta}^{(0)}
-z I_{(0,R)}\big)^{-n}
-(H_{(0,R),\alpha,\beta}-z I_{(0,R)})^{-n} \big)    \no \\
& \quad = {\tr}_{L^2((0,\infty);dx)}\big((H_{(0,\infty),\alpha}^{(0)} -z I_{(0,\infty)})^{-n}
-(H_{(0,\infty),\alpha} -z I_{(0,\infty)})^{-n}\big),    \lb{5-6} \\
& \hspace*{8.1cm}  n\in \N, \; z\in \C\backslash \R.    \no 
\end{align}
To see this, one notes that
\begin{align}
& {\tr}_{L^2((0,R);dx)}\big(\big(H_{(0,R),\alpha,\beta}^{(0)}-z I_{(0,R)}\big)^{-n}
- (H_{(0,R),\alpha,\beta}-z I_{(0,R)})^{-n} \big)   \no \\
& \quad = {\tr}_{L^2((0,\infty);dx)}\big(\big(\big(H_{(0,R),\alpha,\beta}^{(0)}
- z I_{(0,R)}\big)^{-n}-(H_{(0,R),\alpha,\beta}-z I_{(0,R)})^{-n}\big)\oplus 0 \big)\no\\
 &\quad = {\tr}_{L^2((0,\infty);dx)}\bigg[\bigg(\big(H_{(0,R),\alpha,\beta}^{(0)}
 - z I_{(0,R)}\big)^{-1}\oplus \frac{-1}{z} I_{(R,\infty)}\bigg)^n   \\
 & \qquad  -  \bigg((H_{(0,R),\alpha,\beta}-z I_{(0,R)})^{-1}
 \oplus \frac{-1}{z} I_{(R,\infty)}\bigg)^n  \bigg],   \quad 
n\in \N, \; z\in \C\backslash \R.    \lb{5-7}
\end{align}
Since the trace functional is continuous with respect to the 
$\cB_1\big(L^2((0,\infty);dx)\big)$-norm, to verify \eqref{5-6}, it suffices to prove that
\begin{align}
& \sum_{k=1}^n \bigg(\big(H_{(0,R),\alpha,\beta}^{(0)}-z I_{(0,R)}\big)^{-1} 
\oplus\frac{-1}{z} I_{(R,\infty)}\bigg)^{n-k} 
\bigg[\big(H_{(0,R),\alpha,\beta}^{(0)}-z I_{(0,R)}\big)^{-1}    \no \\
& \quad - (H_{(0,R),\alpha,\beta}-z I_{(0,R)})^{-1}\oplus 0 \bigg]    
\bigg((H_{(0,R),\alpha,\beta} -z I_{(0,R)})^{-1}
\oplus\frac{-1}{z} I_{(R,\infty)}\bigg)^{k-1}  \lb{5-8}
\end{align}
converges to 
\begin{align}
& \sum_{k=1}^n \big(H_{(0,\infty),\alpha}^{(0)} -z I_{(0,\infty)}\big)^{k-n}
\big[\big(H_{(0,\infty),\alpha}^{(0)} -z I_{(0,\infty)}\big)^{-1}
- (H_{(0,\infty),\alpha} -z I_{(0,\infty)})^{-1} \big]    \no \\
& \qquad \times (H_{(0,\infty),\alpha} -z I_{(0,\infty)})^{1-k}   \lb{5-9}
\end{align}
in $\cB_1\big(L^2((0,\infty);dx)\big)$ as $R\rightarrow \infty$, since \eqref{5-8} 
is the operator under the trace on the right-hand side  of \eqref{5-7} and \eqref{5-9} is the operator under the trace on the right-hand side  of \eqref{5-6}.\footnote{Here we have made use of the identity $A^n-B^n=\sum_{k=1}^nA^{n-k}(A-B)B^{k-1}$ for $A, B \in \cB(\cH)$.}

By \eqref{4.38} , one concludes that 
\begin{align}
\slim_{R\rightarrow \infty} \bigg(\big(H_{(0,R),\alpha,\beta}^{(0)}-z I_{(0,R)}\big)^{-1}
\oplus\frac{-1}{z} I_{(R,\infty)}\bigg)^{n-k}
&=\big(H_{(0,\infty),\alpha}^{(0)} -z I_{(0,\infty)}\big)^{k-n},    \\
\slim_{R\rightarrow \infty} \bigg((H_{(0,R),\alpha,\beta} -z I_{(0,R)})^{-1}
\oplus\frac{-1}{z} I_{(R,\infty)}\bigg)^{k-1}&=(H_{(0,\infty),\alpha}-z I_{(0,\infty)})^{1-k}.
\end{align}
Thus, convergence of \eqref{5-8} to \eqref{5-9}, will follow from Gr\"{u}mm's Theorem \cite{Gr73} (see also the discussion in \cite[Ch.\ 2]{Si05}) if we can show that 
\begin{align}
\begin{split}
& \lim_{R\rightarrow \infty}\big[\big(\big(H_{(0,R),\alpha,\beta}^{(0)}-z I_{(0,R)}\big)^{-1}
- (H_{(0,R),\alpha,\beta}-z I_{(0,R)})^{-1}\big)\oplus 0\big]    \lb{5-10} \\
& \quad =\big(H_{(0,\infty),\alpha}^{(0)} -z I_{(0,\infty)}\big)^{-1}
- (H_{(0,\infty),\alpha} -z I_{(0,\infty)})^{-1}
\, \text{ in $\cB_1\big(L^2((0,\infty);dx)\big)$.}
\end{split}
\end{align}
 The $\cB_1\big(L^2((0,\infty);dx)\big)$-convergence in \eqref{5-10} follows readily from the identities 
\begin{align}
&\big((H_{(0,R),\alpha,\beta}-z I_{(0,R)})^{-1}
- \big(H_{(0,R),\alpha,\beta}^{(0)}-z I_{(0,R)}\big)^{-1}\big)\oplus 0     \no \\
& \quad =\big(\overline{\big(H_{(0,R),\alpha,\beta}^{(0)}-z I_{(0,R)}\big)^{-1}v_R}
\oplus0\big)    \no\\
&\qquad \times \bigg(\Big[I_{(0,\infty)}
+\big(\overline{u_R\big(H_{(0,R),\alpha,\beta}^{(0)}-z I_{(0,R)}\big)^{-1}v_R}\oplus 0\big) \Big]^{-1}-\big(0 \oplus I_{(R,\infty)} \big)\bigg)    \no\\
&\qquad \times 
\big( u_R\big(H_{(0,R),\alpha,\beta}^{(0)}-z I_{(0,R)}\big)^{-1}\oplus 0\big),  \lb{5-11}\\
&\big(H_{(0,\infty),\alpha}^{(0)} -z I_{(0,\infty)}\big)^{-1}
- (H_{(0,\infty),\alpha} -z I_{(0,\infty)})^{-1}   \no\\
&\quad = \overline{\big(H_{(0,\infty),\alpha}^{(0)} -z I_{(0,\infty)}\big)^{-1}v}
\Big[I_{(0,\infty)}+\overline{u\big(H_{(0,\infty),\alpha}^{(0)} -z I_{(0,\infty)}\big)^{-1}v} 
\Big]^{-1}   \no \\
& \qquad \times u\big(H_{(0,\infty),\alpha}^{(0)} -z I_{(0,\infty)}\big)^{-1}.   \lb{5-12}
\end{align}
Applying \eqref{64} and \eqref{64d} yields the strong convergence
\begin{align}
\begin{split} \lb{5-13}
& \slim_{R\rightarrow \infty}\bigg(\big[I_{(0,\infty)}
+\big(\overline{u_R(H_{(0,R),\alpha,\beta}^{(0)}-z I_{(0,R)})^{-1}v_R}
\oplus 0\big) \big]^{-1}-\big(0 \oplus I_{(R,\infty)} \big)\bigg)   \\
& \quad = \bigg[I_{(0,\infty)}+\overline{u(H_{(0,\infty),\alpha}^{(0)} -z I_{(0,\infty)})^{-1}v}\bigg]^{-1}.
\end{split}
\end{align}
Therefore, \eqref{4.39} and \eqref{5-13} together with Gr\"{u}mm's Theorem \cite{Gr73} yield
\begin{align}
&\lim_{R\rightarrow \infty}\bigg(\Big[I_{(0,\infty)}+\big(\overline{u_R\big(H_{(0,R),\alpha,\beta}^{(0)}-z I_{(0,R)}\big)^{-1}v_R}\oplus 0\big) \Big]^{-1}-\big(0 \oplus I_{(R,\infty)} \big)\bigg)    \no \\
& \qquad \times \big(u_R\big(H_{(0,R),\alpha,\beta}^{(0)}-z I_{(0,R)}\big)^{-1}\oplus 0\big)  \no\\
&\quad =\Big[I_{(0,\infty)}
+\overline{v\big(H_{(0,\infty),\alpha}^{(0)} -z I_{(0,\infty)}\big)^{-1}v}\Big]^{-1}u\big(H_{(0,\infty),\alpha}^{(0)} -z I_{(0,\infty)}\big)^{-1}\lb{5-14}
\end{align}
in $\cB_2\big(L^2((0,\infty);dx)\big)$.  The convergence in \eqref{5-14} and \eqref{4.39a} yields convergence of the right-hand side  of \eqref{5-11} to 
\eqref{5-12} in $\cB_1\big(L^2((0,\infty);dx)\big)$, implying \eqref{5-10}.

Employing \eqref{5-4}, \eqref{5-5}, and \eqref{5-10}, we have shown that
\begin{equation}\lb{5-15}
\lim_{R\rightarrow \infty}\int_{\R}\frac{\xi\big(\lambda; H_{(0,R),\alpha,\beta}, 
H_{(0,R),\alpha,\beta}^{(0)}\big) d\lambda}{(\lambda-z)^{n+1}} 
= \int_{\R}\frac{\xi\big(\lambda; H_{(0,\infty),\alpha}, 
H_{(0,\infty),\alpha}^{(0)}\big) d\lambda}{(\lambda-z)^{n+1}}, \quad n\in \N,
\end{equation}
so that \eqref{31} holds in the special case $a=z$.  Thus, it remains to settle the case $z\neq a$.

For $z\neq a$ one notes that 
\begin{align}
\int_{\R}\frac{\xi\big(\lambda; H_{(0,R),\alpha,\beta}, 
H_{(0,R),\alpha,\beta}^{(0)}\big) d\lambda}{(\lambda-a)(\lambda-z)^{n+1}} 
&=\frac{1}{z-a}\Bigg[\int_{\R}\frac{\xi\big(\lambda; H_{(0,R),\alpha,\beta}, 
H_{(0,R),\alpha,\beta}^{(0)}\big) d\lambda}{(\lambda - z)^{n+1}}    \no \\
& \hspace*{.7cm} -\int_{\R}\frac{\xi\big(\lambda; H_{(0,R),\alpha,\beta}, 
H_{(0,R),\alpha,\beta}^{(0)}\big) d\lambda}{(\lambda-a)(\lambda-z)^n}\Bigg],  
\lb{5-16}\\
\int_{\R}\frac{\xi\big(\lambda; H_{(0,\infty),\alpha}, 
H_{(0,\infty),\alpha}^{(0)}\big)d\lambda}{(\lambda-a)(\lambda-z)^{n+1}} 
&=\frac{1}{z-a}\Bigg[\int_{\R}\frac{\xi\big(\lambda; H_{(0,\infty),\alpha}, 
H_{(0,\infty),\alpha}^{(0)}\big) d\lambda}{(\lambda - z)^{n+1}}     \no \\
& \hspace*{.7cm} -\int_{\R}\frac{\xi\big(\lambda; H_{(0,\infty),\alpha}, 
H_{(0,\infty),\alpha}^{(0)}\big) d\lambda}{(\lambda-a)(\lambda-z)^n}\Bigg]. 
\lb{5-17}
\end{align}
Convergence in \eqref{31} now follows from \eqref{5-15} and the two 
identities \eqref{5-16} and \eqref{5-17} via a simple induction on $n\in\bbN$.  
We emphasize that \eqref{5-3} yields the crucial first induction step, $n=1$, since \eqref{5-3} implies, via \eqref{5-16} and \eqref{5-17}, that
\begin{equation}
\lim_{R\rightarrow \infty}\int_{\R}\frac{\xi\big(\lambda; H_{(0,R),\alpha,\beta}, 
H_{(0,R),\alpha,\beta}^{(0)}\big) d\lambda}{(\lambda-a)(\lambda-z)} 
=\int_{\R}\frac{\xi\big(\lambda; H_{(0,\infty),\alpha}, 
H_{(0,\infty),\alpha}^{(0)}\big) d\lambda}{(\lambda-a)(\lambda-z)}  
\end{equation}
for $z\neq a$.
\end{proof}

In the following we denote by $C_{\infty}(\R)$ the space of continuous functions on $\bbR$ vanishing at infinity.

\begin{lemma} \lb{l5.2}
Let $f, f_n\in L^1(\bbR; d\lambda)$ and suppose that for some fixed $M>0$, 
$\|f_n\|_{L^1(\bbR;d\lambda)}\leq M$, $n\in\bbN$.  If
\begin{equation} \lb{A2}
\lim_{n\rightarrow \infty} \int_{\bbR} f_n(\lambda) d\lambda \, 
P((\lambda+i)^{-1},(\lambda-i)^{-1}) = \int_{\bbR}  
f(\lambda) d\lambda \, P((\lambda+i)^{-1},(\lambda-i)^{-1})
\end{equation}
for all polynomials $P(\cdot,\cdot)$ in two variables, then
\begin{equation}
\lim_{n\rightarrow \infty}\int_{\bbR} 
f_n(\lambda) d\lambda \, g(\lambda)=\int_{\bbR}  
f(\lambda) d\lambda \,g(\lambda), \quad g\in C_{\infty}(\bbR). 
\end{equation}
\end{lemma}
\begin{proof}
Let $\varepsilon>0$ and $g\in C_{\infty}(\bbR)$.  Since by a 
Stone--Weierstrass argument, polynomials in 
$(\lambda \pm i)^{-1}$ are dense in $C_{\infty}(\bbR)$, there is a polynomial $P(\cdot,\cdot)$ in two variables such that writing 
\begin{equation}
\cP(\lambda)=P((\lambda+i)^{-1},(\lambda-i)^{-1}), 
\quad \lambda \in\bbR,
\end{equation}
one concludes that  
\begin{equation}
\|g-\cP\|_{L^\infty(\bbR; d\lambda)} \leq 
\frac{\varepsilon}{2[M+\|f\|_{L^1(\bbR;d\lambda)}]}. 
\end{equation}
By \eqref{A2}, there exists an $N(\varepsilon) \in \bbN$ such that
\begin{equation}
\bigg|\int_{\bbR} f_n(\lambda) d\lambda \, \cP(\lambda) -\int_{\bbR} 
f(\lambda) d\lambda \, \cP(\lambda) \bigg| \leq \frac{\varepsilon}{2}
\, \text{ for all $n \geq N(\varepsilon)$.} 
\end{equation}
Therefore, if $n \geq N(\varepsilon)$, 
\begin{align}
& \bigg| \int_{\bbR} f_n(\lambda)d\lambda \, g(\lambda)  
- \int_{\bbR} f(\lambda) d\lambda \, g(\lambda) \bigg| 
\leq \big[\|f_n\|_{L^1(\bbR;d\lambda)}+\|f\|_{L^1(\bbR;d\lambda)}\big] 
\|g-\cP\|_{L^\infty(\bbR; d\lambda)}   \no \\
& \quad + \bigg|\int_{\bbR} f_n(\lambda) d\lambda \, \cP(\lambda) 
-\int_{\bbR} f(\lambda) d\lambda \, \cP(\lambda)\bigg| 
\leq \varepsilon.  
\end{align}
\end{proof}

Next, we continue with some preparations needed to prove the principal results of this section. We start by recalling some basic notions regarding the convergence of positive measures (essentially following Bauer 
\cite[\&\ 30]{Ba01}). Denoting by $\mathscr{M}_+(E)$ the set of all {\it positive Radon measures} on a locally compact space $E$, and by 
\begin{equation}
\mathscr{M}_+^b(E)=\{\mu \in \mathscr{M}_+(E)\,|\, \mu(E)<+\infty\}, 
\end{equation}
the set of all \textit{finite} positive Radon measures on $E$, we note that in the special case $E=\R^n$, $n\in\bbN$, $\mathscr{M}_+^b(\R^n)$ represents the set of all finite positive Borel measures on $\R^n$.  

If $\mu$ is a Radon measure, a point $x\in E$ is called an \textit{atom} of $\mu$ if $\mu(\{x\})>0$.  

In the following, $C_0(E)$ denotes the continuous functions on $E$ with compact support, and $C_b(E)$ represents the bounded continuous functions on $E$.

\begin{definition} \lb{d5.3}
Let $E$ be a locally compact space. \\
$(i)$ A sequence $\{\mu_n\}_{n\in\bbN} \subset \mathscr{M}_+(E)$ is said to be \textit{vaguely convergent} to a Radon measure $\mu \in \mathscr{M}_+(E)$ if
\begin{equation}
\lim_{n\to \infty}\int_E d\mu_n\, g = \int_E d\mu \, g, \quad g \in C_0(E). 
\end{equation}
$(ii)$ A sequence $\{\mu_n\}_{n\in\bbN} \subset \mathscr{M}_+^b(E)$ is said to be \textit{weakly convergent} to $\mu \in \mathscr{M}_+^b(E)$ if 
\begin{equation}
\lim_{n\to \infty}\int_E d\mu_n \, f=\int_E d\mu \, f, \quad f\in C_b(E). 
\end{equation}
$(iii)$ A Borel set $B \subset E$ is called \textit{boundaryless} with respect to the measure $\mu \in \mathscr{M}_+^b(E)$  
$($in short, $\mu$-\textit{boundaryless}$)$, if the boundary $\partial B$ of $B$ has $\mu$-measure equal to zero, $\mu(\partial B) = 0$.
\end{definition} 

\begin{theorem}[\cite{HLMW01}, Proposition 4.3] \lb{t5.4}
Let $\mu, \mu_n$, $n\in \N$, be positive $($not necessarily finite$)$ Borel measures on $\R$.  Then the following two items are equivalent: \\
$(i)$ The finiteness $\mu((-\infty,\lambda))<\infty$ holds for all $\lambda\in \R$ and the relation 
\begin{equation}\lb{176}
\lim_{n\rightarrow \infty}\mu_n((-\infty,\lambda))=\mu((-\infty,\lambda))
\end{equation}
holds for all $\lambda\in \R$ except for at most countably many $\lambda \in \bbR$ with $\mu(\{\lambda\})\neq 0$. \\
$(ii)$ The sequence $\{\mu_n\}_{n\in\bbN}$ converges vaguely to $\mu$ as 
$n\to \infty$ and the relation
\begin{equation}\lb{177}
\lim_{\lambda\downarrow -\infty}\Big(\limsup_{n\to \infty}
\big(\mu_n((-\infty,\lambda))\big) \Big)= 0
\end{equation}
holds. 
\end{theorem}

\begin{theorem}[\cite{Ba01}, Theorem 30.8] \lb{Bauer30.8}
Suppose that the sequence $\{\mu_n\}_{n\in\bbN} \subset \mathscr{M}_+^b(E)$ converges vaguely to the measure $\mu \in \mathscr{M}_+^b(E)$.  Then the following statements are equivalent: \\
$(i)$ The sequence $\mu_n$ converges weakly to $\mu$ as $n\to\infty$.  \\
$(ii)$ $\lim_{n\rightarrow \infty} \mu_n (E) = \mu (E)$.   \\
$(iii)$ For every $\varepsilon>0$ there exists a compact set $K_{\varepsilon}$  of $E$ such that 
\begin{equation}
\mu_n(E \backslash K_{\varepsilon}) \leq \varepsilon, \quad n\in \N.
\end{equation}
\end{theorem}

\begin{theorem}[\cite{Ba01}, Theorem 30.12]\lb{Bauer30.12}
Suppose that the sequence $\{\mu_n\}_{n\in\bbN} \subset \mathscr{M}_+^b(E)$ converges weakly to $\mu \in \mathscr{M}_+^b(E)$. Then
\begin{equation}
\lim_{n\rightarrow \infty}\int_E d\mu_n \, f = \int_E d\mu \, f
\end{equation}
holds for every bounded Borel measurable function $f$ that is $\mu$-almost everywhere continuous on $E$.  In particular, 
\begin{equation}
\lim_{n\rightarrow \infty}\mu_n(B)=\mu(B)
\end{equation}
holds for every $\mu$-boundaryless Borel set $B$.
\end{theorem}

As usual, finite signed Radon measures are viewed as differences of finite positive Radon measures in the following. 

Given these preparations, we now return to our concrete case at hand: Decomposing $V$ as 
\begin{equation}
V = V_+ - V_-, \quad V_{\pm} = [|V| \pm V]/2, 
\end{equation} 
and analogously for $V_R$, that is, $V_R = V_{R,+} - V_{R,-}$, we now decompose the spectral shift functions 
$\xi\big(\cdot; H_{(0,R);\alpha,\beta}, H_{(0,R);\alpha,\beta}^{(0)}\big)$ and 
$\xi\big(\cdot; H_{(0,\infty);\alpha}, H_{(0,\infty);\alpha}^{(0)}\big)$ according to 
\eqref{2.11a}, \eqref{2.11c} into a positive and negative contribution as follows,
\begin{align}
\xi\big(\cdot; H_{(0,R);\alpha,\beta}, H_{(0,R);\alpha,\beta}^{(0)}\big) &=
\xi\big(\cdot; H_{(0,R);\alpha,\beta}^{(0)} +_q V_{R,+} -_q V_{R,-}, 
H_{(0,R);\alpha,\beta}^{(0)} +_q V_{R,+}\big)    \no \\
& \quad + 
\xi\big(\cdot; H_{(0,R);\alpha,\beta}^{(0)} +_q V_{R,+}, 
H_{(0,R);\alpha,\beta}^{(0)}\big), \\
\xi\big(\cdot; H_{(0,\infty);\alpha}, H_{(0,\infty);\alpha}^{(0)}\big) &=
\xi\big(\cdot; H_{(0,\infty);\alpha}^{(0)} +_q V_+ -_q V_-, 
H_{(0,\infty);\alpha}^{(0)} +_q V_+\big)    \no \\
& \quad 
+ \xi\big(\cdot; H_{(0,\infty);\alpha}^{(0)} +_q V_+, H_{(0,\infty);\alpha}^{(0)}\big) 
\end{align} 
where 
\begin{align}
& \xi\big(\cdot; H_{(0,R);\alpha,\beta}^{(0)} +_q V_{R,+}, 
H_{(0,R);\alpha,\beta}^{(0)}\big) \geq 0,    \\
& \xi\big(\cdot; H_{(0,R);\alpha,\beta}^{(0)} +_q V_{R,+} -_q V_{R,-}, 
H_{(0,R);\alpha,\beta}^{(0)} +_q V_{R,+}\big) \leq 0,   \\
& \xi\big(\cdot; H_{(0,\infty);\alpha}^{(0)} +_q V_+, H_{(0,\infty);\alpha}^{(0)}\big) \geq 0,    \\
& \xi\big(\cdot; H_{(0,\infty);\alpha}^{(0)} +_q V_+ -_q V_-, 
H_{(0,\infty);\alpha}^{(0)} +_q V_+\big) \leq 0. 
\end{align}

\begin{theorem}\lb{t5.3}
Assume \eqref{3.8} and let $g\in C_{\infty}(\R)$. Then
\begin{equation}\lb{32}
\lim_{R\rightarrow \infty}\int_{\R}\frac{\xi\big(\lambda; H_{(0,R),\alpha,\beta}, 
H_{(0,R),\alpha,\beta}^{(0)}\big) d\lambda}{\lambda^2 + 1}g(\lambda) 
= \int_{\R}\frac{\xi\big(\lambda; H_{(0,\infty),\alpha}, 
H_{(0,\infty),\alpha}^{(0)}\big) d\lambda}{\lambda^2 + 1}g(\lambda).
\end{equation}
\end{theorem}
\begin{proof}
Abbreviating temporarily 
\begin{align}
\xi\big(\cdot; H_{(0,R);\alpha,\beta}, H_{(0,R);\alpha,\beta}^{(0)}\big) &=
\xi_{R,+}(\cdot) - \xi_{R,-}(\cdot),    \lb{5.46} \\
\xi\big(\cdot; H_{(0,\infty);\alpha}, H_{(0,\infty);\alpha}^{(0)}\big) &= 
\xi_{+}(\cdot) - \xi_{-}(\cdot), 
\end{align}
for ease of notation, where 
\begin{align}
\xi_{R,+}(\cdot) &= \xi\big(\cdot; H_{(0,R);\alpha,\beta}^{(0)} +_q V_{R,+}, 
H_{(0,R);\alpha,\beta}^{(0)}\big) \geq 0,    \\
\xi_{R,-}(\cdot) &= - \xi\big(\cdot; H_{(0,R);\alpha,\beta}^{(0)} +_q V_{R,+} -_q V_{R,-}, H_{(0,R);\alpha,\beta}^{(0)} +_q V_{R,+}\big) \geq 0,   \\
\xi_{+}(\cdot) &= \xi\big(\cdot; H_{(0,\infty);\alpha}^{(0)} +_q V_+, H_{(0,\infty);\alpha}^{(0)}\big) \geq 0,    \\
\xi_{-}(\cdot) &= - \xi\big(\cdot; H_{(0,\infty);\alpha}^{(0)} 
+_q V_+ -_q V_-, H_{(0,\infty);\alpha}^{(0)} +_q V_+\big) \geq 0, 
\lb{5.51}
\end{align}
the basic idea is to verify that 
\begin{equation} \lb{5-18}
 \lim_{R\rightarrow \infty}\int_{\R}\frac{\xi_{R,\pm}(\lambda) d\lambda}{\lambda^2 + 1} \, P((\lambda+i)^{-1},(\lambda-i)^{-1})   
= \int_{\R}\frac{\xi_{\pm}(\lambda) d\lambda}{\lambda^2 + 1} \, 
P((\lambda+i)^{-1},(\lambda-i)^{-1})
\end{equation}
for all polynomials $P(\cdot,\cdot)$ in two variables, and then rely on the  
Stone--Weierstrass approximation in Lemma \ref{l5.2} to get 
\begin{equation} 
 \lim_{R\rightarrow \infty}\int_{\R}\frac{\xi_{R,\pm}(\lambda) d\lambda}{\lambda^2 + 1} \, g(\lambda)   
= \int_{\R}\frac{\xi_{\pm}(\lambda) d\lambda}{\lambda^2 + 1} \, g(\lambda), 
\quad g\in C_{\infty}(\R),    \lb{5.53}
\end{equation}
and hence \eqref{32}. (For another 
application of this technique in spectral theory, see, e.g., 
\cite[Theorem\ VIII.20]{RS80}.) To prove \eqref{5-18}, it suffices to verify that 
\begin{align}
\begin{split} \lb{5-19}
\lim_{R\rightarrow \infty}\int_{\R}\frac{\xi_{R,\pm}(\lambda) d\lambda}{\lambda^2 + 1}\frac{1}{(\lambda+i)^m(\lambda-i)^n} =\int_{\R}\frac{\xi_{\pm}(\lambda) 
d\lambda}{\lambda^2 + 1}\frac{1}{(\lambda+i)^m(\lambda-i)^n},& \\ 
 m,n\in \N\cup \{0\},&
\end{split}
\end{align}
which, in turn, follows once one proves  
\begin{equation} \lb{5-20}
\lim_{R\rightarrow \infty}\int_{\R}\frac{\xi_{R,\pm}(\lambda) d\lambda}
{\lambda^2 + 1}\frac{1}{(\lambda\pm i)^n} 
= \int_{\R}\frac{\xi_{\pm}(\lambda) d\lambda}
{\lambda^2 + 1}\frac{1}{(\lambda\pm i)^n},  \quad n\in \N\cup \{0\},
\end{equation}
since
\begin{equation}
\frac{1}{(\lambda+i)^m(\lambda-i)^n}=\sum_{j=1}^m\frac{c_j}{(\lambda+i)^j}+\sum_{j=1}^n\frac{\widehat{c}_j}{(\lambda-i)^j}
\end{equation}
for appropriate constants $c_j$ and $\widehat{c}_j$.
Choosing $z=\pm i$ and $a=\mp i$ in \eqref{31} yields \eqref{5-20}, and therefore \eqref{5-18} for all polynomials $P$. At this point, \eqref{32} follows from Lemma \ref{l5.2} once one shows the existence of an $M>0$ for which 
\begin{equation} 
\bigg|\int_{\R} \f{\xi_{R,\pm}(\lambda) d\lambda}{\lambda^2 + 1}\bigg| \leq M 
\end{equation} 
for $R$ sufficiently large.  Taking \eqref{5-20} with $n=0$, yields the convergence 
\begin{equation} 
\lim_{R\to\infty} \int_{\R} \f{\xi_{R,\pm}(\lambda) d\lambda}{\lambda^2 + 1} 
= \int_{\R} \f{\xi_{\pm}(\lambda) d\lambda}{\lambda^2 + 1}.     \lb{5.58}
\end{equation}
As a result, $\int_{\R}\xi_{R,\pm}(\lambda) d\lambda \, (\lambda^2 + 1)^{-1}$ is uniformly bounded with respect to $R$.
\end{proof}

An immediate consequence of Theorem \ref{t5.3} is the following vague convergence result: 

\begin{corollary} \lb{c5.4}
Assume \eqref{3.8} and let $g\in C_0(\R)$. Then
\begin{equation}
\lim_{R\rightarrow \infty}\int_{\R}\xi\big(\lambda; H_{(0,R),\alpha,\beta}, 
H_{(0,R),\alpha,\beta}^{(0)}\big) d\lambda \, g(\lambda)
= \int_{\R}\xi\big(\lambda; H_{(0,\infty),\alpha}, 
H_{(0,\infty),\alpha}^{(0)}\big) d\lambda \, g(\lambda).
\end{equation}
\end{corollary}

\begin{remark} \lb{r5.5}
Since by \eqref{4.67} and \eqref{4.69}, $H_{(0,R),\alpha,\beta}$ and 
$H_{(0,\infty),\alpha}$, are 
uniformly bounded from below and hence we use our convention \eqref{2.9} of 
vanishing spectral shift functions in a fixed (i.e., $R$-independent) neighborhood 
of $-\infty$, no restrictions on $g$ near $- \infty$ need to be imposed (apart from measurability of $g$, of course).  
\end{remark}

\begin{remark} \lb{r5.6}
In the finite interval context, the corresponding spectral shift function 
$\xi\big(\cdot; H_{(0,R),\alpha,\beta}, H_{(0,R),\alpha,\beta}^{(0)}\big)$ is actually the difference of two eigenvalue counting functions (c.f.\ also 
\cite{HM10}, \cite{HM10a}, \cite{Ki87}), that is, 
\begin{equation}
\xi\big(\lambda; H_{(0,R),\alpha,\beta}, 
H_{(0,R),\alpha,\beta}^{(0)}\big)=N(\lambda;H_{(0,R),\alpha,\beta}^{(0)})-N(\lambda;H_{(0,R),\alpha,\beta} ), \quad \lambda \in\bbR,
\end{equation}
where we have written $N(\lambda;S)$ to denote the number of eigenvalues (counted according to multiplicity) of the self-adjoint operator $S$ that are less than or equal to $\lambda$.  Consequently, 
$\xi\big(\cdot; H_{(0,R),\alpha,\beta}, H_{(0,R),\alpha,\beta}^{(0)}\big)$ is an integer-valued function.  Thus, one cannot expect that the integer-valued functions 
$\xi\big(\cdot; H_{(0,R),\alpha,\beta}, H_{(0,R),\alpha,\beta}^{(0)}\big)$ converge pointwise to the spectral shift function 
$\xi\big(\cdot; H_{(0,\infty),\alpha}, H_{(0,\infty),\alpha}^{(0)}\big)$, since the latter may, under the assumption $V \in L^1((0,\infty); dx)$ (resp., \eqref{3.100}), $V$ real-valued, be taken continuous (and non-constant) on the half-line $(0,\infty)$, and the pointwise limit of integer-valued functions (if it exists) is necessarily integer-valued.  This line of reasoning also rules out convergence of 
$\xi\big(\lambda; H_{(0,R),\alpha,\beta}, H_{(0,R),\alpha,\beta}^{(0)}\big)(1+\lambda^2)^{-1}$ to $\xi\big(\lambda; H_{(0,\infty),\alpha}, 
H_{(0,\infty),\alpha}^{(0)}\big)(1+\lambda^2)^{-1}$ in the topology of 
$L^1((0,\infty);d\lambda)$ as $R\to\infty$, since the latter would imply the existence of a pointwise a.e.\ convergent subsequence, 
$\xi\big(\cdot; H_{(0,R_k),\alpha,\beta}, H_{(0,R_k),\alpha,\beta}^{(0)}\big)
\underset{k\to\infty}{\longrightarrow} 
\xi\big(\cdot; H_{(0,\infty),\alpha}, H_{(0,\infty),\alpha}^{(0)}\big)$.

A weaker notion of convergence is that of \textit{convergence in measure}.  However, since convergence in measure implies the existence of a pointwise a.e.\ convergent subsequence (cf., e.g., \cite[Corollary\ 2.32]{Fo99}), again one cannot expect that 
$\xi\big(\cdot; H_{(0,R),\alpha,\beta}, H_{(0,R),\alpha,\beta}^{(0)}\big)
\underset{R\to\infty}{\longrightarrow} 
\xi\big(\cdot; H_{(0,\infty),\alpha}, H_{(0,\infty),\alpha}^{(0)}\big)$ in measure. 

This argument (and its analog in higher dimensions) shows that the hope of a.e.\ convergence of spectral shift functions expressed in \cite[Remark\ 1.5\,(ii)]{HM10} for appropriate subsequences of $\{R_k\}_{k\in\bbN}$ cannot possibly materialize.  
\end{remark}

However, while Theorem \ref{t5.3} already improves on known results in the literature, it is not the best possible result in this direction and in the remainder of this section we will show how to improve on it. 

First, however, we note the following consequence of Theorem \ref{t5.4}:

\begin{lemma}\lb{L8}
Assume \eqref{3.8} and $E, E_1, E_2 \in\bbR$, $E_1 < E_2$. Then 
\begin{equation}
\xi\big(\cdot; H_{(0,R),\alpha,\beta}, H_{(0,R),\alpha,\beta}^{(0)}\big), \, 
\xi\big(\cdot; H_{(0,\infty),\alpha}, H_{(0,\infty),\alpha}^{(0)}\big) \in 
L^1((-\infty,E); d\lambda)    \lb{5.61}
\end{equation} 
and 
\begin{equation} \lb{175}
\lim_{R\rightarrow \infty}\int_{-\infty}^E 
\xi\big(\lambda; H_{(0,R),\alpha,\beta}, H_{(0,R),\alpha,\beta}^{(0)}\big) 
d\lambda = \int_{-\infty}^E 
\xi\big(\lambda; H_{(0,\infty),\alpha}, H_{(0,\infty),\alpha}^{(0)}\big) d\lambda,  
\end{equation}
and hence,
\begin{equation} \lb{178}
\lim_{R\rightarrow \infty}\int_{E_1}^{E_2} 
\xi\big(\lambda; H_{(0,R),\alpha,\beta}, H_{(0,R),\alpha,\beta}^{(0)}\big) 
d\lambda = \int_{E_1}^{E_2} 
\xi\big(\lambda; H_{(0,\infty),\alpha}, H_{(0,\infty),\alpha}^{(0)}\big) 
d\lambda.
\end{equation}
\end{lemma}
\begin{proof}
Since $\xi\big(\cdot; H_{(0,R),\alpha,\beta}, H_{(0,R),\alpha,\beta}^{(0)}\big)$,  
$\xi\big(\cdot; H_{(0,\infty),\alpha}, H_{(0,\infty),\alpha}^{(0)}\big)$ both vanish in a neighborhood of $-\infty$, the integrability assertion \eqref{5.61} holds.

Next we decompose 
$\xi\big(\cdot; H_{(0,R),\alpha,\beta}, H_{(0,R),\alpha,\beta}^{(0)}\big)$ 
and 
$\xi\big(\cdot; H_{(0,\infty),\alpha}, H_{(0,\infty),\alpha}^{(0)}\big)$ as in 
\eqref{5.46}--\eqref{5.51}. 

Since by \eqref{4.67} and \eqref{4.69}, $H_{(0,R),\alpha,\beta}$ and 
$H_{(0,\infty),\alpha}$, and similarly, these operators with $V$ replaced by 
$V_{R,\pm}$, $V_{\pm}$ are all uniformly bounded from below with respect to 
$R>0$, \eqref{177} applies with $\mu_n$ replaced by 
$\xi_{R,\pm}(\lambda) d\lambda$ and 
$\xi_{\pm}(\lambda) d\lambda$. Moreover, since these measures are all 
absolutely continuous and hence have no atoms, \eqref{175} holds for all 
$E\in\bbR$. The result \eqref{178} now follows from \eqref{175} by taking differences.
\end{proof}

Given the decomposition \eqref{5.46}--\eqref{5.51}, and introducing the measures
\begin{align}
\begin{split}
\eta_{\xi_{R,\pm}} (A)=\int_A\frac{\xi_{R,\pm} 
(\lambda) d\lambda}{\lambda^2 + 1}, \; R>0, \quad 
\eta_{\xi_{\pm}} (A)=\int_A\frac{\xi_{\pm} 
(\lambda) d\lambda}{\lambda^2 + 1},&    \lb{5.64} \\ 
\text{ $A \subseteq \bbR$ Lebesgue measurable,}&
\end{split}
\end{align}
we are now ready for the principal result of this paper.

\begin{theorem}\lb{T11}
Assume \eqref{3.8}. Then
\begin{align}\lb{179}
& \lim_{R\rightarrow \infty}\int_{\R} 
\frac{\xi\big(\lambda; H_{(0,R),\alpha,\beta}, H_{(0,R),\alpha,\beta}^{(0)}\big) 
d\lambda}{\lambda^2 + 1} \, f(\lambda) = 
\int_{\R}
\frac{\xi\big(\lambda; H_{(0,\infty),\alpha}, H_{(0,\infty),\alpha}^{(0)}\big) 
d\lambda}{\lambda^2 + 1} \, f(\lambda),    \no \\
& \hspace*{9.3cm}  f\in C_b(\R).
\end{align} 
Moreover, \eqref{179} holds for every bounded Borel measurable function $f$ that is $\eta_{\xi_+}$ and $\eta_{\xi_-}$-almost everywhere continuous 
on $\R$.  In particular,
\begin{equation}\lb{180}
\lim_{R\rightarrow \infty} \int_S 
\frac{\xi\big(\lambda; H_{(0,R),\alpha,\beta}, H_{(0,R),\alpha,\beta}^{(0)}\big) 
d\lambda}{\lambda^2 + 1} = \int_S 
\frac{\xi\big(\lambda; H_{(0,\infty),\alpha}, H_{(0,\infty),\alpha}^{(0)}\big) 
d\lambda}{\lambda^2 + 1}
\end{equation}
for every set $S$ whose boundary has $\eta_{\xi_+}$ and 
$\eta_{\xi_-}$-measure equal to zero.
\end{theorem}
\begin{proof}
Again we decompose 
$\xi\big(\cdot; H_{(0,R),\alpha,\beta}, H_{(0,R),\alpha,\beta}^{(0)}\big)$ 
and 
$\xi\big(\cdot; H_{(0,\infty),\alpha}, H_{(0,\infty),\alpha}^{(0)}\big)$ as in 
\eqref{5.46}--\eqref{5.51}. By \eqref{5.53} and Corollary \ref{c5.4}, the 
measure $\eta_{\xi_{R,\pm}}$ vaguely converges to the measure 
$\eta_{\xi_{\pm}}$ as $R\to\infty$, respectively. Moreover, by \eqref{5-19},
\begin{equation}
\lim_{R\to\infty} \eta_{\xi_{R,\pm}} (\R) = \eta_{\xi_{\pm}} (\R). 
\end{equation}
Thus, by Theorem \ref{Bauer30.8}, one concludes weak convergence of the sequence of measures $\eta_{\xi_{R,\pm}}$ to the measure 
$\eta_{\xi_{\pm}}$ as $R\to\infty$.  That \eqref{179} holds for every bounded Borel measurable function that is $\eta_{\xi_{\pm}}$-almost everywhere continuous on $\R$ now follows directly from Theorem \ref{Bauer30.12}. 
Finally, convergence in \eqref{180} is also a direct consequence of 
Theorem \ref{Bauer30.12}.
\end{proof}

As immediate consequences of Theorem \ref{T11}, we have the following two results:

\begin{corollary}
Assume \eqref{3.8}. Then convergence in \eqref{179} holds for any bounded Borel measurable function that is Lebesgue-almost everywhere continuous.  In particular, 
\eqref{180} holds for any set $S$ that is boundaryless with respect to Lebesgue 
measure $($i.e., any set $S$ for which the boundary of $S$ has Lebesgue measure 
equal to zero$)$.
\end{corollary}
\begin{proof}
Noting that $\eta_{\xi_{\pm}}$ are absolutely continuous with respect to Lebesgue measure, the statements follow directly from Theorem \ref{T11}.
\end{proof}

\begin{corollary}\lb{C4}
Assume \eqref{3.8}. If $g$ is a bounded Borel measurable function that is compactly supported and Lebesgue almost everywhere continuous on $\R$, then
\begin{equation}
\lim_{R\rightarrow \infty}\int_{\R} 
\xi\big(\lambda; H_{(0,R),\alpha,\beta}, H_{(0,R),\alpha,\beta}^{(0)}\big) 
d\lambda \, g(\lambda) 
=\int_{\R}\xi\big(\lambda; H_{(0,\infty),\alpha}, H_{(0,\infty),\alpha}^{(0)}\big) 
d\lambda \, g(\lambda).
\end{equation}
\end{corollary}
\begin{proof}
If $g$ satisfies the hypotheses of Corollary \ref{C4}, then choosing 
$f(\lambda):=(\lambda^2 + 1)g(\lambda)$ in \eqref{179} yields the result, noting that 
$f$ is a bounded ($g$ has compact support) Borel measurable function and is continuous Lebesgue-almost everywhere (and thus $\eta_{\xi_{\pm}}$-almost everywhere) on $\R$.
\end{proof}

We should perhaps point out that the regularized trace formula discussed in \cite{Fa57}, which also involves an infinite volume limit $R\to\infty$, differs from the result \eqref{179}.

For completeness we now also show the weak convergence result in Theorem \ref{T11} can be obtained via Helly's Selection Theorem combined with Theorem \ref{t5.4} (without using 
Theorems \ref{Bauer30.8}, \ref{Bauer30.12}, and  \ref{T11}). 

\begin{theorem}[Convergence via Helly's Selection Theorem]\lb{T12}
Assume \eqref{3.8}, then 
\begin{equation}\lb{Helly0}
\lim_{R\rightarrow \infty}\int_{\R}\frac{\xi_R(\lambda) d\lambda}{\lambda^2+1} \, f(\lambda) 
= \int_{\R}\frac{\xi(\lambda) d\lambda}{\lambda^2+1} \, f(\lambda), \quad 
f\in C_b(\R). 
\end{equation} 
\end{theorem}
\begin{proof}
Once more we use the decompositions \eqref{5.46}--\eqref{5.51} and 
introduce the associated non-decreasing continuous functions $\sigma_{R,\pm}$, $R>0$, 
and $\sigma_{\pm}$ on $\R\cup\{\infty\}$ defined by
\begin{align}
\begin{split}
\sigma_{R,\pm} (\lambda) &= \int_{(-\infty,\lambda)} \f{\xi_{R,\pm}(\lambda') d\lambda'}{(\lambda^{\prime})^2 + 1},  \; R>0,   \\
\sigma_{\pm}(\lambda) &= \int_{(-\infty,\lambda)} \f{\xi_{\pm}(\lambda') 
d\lambda'}{(\lambda^{\prime})^2 + 1}, \quad \lambda \in \R\cup\{\infty\}. 
\end{split}
\end{align}
By \eqref{5.58} one has 
\begin{equation}\lb{Helly1a}
\lim_{R\rightarrow \infty}\sigma_{R,\pm} (\infty) = \sigma_{\pm} (\infty).
\end{equation}
Next, we aim at proving
\begin{equation}\lb{Helly1}
\lim_{R\rightarrow \infty}\sigma_{R,\pm} (\lambda) = 
\sigma_{\pm} (\lambda),  \quad \lambda \in \R.
\end{equation}
Having already shown vague convergence of $\{\eta_{R,\pm}\}$ to $\eta_{\pm}$ (cf.\ 
\eqref{5.64} and Corollary \ref{c5.4}), we now verify
\begin{equation}\lb{Helly2}
\lim_{E\downarrow -\infty}\Big(\limsup_{R\rightarrow \infty}\big(\sigma_{R,\pm}(E)\big)\Big)=0,
\end{equation}
as \eqref{Helly1} then follows from Theorem \ref{t5.4}. One notes that \eqref{Helly2} 
follows if one can show that 
\begin{align} 
\begin{split} \lb{Helly3}
& \text{for all $\varepsilon>0$, there exists $E(\varepsilon) \in \R$ such that 
$\sigma_{R,\pm}(E) \leq \varepsilon$}  \\
& \quad \text{for all $E \leq E(\varepsilon)$ and all $R>0$.}
\end{split} 
\end{align}
Let $\varepsilon>0$ be given. By \eqref{Helly1a} there exists an interval 
$K_{\varepsilon}=[a(\varepsilon),b(\varepsilon)]$ such that
\begin{equation}
\eta_{R,\pm} (\R\backslash K_{\varepsilon}) = 
\int_{\R \backslash K_{\varepsilon}} \f{\xi_{R,\pm}(\lambda) d\lambda}{\lambda^2+1}
\leq \varepsilon, \quad R>0,
\end{equation} 
which is to say that
\begin{equation}\lb{Helly4}
\int_{-\infty}^{a(\varepsilon)}\frac{\xi_{R,\pm}(\lambda)}{1+\lambda^2}d\lambda 
+ \int_{b(\varepsilon)}^{\infty}\frac{\xi_{R,\pm}(\lambda)}{1+\lambda^2}d\lambda 
\leq \varepsilon, \quad R>0.
\end{equation}
Thus, one sees that \eqref{Helly3} is valid once one takes $E(\varepsilon)=a(\varepsilon)$.

Now, fix $f\in C_b(\R)$; we may assume without loss that 
$\|f\|_{\infty}=\sup_{\lambda\in R}|f(\lambda)|\neq 0$ (otherwise, $f(\lambda)=0$ for all 
$\lambda\in\bbR$, and the convergence \eqref{Helly0} is trivial). By \eqref{Helly1a}, 
for any $\varepsilon>0$, there exists an $A(\varepsilon)>0$ such that for all 
$a,b>A(\varepsilon)$, 
\begin{equation}
\int_{-\infty}^{a}\frac{\xi_{R,\pm}(\lambda) d\lambda}{\lambda^2+1} 
+ \int_{b}^{\infty}\frac{\xi_{R,\pm}(\lambda) d\lambda}{\lambda^2+1}
\leq \frac{\varepsilon}{\|f\|_{\infty}},    \quad R>0.
\end{equation}
Thus, for all $a,b>A(\varepsilon)$ one has
\begin{align}
\int_{-\infty}^{a}\frac{\xi_{R,\pm}(\lambda) d\lambda}{\lambda^2+1} \, |f(\lambda)|
&\leq \|f\|_{\infty}\int_{-\infty}^{a}\frac{\xi_{R,\pm}(\lambda) d\lambda}{\lambda^2+1} 
\leq \varepsilon,   \quad R>0, \\
\int_{b}^{\infty}\frac{\xi_{R,\pm}(\lambda) d\lambda}{\lambda^2+1} \, |f(\lambda)|
&\leq \|f\|_{\infty}\int_{-\infty}^{a}\frac{\xi_{R,\pm}(\lambda) d\lambda}{\lambda^2+1}
\leq \varepsilon,   \quad R>0.
\end{align}
At this point one has verified the asumptions (2.7) in the generalized version of 
Helly's second selection theorem in the form of \cite[Sect.\ XIV.2]{LS75}, implying \eqref{Helly0}.
\end{proof}

\begin{remark} \lb{r4.16}
We conclude with a brief remark on an alternative approach to continuity of 
spectral shift functions. In cases where a fixed Hilbert space is naturally given, 
an abstract convergence result of spectral shift functions in the 
$L^1\big(\bbR; (\lambda^2 + 1)^{-1}d\lambda\big)$-norm, based on trace norm convergence of resolvents,  
has been given in \cite[Lemma\ 8.7.5]{Ya92} with a refinement in the case of semibounded self-adjoint operators in \cite[Lemma\ 8.3]{KT09}. This is then applied 
to derive relative oscillation theory results in essential spectral gaps between 
different Sturm--Liouville operators in \cite{KT08} and \cite{KT09}. This approach differs from the one developed in this paper as in our case one cannot expect subsequences of spectral functions to converge pointwise a.e.\ as discussed at the end of Remark \ref{r5.6}.
\end{remark}

\appendix
\section{Basic Facts on Spectral Shift Functions}
\lb{sA}
\renewcommand{\theequation}{A.\arabic{equation}}
\renewcommand{\thetheorem}{A.\arabic{theorem}}
\setcounter{theorem}{0} \setcounter{equation}{0}

In this appendix we succinctly summarize properties of the spectral shift function as needed in the bulk of this paper (for details on this material we refer to \cite{BY93}, 
\cite[Ch.\ 8]{Ya92}, \cite{Ya07}, \cite[Sect.\ 0.9, Chs.\ 4, 5, 9]{Ya10}).

We start with the following basic assumptions:

\begin{hypothesis} \lb{h2.1}
Suppose $A$ and $B$ are self-adjoint operators in $\cH$ with $A$ bounded from below. \\
$(i)$ Assume that $B$ can be written as the form sum of $A$ and 
a self-adjoint operator $W$ in $\cH$
\begin{equation}
B = A +_q W,     \lb{2.1}
\end{equation}
where $W$ can be factorized into
\begin{equation}
W = W_1 W_2,    \lb{2.2}
\end{equation}
such that 
\begin{equation}
\dom(W_j) \supseteq \dom\big(|A|^{1/2}\big), \quad j =1,2,   \lb{2.2a}
\end{equation}
and 
\begin{equation}
\ol{W_2 (A - z I_{\cH})^{-1} W_1} \in \cB_1(\cH), \quad z \in \rho(A).   \lb{2.3}
\end{equation}
$(ii)$ In addition, we suppose that for some $($and hence for all\,$)$ 
$z \in \rho(B) \cap \rho(A)$,
\begin{equation}
\big[(B - z I_{\cH})^{-1} - (A - z I_{\cH})^{-1}\big] \in \cB_1(\cH).  \lb{2.4}
\end{equation}
\end{hypothesis}

Given Hypothesis \ref{h2.1}\,$(i)$, one observes that 
\begin{equation}
\dom\big(|B|^{1/2}\big) = \dom\big(|A|^{1/2}\big),    \lb{2.5} 
\end{equation}
and that the resolvent of $B$ can be written as (cf., e.g., the detailed discussion in 
\cite{GLMZ05} and the references therein)
\begin{align}
&(B - z I_{\cH})^{-1} = (A - z I_{\cH})^{-1}    \no \\ 
& \quad - \ol{(A - z I_{\cH})^{-1} W_1} 
\big[I_{\cH} + \ol{W_2 (A - z I_{\cH})^{-1} W_1}\big]^{-1} W_2 (A - z I_{\cH})^{-1}, \lb{2.6}
\\
& \hspace*{7.65cm} z \in \rho(B) \cap \rho(A).     \no
\end{align}
In particular, $B$ is bounded from below in $\cH$.

Moreover, assuming the full Hypothesis \ref{h2.1} one infers that (cf.\ \cite{GZ10}) 
\begin{align}
& {\tr}_{\cH}\big((B - z I_{\cH})^{-1} - (A - z I_{\cH})^{-1}\big)   \no \\
& \quad 
= - \f{d}{dz} \ln\Big({\det}_{\cH}\Big(\ol{(B - z I_{\cH})^{1/2}(A - z I_{\cH})^{-1}
	(B - z I_{\cH})^{1/2}}\Big)\Big)    \no \\
&\quad =  - \f{d}{dz} \ln\Big({\det}_{\cH}\Big(I_{\cH} +\ol{W_2 (A - z I_{\cH})^{-1} 
W_1}\Big)\Big), \quad z \in \rho(B) \cap \rho(A).     \lb{2.7}
\end{align}
Here ${\det}_{\cH}(\cdot)$ denotes the Fredholm determinant (cf.\ \cite[Ch.\ IV]{GK69},
\cite{Si77}, \cite[Ch.\ 3]{Si05}). 

In addition, Hypothesis \ref{h2.1} guarantees the existence of the real-valued spectral shift function $\xi(\cdot; B,A)$ satisfying
\begin{align}
\begin{split} 
{\tr}_{\cH}\big((B - z I_{\cH})^{-1} - (A - z I_{\cH})^{-1}\big) 
= - \int_{\bbR} \f{\xi(\lambda; B,A) \, d\lambda}{(\lambda - z)^2},&   \lb{2.8} \\
z \in\bbC\backslash [\inf(\sigma(B), \sigma(A)),\infty),& 
\end{split}  
\end{align}
and 
\begin{align}
& \xi(\lambda; B,A) = 0, \quad \lambda < \inf(\sigma(B), \sigma(A)),    \lb{2.9} \\
& \xi(\cdot; B,A) \in L^1\big(\bbR; (\lambda^2 + 1)^{-1} d\lambda\big). 
\lb{2.10}
\end{align}
Moreover, for a large class of functions $f$ (e.g., any $f$ s.t.\ 
$\hatt f (\cdot) \in L^1(\bbR; (|p| + 1) dp)$) one infers that 
$[f(B) - f(A)] \in \cB_1(\cH)$ and 
\begin{equation}
{\tr}_{\cH} (f(B) - f(A)) = \int_{\bbR} f'(\lambda) \, \xi(\lambda;B,A) \, d\lambda. 
\lb{2.11}
\end{equation}
This applies, in particular, to powers of the resolvent, where 
$f(\cdot) = (\cdot -z)^{-n}$, $n \in \bbN$, and we refer to \cite[Ch.\ 8]{Ya92} 
for more details. 

Throughout this manuscript we assume that the normalization \eqref{2.9} is applied. 

We also note the following monotonicity result: If 
\begin{align}
\begin{split}
& B \geq A \, \text{ (resp., $B \leq A$) in the sense of quadratic forms, then } \\
& \quad \xi(\lambda; B,A) \geq 0 \, \text{ (resp., $\xi(\lambda; B,A) \leq 0$).}
\lb{2.11a}
\end{split} 
\end{align}
Here, $B \geq A$ is meant in the sense of quadratic forms, that is, 
\begin{align}
& \dom \big(|A|^{1/2}\big) \supseteq \dom \big(|B|^{1/2}\big) \, \text{ and } 
\lb{2.11b} \\ 
& \quad \big(|B|^{1/2} f, \sgn(B) |B|^{1/2} f\big)_{\cH} \geq 
\big(|A|^{1/2} f, \sgn(A) |A|^{1/2} f\big)_{\cH}, \quad 
f \in \dom \big(|B|^{1/2}\big).    \no 
\end{align}

Next, suppose that the self-adjoint operator $C$ in $\cH$ can be written as the form sum of $B$ and a self-adjoint operator $Q$ in $\cH$, $C = B \dot + Q$, where $Q$ can be factored as $Q=Q_1 Q_2$, with $Q, Q_1, Q_2$ satisfying the assumptions of $W, W_1, W_2$ in Hypotheses \ref{h2.1}. Then the formula
\begin{equation}
\xi(\lambda;C,A) = \xi(\lambda;C,B) + \xi(\lambda;B,A) \, 
\text{ for a.e.\ $\lambda \in \bbR$,}     \lb{2.11c}
\end{equation}
holds. 

Finally, we mention the connection between $\xi(\cdot; B,A)$ and the Fredholm determinant in \eqref{2.7},
\begin{equation}
\xi(\lambda; B,A) = \pi^{-1} \lim_{\varepsilon\downarrow 0} \Im\Big(\ln\Big(I_{\cH}
+ \ol{W_2 (A - (\lambda + i \varepsilon) I_{\cH})^{-1} W_1}\Big)\Big) \, 
\text{ for a.e.\ $\lambda \in \bbR$},    \lb{2.12}
\end{equation}
choosing the branch of $\ln({\det}_{\cH}(\cdot))$ on $\bbC_+$ such that 
\begin{equation}
\lim_{\Im(z) \uparrow +\infty} \ln\Big({\det}_{\cH}\Big(I_{\cH} 
+ \ol{W_2 (A - z I_{\cH})^{-1} W_1}\Big)\Big) = 0.    \lb{2.13} 
\end{equation}

\section{A Spectral Shift Function Decomposition}
\lb{sB}
\renewcommand{\theequation}{B.\arabic{equation}}
\renewcommand{\thetheorem}{B.\arabic{theorem}}
\setcounter{theorem}{0} \setcounter{equation}{0}

We conclude this paper with an interesting formula that provides a comparison of the 
spectral shift functions for the interval $[0,R_1]$ and $[0,R_2]$, where 
$0<R_1<R_2$, under the assumption of Dirichlet boundary conditions. 

To simplify our notation a bit, the Dirichlet boundary conditions at $0$, $R_1$, and 
$R_2$ will be indicated by the subscript ``D'' and the symbol $H^{(0)}$ will be 
replaced by $-\Delta$.

We begin with the following result. 

\begin{lemma}\lb{lB.1}
Let $0<R_1<R_2$ and assume that $V \in L^1((0,R_2);dx)$ is real-valued. Then 
$\xi(\cdot;-\Delta_{(0,R_1)\cup(R_1,R_2),D}+V_{R_2}, -\Delta_{(0,R_1)\cup(R_1,R_2),D})$ 
can be decomposed into a sum of two spectral shift functions as follows,
\begin{align}  
& \xi(\lambda; -\Delta_{(0,R_1)\cup(R_1,R_2),D}+V_{R_2}, 
-\Delta_{(0,R_1)\cup(R_1,R_2),D})      \no \\ 
& \quad = \xi(\lambda; -\Delta_{(0,R_1),D} + V_R, -\Delta_{(0,R_1),D})   \lb{164} \\
& \qquad +\xi(\lambda; -\Delta_{(R_1, R_2),D} + V|_{(R_1, R_2)},-\Delta_{(R_1, R_2),D})  
\,  \text{ for a.e.\ $\lambda \in\bbR$.}   \no 
\end{align}
\end{lemma}
\begin{proof}
The proof relies essentially on the classical fact that 
\begin{equation} 
-\Delta_{(0,R_1)\cup(R_1,R_2),D}=-\Delta_{(0,R_1),D}\oplus -\Delta_{(R_1, R_2),D}  
\end{equation}
(cf., e.g., \cite[Proposition\ X.III.13.3]{RS78}) and of course
\begin{equation}
V_{R_2} = V_{R_1} \oplus V|_{(R_1,R_2)},
\end{equation}  
both with respect to the decomposition 
\begin{equation}
L^2((0,R_2); dx) = L^2((0,R_1); dx) \oplus L^2((R_1,R_2); dx). 
\end{equation}

One computes
\begin{align}
&{\det}_{L^2((0,R_2);dx)}\big(I_{(0,R_2)}+u_{R_2} (-\Delta_{(0,R_1)\cup(R_1,R_2),D} 
- z)^{-1}v_{R_2}\big)   \\
&\quad ={\det}_{L^2((0,R_2);dx)}\big(I_{(0,R_1)}\oplus I_{(R_1,R_2)}+ u_{R_1} 
\oplus u|_{(R_1,R_2)}\big) \no \\
&\qquad \times (-\Delta_{(0,R_1),D}\oplus -\Delta_{(R_1, R_2),D}-z)^{-1}(v_{R_1}
\oplus v|_{(R_1,R_2)})\big)\no \\
&\quad ={\det}_{L^2((0,R_1);dx)}\big(I_{(0,R_1)} 
+ u_{R_1}(-\Delta_{(0,R_1),D}-z)^{-1}v_{R_1}\big)\no \\
&\qquad \times {\det}_{L^2((R_1,R_2);dx)}\big(I+u|_{(R_1,R_2)}
(-\Delta_{(R_1, R_2),D}-z)^{-1}v|_{(R_1,R_2)}\big).  \lb{166}
\end{align}
Therefore, replacing $z$ by $\lambda + i\varepsilon $, taking arguments, and finally the 
limit $\varepsilon \downarrow 0$ proves \eqref{164}.
\end{proof}

Krein's resolvent formula then shows that the difference of the resolvents of 
$-\Delta_{(0,R_2),D}$ and $-\Delta_{(0,R_1)\cup(R_1,R_2),D}$ is a rank one operator.  Explicitly, in terms of integral kernels, one has
\begin{align}
\begin{split} 
& G_{(0,R_1)\cup(R_1,R_2),D}^{(0)}(z,x,x^{\p}) = G_{(0,R_2),D}^{(0)}(z,x,x^{\p}) \\
& \quad -\frac{G_{(0,R_2),D}^{(0)}(z,x,R)G_{(0,R_2),D}^{(0)}(z,R,x^{\p})}
{G_{(0,R_2),D}^{(0)}(z,R,R)},
\quad z\in \C \backslash  \R, \lb{167} 
\end{split} 
\end{align}
where $G_{\Omega,D}^{(0)}(z,x,x^{\p})$ denotes the integral kernel of 
$(-\Delta_{\Omega,D} - z I_{\Omega})^{-1}$.  Therefore, we write
\begin{equation}\lb{168}
(-\Delta_{(0,R_2),D}-z)^{-1}=( -\Delta_{(0,R_1)\cup(R_1,R_2),D}-z)^{-1}+T(z), \quad z\in \C \backslash  \R,
\end{equation}
where $T(z)$ is the rank-one operator generated by the integral kernel given by the second term on the right-hand side of \eqref{167}.  Thus, one computes
\begin{align}
&{\det}_{L^2((0,R_2);dx)}\big(I_{(0,R_2)}+u_{R_2}(-\Delta_{(0,R_2),D}-z)^{-1}v_{R_2}\big)\no \\
& \quad ={\det}_{L^2((0,R_2);dx)}\big(I_{(0,R_2)}+u_{R_2}(-\Delta_{(0,R_1)\cup(R_1,R_2),D}-z)^{-1}v_{R_2}+u_{R_2}T(z)v_{R_2}\big)\no \\
& \quad ={\det}_{L^2((0,R_2);dx)}\big(\big[I_{(0,R_2)}+u_{R_2}(-\Delta_{(0,R_1)\cup(R_1,R_2),D}-z)^{-1}v_{R_2}\big]     \no \\
&\qquad \times \big\{I_{(0,R_2)}+\big[I_{(0,R_2)}+u_{R_2}(-\Delta_{(0,R_1)\cup(R_1,R_2),D}-z)^{-1}v_{R_2}\big]^{-1}u_{R_2}T(z)v_{R_2}\big\}\big)\no \\
& \quad = {\det}_{L^2((0,R_2);dx)}\big(I_{(0,R_2)}+u_{R_2}(-\Delta_{(0,R_1)\cup(R_1,R_2),D}-z)^{-1}v_{R_2}\big)  \no \\
&\qquad \times{\det}_{L^2((0,R_2);dx)}\big(I_{(0,R_2)}+\big[I_{(0,R_2)}
+u_{R_2}(-\Delta_{(0,R_1)\cup(R_1,R_2),D}-z)^{-1}v_{R_2}\big]^{-1}   \no \\ 
& \hspace*{3.5cm} \times u_{R_2}T(z)v_{R_2}\big), \quad  z\in \C \backslash  \R.   
  \lb{168b}
\end{align}
Since ${\det}_{\cH}(I+A)=1+\tr_{\cH} (A)$, whenever $A$ is a rank-one operator in $\cH$, 
the second determinant on the right-hand side of \eqref{168b} can be computed as follows:
\begin{align}
&{\det}_{L^2((0,R_2);dx)}\big(I_{(0,R_2)}
+ \big[I_{(0,R_2)}+u_{R_2}(-\Delta_{(0,R_1)\cup
(R_1,R_2),D}-z)^{-1}v_{R_2}\big]^{-1}   \no \\ 
& \qquad \times u_{R_2}T(z)v_{R_2}\big)\no \\
& \quad =1+\frac{1}{G_{(0,R_2),D}^{(0)}(z,R,R)}    \no \\
&\qquad \times \big(v_{R_2}G_{(0,R_2),D}^{(0)}(z,R,\cdot), \big[I_{(0,R_2)}+u_{R_2}
(-\Delta_{(0,R_1)\cup(R_1,R_2),D}-z)^{-1}v_{R_2}\big]^{-1}     \no \\ 
&\qquad \quad \;\, \times u_{R_2} G_{(0,R_2),D}^{(0)}(z,\cdot,R)\big)_{L^2((0,R_2);dx)}   \no \\ 
& \quad := \cS(z), \quad z\in \C \backslash  \R.    \lb{169}
\end{align}
Finally, replacing $z$ by $\lambda+i\varepsilon $, taking arguments, and taking the 
limit $\varepsilon  \downarrow 0$ throughout \eqref{168b} yields
\begin{align}\lb{170}
& \xi(\lambda; -\Delta_{(0,R_2),D} + V_R, -\Delta_{(0,R_2),D})    \no \\ 
& \quad =
\xi(\lambda; -\Delta_{(0,R_1)\cup(R_1,R_2),D}+V|_{(0,R_2)}, -\Delta_{(0,R_1)\cup(R_1,R_2),D})      \\
& \qquad +
\lim_{\varepsilon  \downarrow 0}\Im(\ln (\cS (\lambda + i  \varepsilon ))) 
\,  \text{ for a.e.\ $\lambda \in\bbR$.}     \no 
\end{align}
By Lemma \ref{lB.1}, the spectral shift function on the right-hand side of \eqref{170} is the sum of two spectral shift functions.  Thus, in summary, we have the following decomposition formula:

\begin{lemma}[Spectral Shift Function Comparison Formula] Let $0<R_1<R_2$ and assume that $V \in L^1((0,R_2);dx)$ is real-valued. Then 
\begin{align}
& \xi(\lambda; -\Delta_{(0,R_2),D} + V_R, -\Delta_{(0,R_2),D})   \no \\
& \quad = \xi(\lambda; -\Delta_{(0,R_1),D} + V_R, -\Delta_{(0,R_1),D})     \no \\
& \qquad 
+ \xi\big(\lambda;-\Delta_{(R_1, R_2),D}+V|_{(R_1,R_2)}, -\Delta_{(R_1, R_2),D})   \\
& \qquad + \lim_{\varepsilon  \downarrow 0}
\Im(\ln (\cS (\lambda + i \varepsilon ))) \, \text{ for a.e.\ $\lambda \in\bbR$.}     \no 
\end{align}
\end{lemma}

\medskip

\noindent {\bf Acknowledgments.}
We are indebted to Barry Simon for helpful discussions. 

\medskip


\end{document}